\documentclass[11pt]{article}
\usepackage[margin=.8in,left=.8in]{geometry}
\usepackage{amsmath}
\usepackage{amsfonts}
\usepackage{amssymb}
\usepackage{amsthm}
\usepackage{mathtools}
\usepackage{subcaption}
\usepackage{graphicx}
\usepackage{color}
\usepackage{xcolor}
\usepackage{xfrac}
\usepackage{stmaryrd}
\usepackage[utf8]{inputenc}
\usepackage{rotating}
\usepackage{hyperref}
\usepackage[all]{xy}

\usepackage{lmodern} 

\usepackage{tikz}
\usetikzlibrary{decorations.pathmorphing}
\usepackage{tikz-cd}
\usepackage{lipsum}
\usepackage{adjustbox}

\definecolor{plum}{rgb}{0.36078, 0.20784, 0.4}
\definecolor{chameleon}{rgb}{0.30588, 0.60392, 0.023529}
\definecolor{cornflower}{rgb}{0.12549, 0.29020, 0.52941}
\definecolor{scarlet}{rgb}{0.8, 0, 0}
\definecolor{brick}{rgb}{0.64314, 0, 0}
\definecolor{sunrise}{rgb}{0.80784, 0.36078, 0}
\definecolor{lightblue}{rgb}{0.15,0.35,0.75}
\definecolor{carolina}{RGB}{153, 186, 221}
\definecolor{darkblue}{rgb}{0.05,0.25,0.65}

\usepackage{multirow}

\usepackage{stmaryrd}    

\usepackage{enumerate} 

\usepackage{helvet}   

\usepackage{mathptmx}
\usepackage{amsmath}
\usepackage{graphicx}

\usepackage{floatflt}  
\usepackage{array}     
\newcolumntype{L}[1]{>{\raggedright\let\newline\\\arraybackslash\hspace{0pt}}m{#1}}
\newcolumntype{C}[1]{>{\centering\let\newline\\\arraybackslash\hspace{0pt}}m{#1}}
\newcolumntype{R}[1]{>{\raggedleft\let\newline\\\arraybackslash\hspace{0pt}}m{#1}}

\makeatletter
\newcommand{\raisemath}[1]{\mathpalette{\raisem@th{#1}}}
\newcommand{\raisem@th}[3]{\raisebox{#1}{$#2#3$}}
\makeatother

\usepackage{tabularx}   

\usepackage[new]{old-arrows}   

\newdir{> }{{}*!/10pt/@{>}}

\makeatletter
\newif\if@sup
\newtoks\@sups
\def\append@sup#1{\edef\act{\noexpand\@sups={\the\@sups #1}}\act}%
\def\reset@sup{\@supfalse\@sups={}}%
\def\mk@scripts#1#2{\if #2/ \if@sup ^{\the\@sups}\fi \else%
  \ifx #1_ \if@sup ^{\the\@sups}\reset@sup \fi {}_{#2}%
  \else \append@sup#2 \@suptrue \fi%
  \expandafter\mk@scripts\fi}
\def\tensor#1#2{\reset@sup#1\mk@scripts#2_/}
\def\multiscripts#1#2#3{\reset@sup{}\mk@scripts#1_/#2%
  \reset@sup\mk@scripts#3_/}
\makeatother

\makeatletter
\newbox\slashbox \setbox\slashbox=\hbox{$/$}
\def\itex@pslash#1{\setbox\@tempboxa=\hbox{$#1$}
  \@tempdima=0.5\wd\slashbox \advance\@tempdima 0.5\wd\@tempboxa
  \copy\slashbox \kern-\@tempdima \box\@tempboxa}
\def\slash{\protect\itex@pslash}
\makeatother

\def\clap#1{\hbox to 0pt{\hss#1\hss}}

\def\mathrlap{\mathpalette\mathrlapinternal}
\def\mathclap{\mathpalette\mathclapinternal}

\def\mathrlapinternal#1#2{\rlap{$\mathsurround=0pt#1{#2}$}}
\def\mathclapinternal#1#2{\clap{$\mathsurround=0pt#1{#2}$}}

\let\oldroot\root
\def\root#1#2{\oldroot #1 \of{#2}}
\renewcommand{\sqrt}[2][]{\oldroot #1 \of{#2}}

\DeclareSymbolFont{symbolsC}{U}{txsyc}{m}{n}
\SetSymbolFont{symbolsC}{bold}{U}{txsyc}{bx}{n}
\DeclareFontSubstitution{U}{txsyc}{m}{n}

\DeclareSymbolFont{stmry}{U}{stmry}{m}{n}
\SetSymbolFont{stmry}{bold}{U}{stmry}{b}{n}

\DeclareFontFamily{OMX}{MnSymbolE}{}
\DeclareSymbolFont{mnomx}{OMX}{MnSymbolE}{m}{n}
\SetSymbolFont{mnomx}{bold}{OMX}{MnSymbolE}{b}{n}
\DeclareFontShape{OMX}{MnSymbolE}{m}{n}{
    <-6>  MnSymbolE5
   <6-7>  MnSymbolE6
   <7-8>  MnSymbolE7
   <8-9>  MnSymbolE8
   <9-10> MnSymbolE9
  <10-12> MnSymbolE10
  <12->   MnSymbolE12}{}

\makeatletter
\def\Decl@Mn@Delim#1#2#3#4{%
  \if\relax\noexpand#1%
    \let#1\undefined
  \fi
  \DeclareMathDelimiter{#1}{#2}{#3}{#4}{#3}{#4}}
\def\Decl@Mn@Open#1#2#3{\Decl@Mn@Delim{#1}{\mathopen}{#2}{#3}}
\def\Decl@Mn@Close#1#2#3{\Decl@Mn@Delim{#1}{\mathclose}{#2}{#3}}
\Decl@Mn@Open{\llangle}{mnomx}{'164}
\Decl@Mn@Close{\rrangle}{mnomx}{'171}
\Decl@Mn@Open{\lmoustache}{mnomx}{'245}
\Decl@Mn@Close{\rmoustache}{mnomx}{'244}
\makeatother

\makeatletter
\DeclareRobustCommand\widecheck[1]{{\mathpalette\@widecheck{#1}}}
\def\@widecheck#1#2{%
    \setbox\z@\hbox{\m@th$#1#2$}%
    \setbox\tw@\hbox{\m@th$#1%
       \widehat{%
          \vrule\@width\z@\@height\ht\z@
          \vrule\@height\z@\@width\wd\z@}$}%
    \dp\tw@-\ht\z@
    \@tempdima\ht\z@ \advance\@tempdima2\ht\tw@ \divide\@tempdima\thr@@
    \setbox\tw@\hbox{%
       \raise\@tempdima\hbox{\scalebox{1}[-1]{\lower\@tempdima\box
\tw@}}}%
    {\ooalign{\box\tw@ \cr \box\z@}}}
\makeatother

\makeatletter
\def\udots{\mathinner{\mkern2mu\raise\p@\hbox{.}
\mkern2mu\raise4\p@\hbox{.}\mkern1mu
\raise7\p@\vbox{\kern7\p@\hbox{.}}\mkern1mu}}
\makeatother

\newcommand{\gt}{>}

\renewcommand{\(}{\begin{equation}}
\renewcommand{\)}{\end{equation}}
\newcommand{\bea}{\begin{eqnarray*}}
\newcommand{\eea}{\end{eqnarray*}}

\usepackage{cleveref}

\crefformat{section}{\S#2#1#3} 
\crefformat{subsection}{\S#2#1#3}
\crefformat{subsubsection}{\S#2#1#3}

\theoremstyle{italics}
\newtheorem{theorem}{Theorem}[section]
\newtheorem{lemma}[theorem]{Lemma}
\newtheorem{prop}[theorem]{Proposition}

\theoremstyle{definition}
\newtheorem{defn}[theorem]{Definition}

\newtheorem{example}[theorem]{Example}

\newtheorem{remark}[theorem]{Remark}
\newtheorem{note[theorem]}{Note}

\usepackage{amsfonts}
\usepackage{colortbl}

\begin{document}

\title{
    Lift of fractional D-brane charge to equivariant Cohomotopy theory
  }

\author{
  Simon Burton\thanks{Department of Physics and Astronomy, University College London, London, WC1E 6BT, UK}, \;
  Hisham Sati\thanks{Division of Science and Mathematics, New York University, Abu Dhabi, UAE}, \;
  Urs Schreiber\thanks{Division of Science and Mathematics, New York University, Abu Dhabi, UAE, on leave from Czech Academy of Science}
}

\maketitle

\begin{abstract}
  The lift of K-theoretic D-brane
  charge to M-theory was recently hypothesized to
  land in Cohomotopy cohomology theory. To further check
  this \emph{Hypothesis H}, here we explicitly compute the constraints
  on fractional D-brane charges at ADE-orientifold singularities
  imposed by the existence of lifts from equivariant K-theory
  to equivariant Cohomotopy theory, through Boardman's
  comparison homomorphism.
  We check the relevant cases and find that
  this condition singles out precisely those fractional D-brane
  charges which do not take irrational values, in any twisted sector.
  Given that the possibility of irrational D-brane charge has been
  perceived as a paradox in string theory, we conclude that
  Hypothesis H serves to resolve this paradox.

  Concretely, we first explain that the Boardman homomorphism,
  in the present case, is the map
  from the Burnside ring to the representation ring
  of the singularity group
  given by forming virtual permutation representations.
  Then we describe an explicit algorithm that computes
  the image of this comparison map for any finite group.
  We run this algorithm for binary Platonic groups, hence for finite subgroups of ${\rm SU}(2)$;
  and we find explicitly that for the
  three exceptional subgroups and
  for the first few cyclic and binary dihedral subgroups
  the comparison morphism surjects
  precisely onto the sub-lattice
  of the real representation ring spanned by the non-irrational characters.
\end{abstract}

\tableofcontents

\vfill

\newpage

We present here a curious computation in elementary representation
theory (Theorem \ref{CokernelOfBetaInVariousExamples} below),
the background for which we introduce in detail in
\cref{TheBoardmanHomomorphism} below.
Besides its mathematical content, which is of interest in itself
as explained in \cref{GeneralFacts} below,
we argue that this result impacts on
open questions in the foundations of string theory, as
 we explain next in \cref{BraneChargeQuantizationInMTheory}.

\section{Fractional brane charge quantization in M-theory}
\label{BraneChargeQuantizationInMTheory}

\noindent {\bf The issue of irrational D-brane charge.}
It is a long-standing conjecture \cite[Sec. 5.1]{Witten98} that the charge lattice of fractional D-branes \cite{DouglasGreeneMorrison97} stuck at $G$-orientifold singularities is the $G$-equivariant
K-theory of the singular point, hence the representation ring
of $G$ (e.g. \cite{Greenlees05}).
However, it was argued already in \cite[4.5.2]{BDHKMMS02} that not all elements of the representation ring
can correspond to viable D-brane charges, and a rationale was sought for identifying a sub-lattice of physical charges.
Independently, in \cite[(2.8)]{BachasDouglasSchweigert00} it was
highlighted that the possibility of irrational D-brane RR-charge
is a ``paradox'' \cite{BachasDouglasSchweigert00} that needs to be resolved.

\medskip
However, for fractional D-branes stuck at $G$-orbifold singularities, the RR-charge is rationally proportional
to the character (recalled as Def. \ref{Character} below)
of the corresponding representation (by \cite[(3.8)]{DouglasGreeneMorrison97}\cite[(4.65)]{BilloCrapsRoose00}\cite[(4.102)]{RecknagelSchomerus13}):
\begin{center}
\hypertarget{Table1}{}
\begin{tabular}{|c||c|c|c|}
  \hline
  \begin{tabular}{c}
    \bf Orbifold
    \\
    \bf D-brane theory
  \end{tabular}
  &
  \small
  \color{blue}
  \begin{tabular}{c}
    D-brane charge
    \\
    at $G$-singularity
  \end{tabular}
  &
  \small
  \color{blue}
  \begin{tabular}{c}
    mass
  \end{tabular}
  &
  \small
  \color{blue}
  \begin{tabular}{c}
    charge in
    \\
    $g$-twisted sector
  \end{tabular}
  \\
  \hline
  \multicolumn{1}{c|}{}
  &
  $V
    \in
    \begin{array}{c}
      \mathrm{KO}_G
      \\
      \begin{rotate}{90}
       $\small\!\simeq$
      \end{rotate}
      \\
      R_{{}_{\mathbb{R}}}(G)
    \end{array}
  $
  &
  $
    \begin{array}{c}
      \mathrm{dim}(V)
      \\
      \begin{rotate}{90}
       $\small\!=$
      \end{rotate}
      \\
      \chi_{{}_{V}}(e)
    \end{array}
  $
  &
  $
      \begin{array}{c}
      \mathrm{tr}_{{}_{V}}(g)
      \\
      \begin{rotate}{90}
       $\small\!=$
      \end{rotate}
      \\
      \chi_{{}_{V}}(g)
    \end{array}
  $
  \\
  \hline
  \begin{tabular}{c}
    \bf Representation
    \\
    \bf theory
  \end{tabular}
  &
  \small
  \color{blue}
  \begin{tabular}{c}
    linear
    \\
    $G$-representation
  \end{tabular}
  &
  \small
  \color{blue}
  \begin{tabular}{c}
    character value
    \\
    at neutral element
  \end{tabular}
  &
  \small
  \color{blue}
  \begin{tabular}{c}
    character value
    \\
    at element $g \in G$
  \end{tabular}
  \\
  \hline
\end{tabular}
\end{center}
\noindent
{\footnotesize \bf Table 1 -- The translation} 
{\footnotesize between fractional D-brane charge 
at $G$-orbifold singularities and characters of linear 
$G$-representations.}

\medskip

In view of \cite{BachasDouglasSchweigert00} this 
means that representations with irrational characters
would reflect physically spurious fractional D-brane charges,
even though they do appear in equivariant K-theory.
Several authors tried to find a resolution of the paradox
of possible irrational D-brane charge
\cite{Taylor00}\cite{Zhou01}\cite{Rajan02}
but the situation has remained inconclusive.

\medskip

\noindent {\bf The open problem of formulating M-theory.}
Of course, perturbative string theory, where this paradox is encountered,
is famously supposed to be just a limiting case of an elusive non-perturbative theory with working title \emph{M-theory}
(see e.g.\cite{Duff99B}\cite[Sec. 2]{HSS18}).
It is to be expected that the full M-theory implies
constraints not seen from the string perturbation series.
This is directly analogous to the now popular statement that
perturbative string theory, in turn, implies constraints not seen
in effective quantum field theory, separating the
``Landscape'' of effective field theories that do lift to
perturbative string vacua from the ``Swampland'' of those
that do not \cite{Vafa05}, indicated on the right of
\hyperlink{Figure1}{Figure 1}.
However, despite the tight web of hints on its limiting cases,
actually formulating M-theory remains an open problem
(see \cite[p. 2]{NicolaiHelling98}\cite[Sec. 12]{Moore14}\cite[p. 2]{ChesterPerlmutter18}\cite[Sec. 2]{HSS18}).

\medskip
\hypertarget{HypothesisH}{}
\noindent {\bf Hypothesis H -- M-brane charge quantization in Cohomotopy.}
But recent analysis of the
structure of the bouquet of super $p$-brane WZ terms
from the point of view of homotopy theory
(\cite{Sati13}\cite{FSS16a}\cite{FSS15}\cite{BSS18}, see \cite{FSS19a} for review)
has suggested a new hypothesis about the precise
nature of the brane charge structure in M-theory. This
{\it Hypothesis H} \cite{FSS19b}\cite{FSS19c} asserts that
the generalized cohomology theory which quantizes brane charge in M-theory,
in analogy to how twisted K-theory is expected to quantize D-brane charge
in string theory (see \cite[Sec. 1]{BSS18}\cite{GS19} for pointers),
is unstable \emph{Cohomotopy} cohomology theory
\cite{Borsuk36}\cite{Spanier49}\cite{Peterseon56},
specifically \emph{$J$-twisted Cohomotopy} \cite[Def. 3.1]{FSS19b}.
On flat orbifold spacetimes this is
unstable \emph{equivariant Cohomotopy} \cite[(2)]{SS19}\cite{HSS18},
traditionally considered in its stabilized approximation
\cite{Segal70}\cite{Carlsson84}\cite{Lueck}.
This is indicated on the left of \hyperlink{Figure1}{Figure 1},
which contextualizes our proposal and provides a global perspective:

{\hypertarget{Figure1}{}}
$$
  \xymatrix@R=9pt{
    \fbox{
      M-theory
    }
    \ar[rr]^-{
      \mbox{
        \tiny
        \color{blue}
        \begin{tabular}{c}
          small coupling
          \\
          limit
        \end{tabular}
      }
    }
    &&
    \fbox{
      \begin{tabular}{c}
        Perturbative
        \\
        string theory
      \end{tabular}
    }
    \ar[rr]^-{
      \mbox{
        \tiny
        \color{blue}
        \begin{tabular}{c}
          low energy
          \\
          limit
        \end{tabular}
      }
    }
    &&
    \fbox{
      \begin{tabular}{c}
        Effective
        \\
        quantum field theory
      \end{tabular}
    }
    \\
    \fbox{
      \begin{tabular}{c}
        M-brane charge in
        \\
        Cohomotopy theory
        \\
        {\small
          (\hyperlink{HypothesisH}{\it Hypothesis H})
          }
      \end{tabular}
    }
    \ar[rr]^-{
      \mbox{
        \tiny
        \color{blue}
        \begin{tabular}{c}
          Boardman
          \\
          homomorphism
        \end{tabular}
      }
    }_-{\beta}
    &&
    \fbox{
      \begin{tabular}{c}
        D-brane charge
        \\
        in K-theory
        \\
        \small (traditional conjecture)
      \end{tabular}
    }
    \ar[rr]^-{
      \mbox{
        \tiny
        \color{blue}
        \begin{tabular}{c}
          Chern
          \\
          character
        \end{tabular}
      }
    }_-{\mathrm{ch}}
    &&
    \fbox{
      \begin{tabular}{c}
        Fluxes in
        \\
        ordinary Cohomology
        \\
        $\phantom{A}$
      \end{tabular}
    }
    \\
    &
    \mathclap{
    \hspace{4cm}
    \mbox{
      \footnotesize
      \begin{tabular}{rcl}
        $\mathrm{image}(\beta)$ &$\!\!\!=\!\!\!$& $\phantom{\mbox{not}}$ liftable to M-theory
        \\
        $\mathrm{cokernel}(\beta)$ &$\!\!\!=\!\!\!$& not liftable to M-theory
      \end{tabular}
    }
    }
    & &
    \mathclap{
    \hspace{4cm}
    \mbox{
      \footnotesize
      \begin{tabular}{rcl}
        $\mathrm{image}({\rm ch})$ &$\!\!\!=\!\!\!$& Landscape
        \\
        $\mathrm{cokernel}({\rm ch})$ &$\!\!\!=\!\!\!$& Swampland
      \end{tabular}
    }
    }
  }
$$

\vspace{-.2cm}

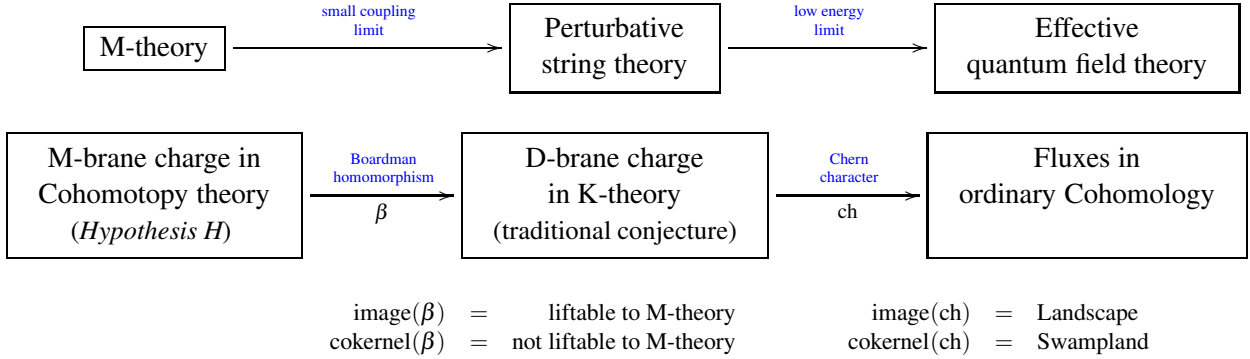
\noindent {\footnotesize {\bf Figure 1:} Relations between M-brane charges in
M-theory via  Cohomotopy and D-brane charges in string theory via K-theory. }

\vspace{.3cm}

\noindent {\bf The Boardman homomorphism to equivariant K-theory.}
Displayed on the left of \hyperlink{Figure1}{\it Figure 1}
is the \emph{Boardman homomorphism} \cite[II.6]{Adams74},
which, under \emph{Hypothesis H},
is the M-theoretic analog of the Chern character
(that maps D-brane charge
in string theory to flux forms in effective field theory)
now mapping brane charge in non-perturbative M-theory to its perturbative approximation in string theory.
As discussed in detail in \cite[3.1.2]{SS19},
classical results in equivariant homotopy theory (see \cite{Blu17})
identify \cite{Segal70}\cite[8.5.1]{tomDieck79}\cite[1.13]{Lueck}
the stable equivariant Cohomotopy of a $G$-orbifold singularity
with
the \emph{Burnside ring} -- a non-linear combinatorial analog of
the $G$-linear representation ring. Furthermore,
under this identification, the Boardman homomorphism to
equivariant K-theory identifies with the homomorphism that linearizes
a combinatorial $G$-set (of M-branes) to a linear $G$-representation
(of D-brane Chan-Paton factors), see
\hyperlink{Figure2}{Figure 2}.
We lay out the relevant definitions in full detail in
\cref{TheBoardmanHomomorphism} below.
As explained in \cite{SS19}, we have the following picture:

\begin{center}
\hypertarget{Figure2}{}
\begin{tikzpicture}[scale=.8]

  \begin{scope}[shift={(0,1)}]

  \draw node at (0,7.2)
    {
      \tiny
      \color{blue}
      \begin{tabular}{c}
        equivariant Cohomotopy
        \\
        vanishing at infinity
        \\
        of Euclidean $G$-space
        \\
        in compatible RO-degree $V$
      \end{tabular}
    };

  \draw node at (0,6)
    {
      $
      \pi^{V}_{{}_{G}}
      \big(
        \big(
          \mathbb{R}^V
        \big)^{\mathrm{cpt}}
      \big)
    $
    };

  \draw[->] (0+2,6)
    to
      node {\colorbox{white}{\small $\Sigma^\infty$}}
      node[above]
        {
          \raisebox{.3cm}{
          \tiny
          \color{blue}
          stabilization
          }
        }
    (6-.8,6);

  \draw node at (6,6.9)
    {
      \tiny
      \color{blue}
      \begin{tabular}{c}
        stable
        \\
        equivariant
        \\
        Cohomotopy
      \end{tabular}
    };

  \draw node at (6,6)
    {
      $
        \mathbb{S}_G^0
      $
    };

  \draw[->] (6+.7,6)
    to
      node{\colorbox{white}{\footnotesize $\beta$}}
      node[above]
        {
          \raisebox{.3cm}{
          \tiny
          \color{blue}
          \begin{tabular}{c}
            Boardman
            \\
            homomorphism
          \end{tabular}
          }
        }
    (11-.9,6);

  \draw node at (6,5.5)
    {
      \begin{rotate}{270}
        $\!\!\simeq$
      \end{rotate}
    };

  \draw node at (6,5)
    {
      $
        A_G
      $
    };

  \draw node at (6,4.3)
    {
      \tiny
      \color{blue}
      \begin{tabular}{c}
        Burnside
        \\
        ring
      \end{tabular}
    };

  \draw[->] (6+.7,5) to
    node
    {
      \raisebox{-3pt}{
      \colorbox{white}
      {
        \small
        $
        \underset
        {
          \mbox{
            \tiny
            \color{blue}
            linearization
          }
        }
        {\footnotesize
          \mathbb{R}[-]
        }
        $
      }
      }
    }
    (11-.7,5);

  \draw node at (11,6.8)
    {
      \tiny
      \color{blue}
      \begin{tabular}{c}
        equivariant
        \\
        K-theory
      \end{tabular}
    };

  \draw node at (11,6)
    {
      $
        \mathrm{KO}_G^0
      $
    };

  \draw node at (11,5.5)
    {
      \begin{rotate}{270}
        $\!\!\!\!\simeq$
      \end{rotate}
    };

  \draw node at (11.2,5)
    {
      $
        \mathrm{RO}(G)
      $
    };

  \draw node at (11,4.3)
    {
      \tiny
      \color{blue}
      \begin{tabular}{c}
        representation
        \\
        ring
      \end{tabular}
    };

  \end{scope}

\begin{scope}[shift={(0,.4)}]

  \draw node at (0,4)
    {$
      \overset{
      }{
        \mbox{
          \tiny
          e.g. one $O^{{}^{-}}\!\!$-plane
          and two branes
        }
      }
    $};

  \draw node at (0,3)
    {
      $\overbrace{\phantom{AAAAAAAAAAAAAAAAAAA}}$
    };

  \draw node at (6,4)
    {$
      \overset{
      }{
        \mbox{
          \tiny
          \begin{tabular}{c}
            minus the trivial $G$-set
            \\
            with two regular $G$-sets
          \end{tabular}
        }
      }
    $};

  \draw node at (11,3)
    {
      $\overbrace{\phantom{AAAAAA}}$
    };

  \draw node at (11,4)
    {$
      \overset{
      }{
        \mbox{
          \tiny
          \begin{tabular}{c}
            minus the trivial $G$-representation
            \\
            plus two times the regular $G$-representation
          \end{tabular}
        }
      }
    $};

  \draw node at (6,3)
    {
      $\overbrace{\phantom{AAAAAA}}$
    };

  \draw[dashed]
    (0,0) circle (2);

  \draw (0,2+.2) node {\footnotesize $\infty$};
  \draw (0,-2-.2) node {\footnotesize $\infty$};
  \draw (2+.3,0) node {\footnotesize $\infty$};
  \draw (-2-.3,0) node {\footnotesize $\infty$};

  \draw[fill=white]
    (0,0) circle (.07);

  \draw[fill=black]
    (18:.6) circle (.07);
  \draw[fill=black]
    (18+90:.6) circle (.07);
  \draw[fill=black]
    (18+180:.6) circle (.07);
  \draw[fill=black]
    (18+270:.6) circle (.07);

  \draw[fill=black]
    (58:1.2) circle (.07);
  \draw[fill=black]
    (58+90:1.2) circle (.07);
  \draw[fill=black]
    (58+180:1.2) circle (.07);
  \draw[fill=black]
    (58+270:1.2) circle (.07);

  \draw[|->] (3.6,0) to ++(.5,0);
  \draw[|->] (8.4,0) to ++(.5,0);

  \begin{scope}[shift={(6,.2)}]

    \draw[fill=white] (0,2) circle (0.07);

    \draw (5,2) node {\small $+\mathbf{1}_{{}_{\mathrm{triv}}}$ };

    \draw[|->, olive] (-.05,2.12) arc (30:325:.2);

    \begin{scope}[shift={(0,.3)}]
    \draw[fill=black] (0,1) circle (0.07);
    \draw[fill=black] (0,.5) circle (0.07);
    \draw[fill=black] (0,0) circle (0.07);
    \draw[fill=black] (0,-.5) circle (0.07);

    \draw[|->, olive] (0-.1,1-.05) arc (90+6:270-16:.2);
    \draw[|->, olive] (0-.1,.5-.05) arc (90+6:270-16:.2);
    \draw[|->, olive] (0-.1,0-.05) arc (90+6:270-16:.2);

    \draw[|->, olive] (0+.1,-.5+.1) to[bend right=60] (0+.1,1-.03);

    \draw (5,0.25) node {\small $-\mathbf{4}_{{}_{\mathrm{reg}}}$ };

    \end{scope}

    \begin{scope}[shift={(0,-1.7)}]
    \draw[fill=black] (0,1) circle (0.07);
    \draw[fill=black] (0,.5) circle (0.07);
    \draw[fill=black] (0,0) circle (0.07);
    \draw[fill=black] (0,-.5) circle (0.07);

    \draw[|->, olive] (0-.1,1-.05) arc (90+6:270-16:.2);
    \draw[|->, olive] (0-.1,.5-.05) arc (90+6:270-16:.2);
    \draw[|->, olive] (0-.1,0-.05) arc (90+6:270-16:.2);

    \draw[|->, olive] (0+.1,-.5+.1) to[bend right=60] (0+.1,1-.03);

    \draw (5,0.25) node {\small $-\mathbf{4}_{{}_{\mathrm{reg}}}$ };

    \end{scope}

\end{scope}

  \end{scope}
\end{tikzpicture}
\end{center}
\vspace{-4mm}
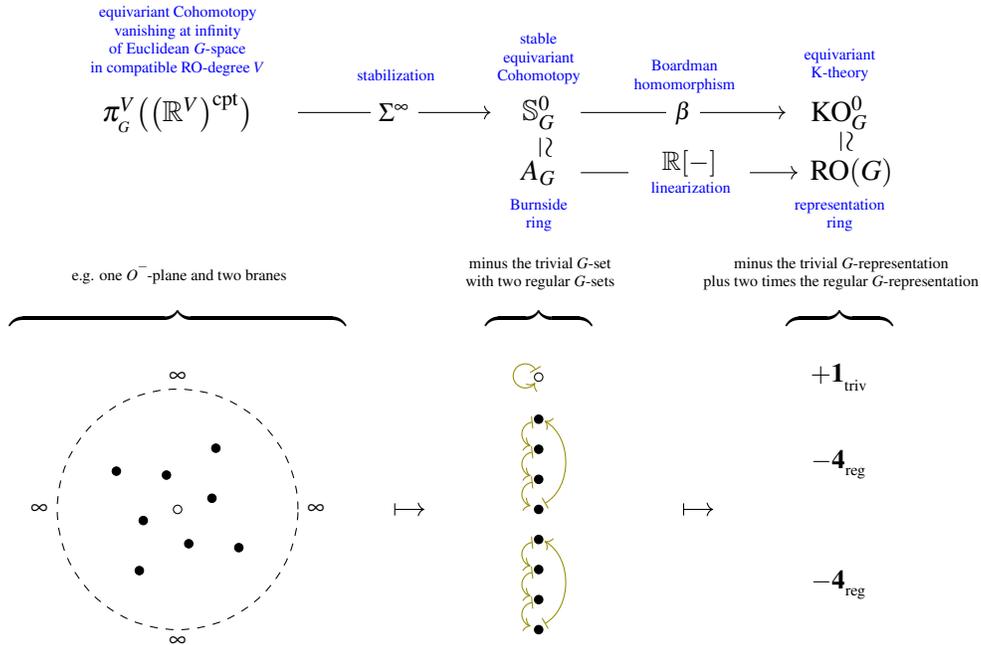
\noindent {\bf \footnotesize Figure 2 -- Virtual $G$-representations of fractional branes classified by
equivariant Cohomotopy} {\footnotesize
in the vicinity of ADE-singularities, seen through the
Boardman homomorphism to equivariant K-theory, according to \cite{SS19}.
Shown is a situation for $G = \mathbb{Z}_4$.}

\medskip

With this explicit form 
of the equivariant Boardman homomorphism at $G$-singularities
($\beta \simeq \mathbb{R}[-]$)
in hand,
we obtain here a concrete computational handle on the implications
of \hyperlink{HypothesisH}{\it Hypothesis H}, 
which may serve as further tests of this hypothesis.

\medskip

\noindent {\bf Fractional D-brane charge at orientifold singularities under Hypothesis H.}
In \cite{SS19} we had laid out in detail how, under
\hyperlink{HypothesisH}{\it Hypothesis H}, equivariant Cohomotopy
measures charges of M5-branes stuck at ADE-orientifold singularities,
and how this translates, under the Boardman homomorphism, to
charges in equivariant K-theory of fractional D-branes
stuck at these singularities. The main result of \cite{SS19}
is that the lift of fractional D-brane charge on orientifolds
from equivariant K-theory through the Boardman homomorphism
to equivariant Cohomotopy (as in \hyperlink{Figure2}{Figure 2}) 
implies expected 
anomaly cancellation conditions for D-branes in orientifolds,
the \emph{RR-field tadpole cancellation conditions}.

\medskip
\noindent {\bf Approach: Numerical fractional D-brane charge under Hypothesis H.}
Here we investigate further implications of this lift of brane charge
from K-theory to Cohomotopy:
Using the explicit form of the Boardman homomorphism
explained below in \cref{TheBurnsideRingAndEquivariantKTheory}, we develop
in \cref{TheAlgorithm} an effective algorithm for
explicitly computing the image of the Boardman homomorphism as
a sublattice in the representation ring, 
and hence in the equivariant K-theory ring of fractional D-brane charges.
This allows us to concretely read off which fractional D-brane
charges lift to M-theory under \hyperlink{HypothesisH}{Hypothesis H},
as indicated in \hyperlink{Figure1}{Figure 1}.

\medskip
\noindent {\bf Physics result.} Our main result, Theorem \ref{CokernelOfBetaInVariousExamples} below,
establishes that, in all the examples of singularity groups which we compute,
the image of the Boardman homomorphism in the real 
representation ring
(appropriate for orientifold charges) consists precisely of the
\emph{integral} characters; these are equivalently the
non-irrational characters (by Prop. \ref{RationalCharactersAreInteger} below).
Under the translation of
\hyperlink{Table1}{Table 1}
this means that \hyperlink{HypothesisH}{\it Hypothesis H} removes
precisely the irrational D-brane
charges from the charge lattice of fractional D-branes stuck at
ADE-singularities. Hence, in as far as irrational D-brane charge
is a paradox \cite{BachasDouglasSchweigert00}\cite{Taylor00}\cite{Zhou01}\cite{Rajan02},
\hyperlink{HypothesisH}{Hypothesis H} resolves this paradox.

\medskip
We explain now in \cref{TheBoardmanHomomorphism} 
how these issues are formulated mathematically,
translating the question of lifting D-brane charge to Cohomotopy
to a question purely in representation theory.

\section{The equivariant Boardman homomorphism $\beta$}
\label{TheBoardmanHomomorphism}

\subsection{The Burnside ring and K-theory}
\label{TheBurnsideRingAndEquivariantKTheory} 

Finite group actions
control orbifold spacetime singularities in string theory.
For $G$ a finite group one may consider linear as well as purely combinatorial \emph{actions} of $G$ on some set
(see e.g. \cite{Dr71}\cite{tomDieck79}\cite{Ker}\cite{LP}). We will be concerned with the relation between these  two
types of actions (see e.g. \cite{tomDieck79}\cite{Be}\cite{Bo}).

\medskip

\noindent {\bf Basics.} Traditionally,
 for $k$ any field, the linear actions receive more attention as they are the $k$-linear \emph{representations}
of $G$, namely the group homomorphisms $G \to \mathrm{Aut}_k(V)$ from $G$ to the $k$-linear invertible
maps from a given $k$-vector space $V$ to itself. More elementary than this concept is that of plain \emph{$G$-sets},
which are instead group homomorphism of the form $G \to \mathrm{Aut}(S)$, from $G$ to all the invertible
functions from some set $S$ to itself (permutations).

\medskip
To emphasize that these two concepts, while clearly different, are conceptually related, one may appeal to the
lore of the ``field with one element'' $\mathbb{F}_1$ \cite{Ti}\cite{KS}\cite{So}\cite{CCM}\cite{Ma}
(see \cite{Th}\cite{Lo} for recent surveys)  and regard plain sets as vector spaces over $\mathbb{F}_1$,
and plain set-theoretic permutations as being the $\mathbb{F}_1$-linear maps
$\mathrm{Aut}(S) = \mathrm{Aut}_{\mathbb{F}_1}(S)$.

\medskip
In any case, for every finite group $G$ and field $k$, the isomorphism classes of finite nonlinear  actions
and of finite-dimensional linear actions of $G$ form two rings-without-negatives,
\begin{equation}
  \label{TheRingsWithoutNegatives}
  \left(G \mathrm{Set}^{\mathrm{fin}}/_\sim,\; \sqcup, \times\right)
  =
  \left(
    G \mathrm{Rep}^{\mathrm{fin}}_{\mathbb{F}_1}/_\sim,\; \oplus_{\mathbb{F}_1}, \otimes_{\mathbb{F}_1}
  \right)
 \qquad
 \text{and}
 \qquad
  \left(G \mathrm{Rep}^{\mathrm{fin}}_k/_\sim,\; \oplus_k, \otimes_k\right),
\end{equation}
where addition is given by disjoint union of $G$-sets and by direct sum of $G$-representations, respectively, while the product operation
is given by Cartesian product of $G$-sets and by tensor product of $G$-representations, respectively.

\medskip
We will be interested in a canonical comparison map from $G$-sets to $G$-representation:
By forming $k$-linear combinations of elements of a finite set, every
$G$-set
$G \;\;\;\raisebox{-1pt}{\begin{rotate}{90} $\curvearrowright$\end{rotate}}\;S$
in $A(G)$ spans a $k$-linear representation
\begin{equation}
  \label{PermutationRepresentation}
  \xymatrix{ k[S] \ar@(ul,ur)|{\;\!G\;\!} }
  \phantom{AAA}
  g(v)
  =
  g\Big(
    \underset{s \in S}{\sum} \, \underset{ \in k }{\underbrace{v_s}} \cdot \underset{ \in S}{\underbrace{s}}
  \Big)
  \;:=\;
  \underset{h \in G}{\sum} v_i \cdot g(s)\;.
\end{equation}
These are called the \emph{permutation representations}.
The archetypical example is the \emph{regular representation} $k[G]$, which is the linearization of $G$
acting on its own underlying set by group multiplication from the left.
A simple but important example is the trivial 1-dimensional representation $\mathbf{1}$ of $G$, which is
the linearization of the point, regarded as a $G$-set:
\begin{equation}
  \label{trivialRepresentation}
  \mathbf{1} \simeq k[\ast]\;.
\end{equation}
The construction \eqref{PermutationRepresentation} of permutation representations from  $G$-sets  is
clearly linear and multiplicative, in that it extends to a homomorphism between the above
rings-without-negatives \eqref{TheRingsWithoutNegatives}
\begin{equation}
  \label{betaOnPlainReps}
  \xymatrix{
    \left(G \mathrm{Set}^{\mathrm{fin}}/_\sim, \sqcup, \times \right)
    \ar[rr]^-{k[-]}
    &&
    \left( G \mathrm{Rep}^{\mathrm{fin}}/_\sim , \oplus, \otimes \right)
  }.
\end{equation}
While this establishes the canonical comparison map between nonlinear and linear $G$-actions, there is nothing
much of interest to be said about it.

\medskip
\noindent {\bf Passage to K-theory.}
This situation changes drastically as soon as we consider not
just plain $G$-sets and $G$-representations, but also their ``anti-$G$-sets'' and ``anti-$G$-representations'',
namely as we group-complete the rings-without-negatives in \eqref{TheRingsWithoutNegatives} to actual rings, by
adjoining additive inverses for all elements.
Concretely, a \emph{virtual} $G$-representation is represented by a pair $(V^+, V^-)$ of two $G$-representations,
thought of as a plain $G$-representation
$V^+$ and an \emph{anti}-$G$-representation $V^-$; and the K-group completion is obtained by quotienting out from the
evident group (ring) that these virtual representation/anti-representation pairs form the equivalence relation
$$
  (V, V) \sim 0
  \phantom{AAA}
  \mbox{for all $V \in G\mathrm{Rep}^{\mathrm{fin}}_k/_\sim$}
  \,.
$$
This can viewed  as exhibiting pair-creation/annihilation of bound states of a representation with its own
anti-representation. The resulting ring is called the \emph{representation ring} of $G$, denoted
\begin{equation}
  \label{RepresentationRing}
  R_k(G) \;:=\; K\left( G \mathrm{Rep}^{\mathrm{fin}}_k/_{\sim}, \oplus, \otimes \right)
  \,.
\end{equation}
An analogous construction applies to virtual $G$-sets represented by pairs consisting of a $G$-set and an
anti-$G$-set, subject to the relation of pair-creation/annihilation of bound states of a $G$-set $S$ with its
anti-$G$-set:
$$
  (S,S) \sim 0
  \phantom{AA}
  \mbox{for all $S \in G \mathrm{Set}^{\mathrm{fin}}$}.
$$
Now the resulting ring is known as the \emph{Burnside ring} of $G$  (see e.g. \cite{Sol}\cite{Dr71}\cite{tomDieck79}\cite{Ker}), denoted
\begin{equation}
  \label{BurnsideRing}
  A(G) \;:=\; K\left( G \mathrm{Set}^{\mathrm{fin}}/_{\sim}, \sqcup, \times \right)
  \,.
\end{equation}
While plain permutation representations \eqref{PermutationRepresentation} generally form just a very
small subset of all isomorphism classes of $k$-linear representations, the passage to \emph{virtual}
permutation representations drastically changes the picture: Since every plain permutation representation
decays into a direct sum of irreducible linear representations, the \emph{formal difference} of two permutation representations in a virtual permutation representation may partially cancel out to become equal, in the
representation ring, to a representation that is not itself a plain permutation representation.

\medskip
For example, in the simplest non-trivial case, where $G = C_2$ is the finite group of order 2, the
1-dimensional alternating representation $\mathbf{1}_{\mathrm{alt}}$ is clearly not a permutation
representation itself. But it is a direct summand in the regular representation
$k[C_2] = \mathbf{1} + \mathbf{1}_{\mathrm{alt}}$.
Since the other summand is a permutation representation, $\mathbf{1} = k[\ast]$, the alternating
representation may then be isolated as the formal difference
$\mathbf{1}_{\mathrm{alt}} \;=\; k[C_2] - k[\ast]$, thus as a virtual permutation representation.

\medskip

\medskip
\noindent {\bf The comparison morphism.} Therefore, while the step from plain to \emph{virtual} actions and representations is small, it has drastic
consequences, as it potentially reduces much of linear representation theory to pure combinatorics.
In order to quantify this effect, one observes that the construction \eqref{betaOnPlainReps} of permutation
 representations evidently extends linearly to virtual $G$-sets and virtual $G$-representations, to a homomorphism
\begin{equation}
  \label{TheComparisonMap}
  \fbox{
  \xymatrix{
    A(G)
      \ar[rr]^{ \beta := k[-] }
    &&
    R_k(G)
  }
  }
\end{equation}
from the Burnside ring \eqref{BurnsideRing} to the representation ring \eqref{RepresentationRing}, taking virtual $G$-sets to virtual $G$-representations.

\noindent We may associate to the  representation theory of $G$ over $k$ the following interpretation.

\begin{center}
\begin{tabular}{|c||c|}
\hline
{\bf Space} & {\bf Meaning} \\ \hline \hline
 Cokernel of $\beta$ &
    Linear algebra
    invisible to
    pure combinatorics
  \\ \hline
  Kernel of $\beta$
  &
Pure combinatorics
   invisible to
    linear algebras\\
    \hline
    \end{tabular}
\end{center}
Intuition might suggest that generally the cokernel of $\beta$ is large, while the kernel of $\beta$ is generally small.
This is indeed the case for the restriction \eqref{betaOnPlainReps} of $\beta$ to actual (as opposed to virtual)
$G$-sets and $G$-representations. However,  the inclusion of anti-$G$-sets and anti-representations and passage
to the K-groups of virtual $G$-sets and virtual $G$-representations completely changes the picture.
It turns out that the kernel of $\beta$ almost never vanishes: in characteristic zero
its rank is the difference of the number of non-cyclic by cyclic subgroups of $G$ (Prop. \ref{BetaInjectiveOnlyForCyclicGroups} below).
At the same time, classical results give that the cokernel of $\beta$ often vanishes:
for instance, over $k =\mathbb{Q}$ it vanishes for all cyclic groups (Prop. \ref{RationalBetaForCyclic} below),
as well as for all $p$-groups (\cite{Segal72}, recalled as Prop. \ref{SegalTheorem} below), while for $G = S_n$
a symmetric group, the cokernel of
$\beta$ vanishes even for all fields $k$ of characteristic zero (Prop. \ref{SymmetricSurjectivityOfBeta} below).

\medskip

This is a remarkable state of affairs, which deserves further investigation.

\medskip
\noindent The mathematical incarnation of the goals and results stated at the end of \cref{BraneChargeQuantizationInMTheory}
are the following:

\medskip
\noindent {\bf Goal.} 
Here our goal is to give explicit descriptions of the cokernel of $\beta$,
over the rational, real and complex numbers, for further concrete examples of finite groups $G$.
We are particularly interested in the case of the \emph{binary Platonic groups}, namely
the finite subgroups of $\mathrm{SU}(2)$ (recalled in  \cref{ThePlatonicGroups}).

\medskip

\noindent {\bf Method.} In \cref{TheAlgorithm} we describe an algorithm for the
image of $\beta$, Theorem \ref{AlgorithmForImageOfBeta} below.
In \cref{ImageOfBetaExamples} we apply this algorithm and compute the cokernel of $\beta$ in various concrete examples, Theorem \ref{CokernelOfBetaInVariousExamples}.
In \cref{Background} we collect some background material.

\medskip

\noindent {\bf Results.} The table in Theorem \ref{CokernelOfBetaInVariousExamples} shows that in all examples computed here, notably for the three exceptional finite subgroups of $\mathrm{SU}(2)$ as well as the seven first cases of binary dihedral groups, the image of $\beta$ in the real representation ring consists precisely of the sub-lattice of
integer characters, hence (by Prop. \ref{RationalCharactersAreInteger}) of non-irrational characters.
The same holds true for larger classes of Examples which we have
computed, but are not showing here. We conjecture that it holds
true for all binary dihedral groups.

%

\subsection{Analysis of the image of $\beta$}
\label{GeneralFacts}

In order to set the scene for
the considerations below,
we record some well-known general facts about the image of the comparison morphism $\beta$ \eqref{TheComparisonMap}.

\begin{defn}[Sub-lattice of integer-valued characters]
  \label{SubLatticOfIntegerValuedCharacters}
  For $G$ a finite group and $k$ a field, write
  \begin{equation}
    \label{InclusionOfIntegerCharacters}
    R_k^{\mathrm{int}}(G)
    \xymatrix{\ar@{^{(}->}[r]&}
    R_k(G)
  \end{equation}
  for the sub-lattice of the representation ring given by those
  representations $V \in R_k(G)$ whose characters $\chi_{{}_{V}}$ (Def. \ref{Character})
  are integer-valued
  $$
    \chi_{{}_V} \;:\; g \longmapsto \chi_{{}_V}(g) \in \mathbb{Z} \subset k
    \,.
  $$
\end{defn}
\begin{prop}[$\beta$ takes values in integer-valued characters]
  \label{CorestrictionToIntegerCharacters}
  The comparison morphism $\beta$ \eqref{TheComparisonMap}
  has its image inside the integer-valued characters (Def. \ref{SubLatticOfIntegerValuedCharacters}),
  hence it factors as
  \begin{equation}
    \beta
    \;:\;
    \xymatrix{
      A(G)
      \ar[r]^-{  \beta^{\mathrm{int}}_k}
      &
      R^{\mathrm{int}}_k(G)
      \; \ar@{^{(}->}[r]
      &
      R_k(G)
    }.
  \end{equation}
\end{prop}
\begin{proof}
  By Example \ref{CharacterOfPermutationRepresentation}.
\end{proof}
The following is elementary, but important:
\begin{prop}[e.g. {\cite{Naik}}]
  \label{RationalCharactersAreInteger}
  If the ground field $k$ has characteristic zero, then a character (Def. \ref{Character}) that takes values in the
  rational numbers $\mathbb{Q} \subset k$ in fact already takes values in the integers $\mathbb{Z} \subset \mathbb{Q} \subset k$.
  Hence if $R_{k}^{\mathrm{rat}}(G) \subset R_k(G)$ denotes the sublattice of rational-valued characters, in analogy
  to the sub-lattice of integer-valued characters in Def. \ref{SubLatticOfIntegerValuedCharacters}, then these sub-lattices are in fact equal:
  $$
    \xymatrix{
      R^{\mathrm{int}}_{k}(G)
      \ar[r]^-{=}
      &
      R^{\mathrm{rat}}_k(G)
      \subset
      R_k(G)
    }
    \phantom{AAA}
    \mbox{for $\mathbb{Q} \subset k$}
    \,.
  $$
  In particular, if the ground field $k = \mathbb{Q}$ is itself the rational numbers, then all characters are integer-valued characters (Def. \ref{SubLatticOfIntegerValuedCharacters}),
  hence in this case the canonical inclusion \eqref{InclusionOfIntegerCharacters} is an isomorphism:
  $$
    \xymatrix{
      R^{\mathrm{int}}_{\mathbb{Q}}(G)
      \ar[r]^-{=}
      &
      R_{\mathbb{Q}}(G)
    }.
  $$
\end{prop}
\begin{proof}
  In general, characters are cyclotomic integers. Over the rationals the only cyclotomic integers
  are the actual integers.
\end{proof}

\begin{remark}[Factorizations of the comparison map $\beta$]
 \label{FactorizationsOfBeta}
In summary, Prop. \ref{CorestrictionToIntegerCharacters} and Prop. \ref{RationalCharactersAreInteger} say that we have the following commuting diagram of factorizations of the morphism $\beta$ \eqref{TheComparisonMap}
that sends $G$-sets to their linear permutation representation:
$$
  
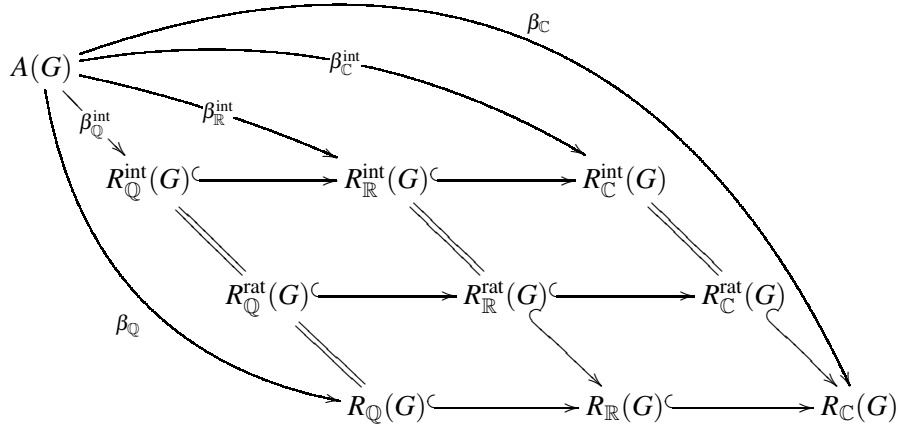
\xymatrix@C=7pt{
    A(G)
    \ar[dr]|{ \beta^{\mathrm{int}}_{\mathbb{Q}} }
    \ar@/^.4pc/[drrr]|{ \beta^{\mathrm{int}}_{\mathbb{R}} }
    \ar@/^2pc/[drrrrr]|{ \beta^{\mathrm{int}}_{\mathbb{C}} }
    \ar@/^6.7pc/[dddrrrrrrr]^{ \beta_{\mathbb{C}} }
    \ar@/_3pc/[dddrrr]_{ \beta_{\mathbb{Q}} }
    \\
    &
    R^{\mathrm{int}}_{\mathbb{Q}}(G)
    \ar@{^{(}->}[rr]
    \ar@{=}[dr]
    &&
    R^{\mathrm{int}}_{\mathbb{R}}(G)
    \ar@{^{(}->}[rr]
    \ar@{=}[dr]
    &&
    R^{\mathrm{int}}_{\mathbb{C}}(G)
    \ar@{=}[dr]
    \\
    &
    &
    R^{\mathrm{rat}}_{\mathbb{Q}}(G)
    \ar@{^{(}->}[rr]
    \ar@{=}[dr]
    &&
    R^{\mathrm{rat}}_{\mathbb{R}}(G)
    \ar@{^{(}->}[rr]
    \ar@{^{(}->}[dr]
    &&
    R^{\mathrm{rat}}_{\mathbb{C}}(G)
    \ar@{^{(}->}[dr]
    \\
    &
    &
    &
    R_{\mathbb{Q}}(G)
    \ar@{^{(}->}[rr]
    &&
    R_{\mathbb{R}}(G)
    \ar@{^{(}->}[rr]
    &&
    R_{\mathbb{C}}(G)
  }
$$
\end{remark}

Hence it is worthwhile to first record what is known about the image of $\beta$ over $\mathbb{Q}$;

\begin{prop}[e.g. {\cite[proof of Prop. 4.5.4]{tomDieck09}}]
  \label{OverQFullRankSublattice}
  Over the rational numbers, $k = \mathbb{Q}$, the image of $A(G) \overset{\beta}{\to}R_{\mathbb{Q}}(G)$ \eqref{TheComparisonMap}
  is at least a sub-lattice of full rank (i.e. has the same \emph{number} of generators as $R_{\mathbb{Q}}(G)$).
  This full-rank sublattice is spanned by the permutation representations \eqref{PermutationRepresentation} of the form
  $\mathbb{Q}[G/C_i]$ for $C_i \subset C_n$ ranging over the cyclic subgroups.
\end{prop}

\begin{prop}[e.g. {\cite[Example (4.4.4)]{tomDieck09}}]
  \label{RationalBetaForCyclic}
  For $G = C_n = \mathbb{Z}/n$ a cyclic group and $k = \mathbb{Q}$ the rational numbers, $\beta$ is an isomorphism
  $$
    \xymatrix{
      A(C_n)
      \ar@{->}[r]^-{\beta}_-{\simeq}
      &
      R_{\mathbb{Q}}(C_n)
    }
    \,.
  $$
\end{prop}

\begin{prop}
  \label{BetaInjectiveOnlyForCyclicGroups}
  The only finite groups $G$ for which $A(G) \xrightarrow{\beta_k} R_k(G)$ is injective over $k = \mathbb{Q}$ (hence over $k = \mathbb{R}, \mathbb{C}$)
  are the cyclic groups $G = C_n$.
\end{prop}
\begin{proof}
  We know that a linear basis for $A(G)$ is given by the cosets $G/H$ for $H$ ranging over conjugacy classes of all subgroups of $G$,
  while a linear basis for $R_{\mathbb{Q}}(G)$ is given by the isomorphism classes of irreducible $\mathbb{Q}$-linear representations.
  But the latter are in bijection to just the \emph{cyclic} subgroups of $G$ (e.g. \cite[Prop. 4.5.4]{tomDieck09} ). This means
  that when $G$ is not itself cyclic, then the cardinality of a linear basis for $A(G)$ is strictly larger than the cardinality
  of a linear basis for $R_{\mathbb{Q}}(G)$, so that no morphism $A(G) \to R_{\mathbb{Q}}(G)$ can be injective.
  On the other hand, when $G$ is a cyclic group then $\beta$ is an isomorphism by Prop. \ref{RationalBetaForCyclic}, and hence in particular
  injective.
\end{proof}
Less immediate is the following result:
\begin{prop}[{\cite{Segal72}}]
  \label{SegalTheorem}
  If the finite group $G$ is a $p$-group, hence if its number of elements is the $n$th power $p^n$ of some prime number $p$
  by some natural number $n \in \mathbb{N}$
  $$
    \vert G \vert
    \;=\;
    p^n
    \,,
  $$
  then over $k = \mathbb{Q}$ the comparison morphism $\xymatrix{A(G) \ar@{->>}[r]^-{\beta} &  R_{\mathbb{Q}}(G)}$ \eqref{TheComparisonMap}
  is surjective.
\end{prop}

The standard representation theory of symmetric groups in terms of Young diagrams and Specht modules
yields the following statement:
\begin{prop}[e.g. {\cite[Section 3]{Dress86}}]
  \label{SymmetricSurjectivityOfBeta}
  Over any ground field $k$ of characteristic zero, and for $G = S_n$ any symmetric group of permutations of $n \in \mathbb{N}$ elements, the comparion map
  $\xymatrix{A(S_n) \ar@{->>}[r] & R_{k}(S_n)}$ \eqref{TheComparisonMap} is surjective.
\end{prop}

For some other classes of finite groups, formulas for the cokernel and kernel of $\beta$ are known;
see, e.g., \cite{BartelDokchitser16}.

\subsection{An algorithm for the image of $\beta$}
\label{TheAlgorithm}


We describe here an algorithm for computing the image and cokernel of $\beta$ \eqref{TheComparisonMap}.
The end result is Theorem \ref{AlgorithmForImageOfBeta} below.
Establishing the algorithm involves only elementary representation
theory (see \cite{Dr71}\cite{tomDieck79}\cite{Be}\cite{Ker} \cite{Robbin06}\cite{Bo}\cite{LP})
and basic monoidal category theory (see \cite{Mc}\cite{Bor})
but seems to be new.

\medskip
Throughout, $G$ is a finite group and $k$ is a field.

\begin{defn}[Category of $G$-sets]
  \label{GSets}
We write $G \mathrm{Set}^{\mathrm{fin}}$ for the category of finite sets equipped with $G$-action, called \emph{$G$-sets}, for short.
\end{defn}
This is a symmetric monoidal category (\cite[vol 2, 6.1]{Bor}) with respect to Cartesian product of $G$-sets, which is given by the plain Cartesian product
of underlying sets, equipped with the diagonal $G$-action.
For example, the underlying set of the group $G$ becomes a $G$-set by the left multiplication action of $G$ on itself.
More generally, for $H \subset G$ any subgroup, the set $G/H$ of cosets is still a $G$-set by the left action of $G$ on itself.
The elements of $G/H$ are equivalence classes of elements $g$ of $G$, often denoted $g H$, for which we will write
$$
  [g] := gH
  \,.
$$
Generally, we use square brackets to indicate the equivalence classes or isomorphism classes. In particular we write
$[G/H]$ for the isomorphism class of the $G$-set $G/H$ as an object of $G \mathrm{Set}$.

\begin{defn}[Category of $k$-linear $G$-representations]
 \label{GReps}
We write $G \mathrm{Rep}^{\mathrm{fin}}_k$ for the category of finite dimensional $k$-linear $G$-representations.
\end{defn}
This is a symmetric monoidal category (\cite[vol 2, 6.1]{Bor}) with respect to the standard tensor product of representations, which we denote simply by ``$\otimes$''.

\begin{defn}[The trivial irrep]
  \label{TheTrivialIrreducibleRepresentation}
  We write
  $$\
    \mathbf{1} \simeq k[\ast] \in G \mathrm{Rep}^{\mathrm{fin}}_k
  $$
  for the trivial 1-dimensional $G$-representation \eqref{trivialRepresentation}, equivalently the permutation representation \eqref{PermutationRepresentation} of the singleton $G$-set.
  This is the \emph{tensor unit} for the tensor monoidal structure on $G \mathrm{Rep}_k$:
For $V \in G \mathrm{Rep}_k$ any representation, the hom-space out of $\mathbf{1}$ into $V$ is the vector space $V^G$ of \emph{$G$-invariants}
in $V$, hence of elements which are fixed by $G$:
\[
  \label{HomOutOfTensorUnitIsFixedSpace}
  V^G \;\simeq\; \mathrm{Hom}(\mathbf{1}, V) \;\in\; \mathrm{Vect}_k
  \,.
\]
\end{defn}
\begin{defn}[The irreducible $G$-sets]
  \label{IrreducibleGSet}
  The action of $G$ on a $G$-set $S$ is called \emph{transitive} if for all pairs of elements $s_1, s_2 \in S$
  there exists a group element that takes them into each other: $g s_2 = g_2$.
\end{defn}
Every transitive $G$-set $S$ is isomorphic to a set of cosets $G/H$, equipped with the canonical $G$-action
induced from the left action of $G$ on itsef, where $H$ is isomorphic to the stabilizer subgroup $\mathrm{Stab}(s) \subset G$ of $s$ in $S$.
Two such $G$-sets of cosets are isomorphic, $G/H_1 \simeq G/H_2$, precisely if $H_1$ and $H_2$ are conjugate to each other, as subgroups of
$G$. If we denote isomorphism classes of $G$-sets by square brackets, and also denote conjugacy classes of subgroups by square brackets, then
this means that
$$
  [G/H_1] = [G/H_2]
  \phantom{AA}
  \Longleftrightarrow
  \phantom{AA}
  [H_1] = [H_2]
  \,.
$$
Consequently, we have the following.
\begin{prop}[Canonical linear basis for Burnside ring]
  \label{CaninicalLinearBasisForBurnsideRing}
Every finite $G$-set is a disjoint union of such transitive $G$-sets (Def. \ref{IrreducibleGSet}). Hence the abelian group underlying the Burnside ring
is the free abelian group on elements $[G/H]$, one for each conjugacy class $[H]$ of subgroups of $H$:
$$
  \xymatrix@R=-4pt{
    \underset{[H] \atop {H \subset G}}{\oplus} \mathbb{Z}[H]
    \ar[r]^-{\simeq}
    &
    A(G)
    \\
    {\phantom{AAA}}\; [H] \ar@{|->}[r] & [G/H]
    \,.
  }
$$
\end{prop}
We would like to get a handle on the following object:
\begin{defn}[Multiplicities multiplication table of the Burnside ring]
  \label{MultiplicationTable}
  Let $\big\{ [H_i] \big\}_{i}$ be an indexing of the set of conjugacy classes $[H]$ of subgroups $H \subset G$.

  \vspace{-2mm}
  \begin{enumerate}[{\bf (i)}]
  \item The \emph{structure constants} of the Burnside ring $A(G)$ is the set of natural numbers $\{n_{i j}^\ell\}$ defined by
  \begin{equation}
    \label{BurnsideMultiplicationTable}
    k[G/H_i] \otimes k[G/H_j]
    \;\simeq\;
    \underset{\ell}{\oplus} \, n_{i j}^\ell \, k[G/H_\ell]
    \,.
  \end{equation}

  \vspace{-5mm}
 \item The \emph{total multiplicities table} of the Burnside ring $A(G)$ is the quadratic matrix whose $(i,j)$-entry
  is
  \begin{equation}
    \label{BurnsideMultiplicationMultiplicities}
    M_{i j}
    \;:=\;
    \underset{\ell}{\sum} n_{i j}^\ell
    \,.
  \end{equation}
\end{enumerate}
\end{defn}

\vspace{-2mm}
Before discussing the crucial role of the total multiplicities
\eqref{BurnsideMultiplicationMultiplicities} for our purpose,
(to which we come in Prop. \ref{TableViaHomSpaces} below) we
first record an efficient way of computing them:

\begin{defn}[e.g {\cite[Def. 1.1]{Pfeiffer97}}]
  \label{TableOfMarks}
  Given a finite group $G$, its \emph{table of marks}
  is the square matrix $m$ indexed by the conjugacy classes $[H]$ of
  subgroups $H \subset G$ whose $[H_i], [H_j]$-entry is
  the number of fixed points of the $H_j$-action on $G/H_i$
  $$
    m_{i j}
    \;:=\;
    \left\vert
      \big( G/H_i\big)^{H_j}
    \right\vert
    \;\in\;
    \mathbb{Z}
    \,.
  $$
\end{defn}
\begin{prop}
[Ordering]
  \label{TableOfMarksIsLowerTriangular}
  There exists a linear ordering $\leq$ of the
  set of conjugacy classes of subgroups of $G$
  which extends the inclusion relation of subgroups, in that
  $$
    \big(
      H_i \subset H_j
    \big)
    \;\Rightarrow\;
    \big(
      [H_i] \leq [H_j]
    \big)
    \,.
  $$
  With respect to any such linear ordering, the table of marks $m$
  (Def. \ref{TableOfMarks}) is a lower triangular matrix
  with positive entries on its diagonal, hence in particular
  an invertible matrix.
\end{prop}
\begin{proof}
  First of all, inclusion of subgroups defines a partial order
  and every partial order extends to a linear order.
  (For our finite ordered set this follows immediately by induction,
  splitting off a minimal element in each step;
  more generally see \cite{Marczewski30}.)
  Then, observe that $H_j$ having any fixed points on $G/H_i$
  means that it is conjugate to a subgroup of the stabilizer group
  of $[e] \in G/H_i$. But the latter is manifestly $H_i$ itself.
  Hence
  $$
    \Big(
      M_{i j}
      =
      \big\vert
        \left(G/H_i\right)^{G_j}
      \big\vert
      \;\gt\;
      0
    \Big)
    \;\Rightarrow\;
    \big(
      [H_j] \leq [H_i]
    \big)
    \,,
  $$
  which says that $m$ is lower triangular.  Finally, it is clear that at least
  $[e] \in G/H_i$ is fixed by $H_i$, hence that the diagonal entries are positive.
\end{proof}

\begin{prop}
[Multiplicities explicitly]
  The Burnside multiplicities $m_{i j}^\ell$ (Def. \ref{MultiplicationTable})
  are given by the following
  algebraic expression in terms of the entries of the table of marks $m$
  (Def. \ref{TableOfMarks}) and its inverse matrix $m^{-1}$ (from Prop. \ref{TableOfMarksIsLowerTriangular}):
  \begin{equation}
    \label{BurnsideMultiplicitiesViaTableOfMarks}
    n_{i j}^\ell
    \;=\;
    \underset{k}{\sum}
    m_{i k }
      \cdot
    m_{j k}
      \cdot
    (m^{-1})^{k \ell}
    \,.
  \end{equation}
\end{prop}
\begin{proof}
  Notice that the entry $m_{i j}$ of the table of marks
  may equivalently be thought of as the cardinality of the
  set of homomorphism from the $G$-set $G/H_j$ to the $G$-set $G/H_i$:
  $$
    m_{i j}
    \;=\;
    \left\vert
      \ \mathrm{Hom}_{G \mathrm{Set}}
      \big(
        G/H_j , G/H_i
      \big)
    \right\vert
    \,.
  $$
  (Because, by transitivity of the action, any such homomorphism
  is determined by its image of $[e] \in G/H_j$ and by $G$-equivariance
  this has to be sent to any $H_j$-fixed point of $G/H_i$.)
  With this we compute as follows:
  $$
  \begin{aligned}
    \underset{\ell}{\sum}
    n_{i j}^\ell \cdot m_{\ell k}
    & =
    \underset{\ell}{\sum}
    n_{i j}^\ell
    \cdot
    \big\vert
      \mathrm{Hom}_{G \mathrm{Set}}
      \big(
        G/H_k
        ,
        G/H_\ell
      \big)
    \big\vert
    \\
    & =
    \Big\vert
      \mathrm{Hom}_{G \mathrm{Set}}
      \big(
        G/H_k
        ,
        \underset{\ell}{\sum}
        n_{i j}^\ell
        \cdot
        G/H_\ell
      \big)
    \Big\vert
    \\
    & =
    \left\vert
      \mathrm{Hom}_{G \mathrm{Set}}
      \big(
        G/H_k,
        \;
        G/H_i \times G/H_j
      \big)
    \right\vert
    \\
    & =
    \left\vert
      \mathrm{Hom}_{G \mathrm{Set}}
      \big(
       G/H_k,
       \;
       G/H_i
      \big)
    \right\vert
    \cdot
    \left\vert
      \mathrm{Hom}_{G \mathrm{Set}}
      \big(
       G/H_k,
       \;
       G/H_j
      \big)
    \right\vert
    \\
    & =
    m_{i k} \cdot m_{j k}
    \,.
  \end{aligned}
$$
Here in the third step we used the defining equation \eqref{BurnsideMultiplicationTable} of the
Burnside multiplicities $n_{i j}^\ell$,
and otherwise we used evident properties of sets of homomorphisms.
Now matrix multiplication of both sides of this
equation with the inverse matrix $m^{-1}$ yields the claimed
relation.
\end{proof}

In order to understand the meaning of the total multiplicities
\eqref{BurnsideMultiplicationMultiplicities}, we consider now some basic facts, all elementary.

\begin{prop}[Self-duality of permutation representations]
  \label{PermutationRepresentationsAreSelfDual}
  If $\mathrm{char}(k) \neq \vert H \vert $, then the permutation representation
  $k[G/H]$ \eqref{PermutationRepresentation} is a dualizable object in the symmetric monoidal representation category $(G \mathrm{Rep}_k, \otimes)$
  (e.g. \cite[vol 2, 6.1]{Bor})
  and is in fact self-dual:
  $$
    k[G/H]^\ast \simeq k[G/H]
    \,.
  $$
\end{prop}
\begin{proof}
  We need to find morphisms
  $$
    \mathbf{1} \overset{\eta}{\longrightarrow} k[G/H] \otimes k[G/H]
    \phantom{AA}\mbox{and}\phantom{AA}
    k[G/H] \otimes k[G/H] \overset{\epsilon}{\longrightarrow} \mathbf{1}
  $$
  in $G \mathrm{Rep}_k$
  that make the following triangle commutes
  (the ``triangle identity'', recalled in  \cref{CategoricalAlgebra}):
  $$
    \xymatrix@R=1em{
      & k[G/H] \otimes k[G/H] \otimes k[G/H]
      \ar[dr]|{ \mathrm{id} \otimes \epsilon }
      \\
      k[G/H]
      \ar[ur]|{ \eta \otimes \mathrm{id} }
      \ar[rr]_-{\mathrm{id}}
      &&
      k[G/H]
    }
  $$
  where we are notationally suppressing the unitors and associators, as usual.
  With $[g] \in G/H \subset k[G/H]$ denoting both the equivalence class of an
  element $g \in G$ as well as the corresponding basis element of $k[G/H]$,
  we claim that the following choice works:
  $$
    \raisebox{20pt}{
    \xymatrix@R=-4pt{
      \mathbf{1} \ar[rr]^-{\eta} && k[G/H] \otimes k[G/H]
      \\
      1 \ar@{|->}[rr] && \tfrac{1}{ \vert H \vert  }  \underset{g \in G}{\sum} [g] \otimes [g] \phantom{AAAA}
    }}
    \phantom{AA} \mbox{and} \phantom{AA}
    \raisebox{20pt}{
    \xymatrix@R=-4pt{
      k[G/H] \otimes k[G/H] \ar[rr]^-{\epsilon} && \mathbf{1}
      \\
      [g_1] \otimes [g_2]
      \ar@{|->}[rr] &&
      \left\{
      {\begin{array}{ccc}
         1 &\vert& [g_1] = [g_2],
         \\
         0 &\vert& \mbox{otherwise}.
      \end{array}}
      \right.
    }}
  $$
  Here the fraction on the left makes sense by assumption on the characteristic of $k$. Also, it is immediate
  that these linear maps do respect the $G$-action and hence are morphisms in $G \mathrm{Rep}_k$.
  Because with this, we check for every $[\tilde g] \in k[G/H]$ that:

 \vspace{-3mm}
  $$
    \xymatrix{
      [\tilde g]
      \ar@{|->}[rr]^-{ \eta \otimes \mathrm{id} }
      &&
      \frac{1}{{\vert H \vert}} \underset{g \in G}{\sum} [g] \otimes [g] \otimes [\tilde g]
      \ar@{|->}[rr]^-{ \mathrm{id} \otimes \epsilon }
      &&
      \underset{ = 1 }{
      \underbrace{
        \frac{1}{{\vert H \vert}} \underset{ {g \in G} \atop { [g] =[\tilde g] } }{\sum}
      }} [\tilde g]
      \ar@{=}[r]
      &
      [\tilde g]\;.
    }
  $$

  \vspace{-.7cm}
\end{proof}
As a direct consequence we obtain:

\begin{prop}[Internal hom between permutation representations]
  \label{InternalHomBetweenPermutationRepresentations}
  If $\mathrm{char}(k) \neq \vert H_1\vert$, then the \emph{internal hom} between the permutation
  representation $k[G/H_1]$ and $k[G/H_2]$ \eqref{PermutationRepresentation}
  exists in $(G \mathrm{Rep}_k,\otimes)$ and is given by the tensor product of representations:
  $$
    \big[ \, k[G/H_1]\,,\, k[G/H_2] \, \big]
    \;\simeq\;
    k[G/H_1] \otimes k[G/H_2]
    \,.
  $$
\end{prop}
\begin{proof}
  Generally, if duals exist, the internal hom is given by
  $$
    [V_1, V_2] \;\simeq\; V_1^\ast \otimes V_2
    \,.
  $$
  With this, the statement follows by Prop. \ref{PermutationRepresentationsAreSelfDual}.
\end{proof}

\begin{prop}[Trivial irrep in transitive permutation representations]
  \label{TrivialIrreducibleRepresentationInPermutationRepresentation}
  The permutation representation $k[G/H]$ of a transitive $G$-set (Def. \ref{IrreducibleGSet}) contains precisely one direct summand of the
  trivial 1-dimensional representation (Def. \ref{TheTrivialIrreducibleRepresentation}):
  \vspace{-2mm}
  $$
    \mathrm{dim}_k
    \mathrm{Hom}\big(  \mathbf{1}, k[G/H] \big)
    \;=\;
    \mathrm{dim}_k
    \big( k[G/H] \big)^G
    \;=\; 1
    \,.
  $$
\end{prop}
\begin{proof}
  By definition, every element $v \in k [G/H]$ is a formal linear combination of cosets $[g]$:
  $$
    v \;=\;
    \underset{[g] \in G/H}{\sum} v_{[g]} \, [g]
  $$
  for coefficients $v_{[g]} \in k$. If there exists $[g_1], [g_2] \in G/H$ such that $v_{[g_1]} \neq v_{[g_2]}$
  then, by transitivity of the $G$-action, there exist $g \in G$ with $g [g_1] = [g_2]$. But since the $[g]$
  constitute a basis of $k[G/H]$, this implies that $g v \neq v$, hence that $v$ is not $G$-invariant.
  Therefore, the only $G$-invariant vectors $v \in k[G/H]$ are those all whose coefficients $v_{[g]}$ agree.
  These clearly form a 1-dimensional subspace.
\end{proof}

As a corollary we obtain:

\begin{prop}[Multiplication table via invariants and via hom-spaces]
 \label{TableViaHomSpaces}
 The entries $M_{i j} := {\sum}_{\ell}\, n_{i j }^\ell $ in the table of multiplication
 multiplicities in the Burnside ring (Def. \ref{MultiplicationTable}) are equivalently

\begin{enumerate}[{\bf (i)}]
\vspace{-3mm}
\item the dimensions of the subspaces of $G$-invariants in the Burnside products;

\vspace{-3mm}
\item  the dimensions of the external homs
 of the two given basis elements;
\end{enumerate}
 $$
   M_{i j}
   \;:=\;
   \underset{\ell}{\sum} n_{i j}^\ell
   \;=\;
   \mathrm{dim}_k \big( k[G/H_i] \otimes k[G/H_j]  \big)^G
   \;=\;
   \mathrm{dim}_k \mathrm{Hom}\big( k[G/H_i] \,,\, k[G/H_j]  \big)
   \,.
 $$

\end{prop}

\begin{proof}
  The first equality follows from Prop. \ref{TrivialIrreducibleRepresentationInPermutationRepresentation} applied to
  the Definition \ref{MultiplicationTable} of the structure constants, which gives the following isomorphism
  $$
    \begin{aligned}
       \big( k[G/H_i] \otimes k[G/H_j]  \big)^G
       & :=
       \Big( \underset{\ell}{\oplus} n_{i j}^\ell \, k[G/H_\ell]  \Big)^G
       \\
       & \simeq \underset{\ell}{\oplus} n_{i j}^\ell \big( k[G/H_\ell]  \big)^G
       \\
       & \simeq \underset{\ell}{\oplus} n_{i j}^\ell \, k\;.
    \end{aligned}
  $$
 The second equality comes from the following sequence of isomorphisms
  $$
    \begin{aligned}
      \mathrm{Hom}\Big( k[G/H_1] , k[G/H_2] \Big)
      & \simeq \mathrm{Hom}\Big( \mathbf{1}, \big[ k[G/H_1] , k[G/H_2] \big] \Big)
      \\
      & \simeq \Big(  \big[ k[G/H_1] , k[G/H_2] \big)  \Big)^G
      \\
      & \simeq \big( k[G/H_1] \otimes k[G/H_2]  \big)^G
      \,.
    \end{aligned}
  $$
  Here the first equivalence expresses a general relation between external and internal homs,
  via the tensor unit (Def. \ref{TheTrivialIrreducibleRepresentation}),
  the second is from \eqref{HomOutOfTensorUnitIsFixedSpace} and the last one is Prop. \ref{InternalHomBetweenPermutationRepresentations}.
\end{proof}

Now it is useful to relate this to Schur's Lemma. For this purpose it turns out to be convenient to
think in terms of the following inner product.
\begin{defn}[Inner product on Burnside ring]
  \label{BurnsideInnerProduct}
  For $V_1, V_2 \in G \mathrm{Rep}_k$, write
  $$
    \big\langle V_1, V_2  \big\rangle
    \;:=\;
    \mathrm{dim}_k \mathrm{Hom}\big( V_1, V_2  \big)
    \;\in\;
    \mathbb{N}
  $$
  for the dimension of the vector space of representation homomorphism between them.
  By $\mathbb{Z}$-linearity this extends to a $\mathbb{Z}$-valued pairing on the Burnside ring:
  \vspace{-2mm}
  $$
    \xymatrix{
      R_k(G) \times R_k(G)
      \ar[rr]^-{ \langle -,-\rangle }
      &&
      \mathbb{Z}
    }.
  $$
\end{defn}
In terms of this pairing, \emph{Schur's Lemma} says the following:
\begin{lemma}[Schur's Lemma]
  \label{Schur}
  The pairing
  $\langle -,-\rangle$ from Def. \ref{BurnsideInnerProduct}
  is a symmetric and $\mathbb{Z}$-bilinear inner product on the abelian group underlying the representation ring
  $R_k(G)$.
  With respect to this inner product, the set of isomorphism classes $\rho_i$ of irreducible representations of $G$
  \begin{enumerate}[{\bf (i)}]
  \vspace{-2mm}
    \item is always an \emph{orthogonal basis}, where each basis element has positive norm-square;
    \vspace{-2mm}
    \item
     is even an \emph{orthonormal basis} if the field $k$ is algebraically closed.
  \end{enumerate}
\end{lemma}
Using this perspective, we amplify the following:
\begin{remark}
[Matching of multiplicities and independence]
  \label{MultiplicationTableIsInnerproductTable}
  The statement of Prop. \ref{TableViaHomSpaces} is that,
  in terms of the inner product $\langle -,-\rangle$ from Def. \ref{BurnsideInnerProduct}, the multiplicities multiplication table coincides
  with the table of inner products of Burnside-basis elements:
  \begin{equation}
    \label{BurnsideMultiplicitiesAsInnerProducts}
    M_{i j}
    \;:=\;
    \underset{\ell}{\sum} n_{i j}^\ell
    \;=\;
    \big\langle k[G/H_i]\,,\, k[G/H_j]  \big\rangle
    \,.
  \end{equation}
  Also notice that Prop. \ref{TableViaHomSpaces} implies that the $k$-linear dimension of $k$-linear hom-spaces between
  $k$-linear permutation  representations of transitive $G$-sets is \emph{independent} of the ground field $k$:
  $$
    \left\langle
      k[G/H_i] \,,\, k[G/H_j]
    \right\rangle
    \;:=\;
    \mathrm{dim}_k
    \mathrm{Hom}
    \left(
      k[G/H_i] \,,\, k[G/H_j]
    \right)
    \;=\;
    M_{i j}
  $$
  since the multiplication multiplicities matrix $M$ of the Burnside ring is manifestly independent of $k$.
\end{remark}

\medskip

We now use this to discuss explicit matrix representations of $\beta$.

\begin{defn}[Upper triangular form of the Burnside multiplicities matrix]
 \label{HermiteNormalFormOfBurnsideMultiplication}
For $G$ a finite group, let
\begin{equation}
  \label{Hermite}
  H := U \cdot M \;\;\in \mathrm{Mat}_{N \times N}(\mathbb{N})
\end{equation}
be an integral upper triangular form (e.g. \cite[p. 3,4]{GilbertPathria90})
of the Burnside multiplication multiplicities matrix (Def. \ref{MultiplicationTable}),
hence with
$
  U \in \mathrm{GL}(N,\mathbb{Z})
$
an invertible matrix whose left-multiplication implements row reduction on $M$
($N$ the number of conjugacy classes of subgroups of $G$).
Write
\begin{equation}
  \label{ZeroRowsDeleted}
  \tilde H := \tilde U \cdot H
\end{equation}
for the result of deleting the zero-rows from \eqref{Hermite}.
Write
\begin{equation}
  \label{transformed}
  V_i
  \;:=\;
  \underset{\ell}{\sum}
  \tilde U_{i}^\ell
  \cdot
  [G/H_\ell]
  \;\in\;
  A(G)
\end{equation}
for the corresponding permutation representations and
write
\begin{equation}
  \label{SchurDimension}
  d_i
    \;:=\;
  \big\langle
    V_i, V_j
  \big\rangle
\end{equation}
for their norm-square with respect to the inner product (Def. \ref{BurnsideInnerProduct}).
\end{defn}

\begin{prop}
[Matrix for $\beta$]
   \label{MatrixBeta}
   Consider the upper triangular form $\tilde H$ of the Burnside multiplication matrix, Def. \ref{HermiteNormalFormOfBurnsideMultiplication}.
 \item {\bf (i)}  Then the corresponding permutation representations $V_i \in R_k(G)$ \eqref{transformed} are orthogonal, in that their inner products (Def. \ref{BurnsideInnerProduct})
   satisfy
   \begin{equation}
     \label{Orthogonal}
      \big\langle V_i \,,\, V_j\big\rangle
      \;=\;
      \delta_{i j} \, d_i
      \phantom{AAA}
      d_i \in \mathbb{N}
      \,;
   \end{equation}
   and they linearly span the image of $\beta$:
   $
     \left\langle
       V_i
     \right\rangle_i
     \;\simeq\;
     \mathrm{im}(\beta) \subset R_k(G)
       $.
\item {\bf (ii)}  Specifically, the matrix that represents $\beta$ with respect to the basis of
   the $G/H_j \in A(G)$ and the basis of $V_i \in \mathrm{Im}(\beta) \subset R_k(G)$  is
       $$
     \beta_{i j} \;=\; \tfrac{1}{d_i} \tilde H_{i j}
     \,.
   $$
   (In particular this means that the $i$th row $\tilde H_{i \bullet}$ of $\tilde H$ is divisible by $d_i$.)
   \item {\bf (iii)}
Hence for every subgroup $H_j \subset G$, the image of $\beta$ on the corresponding $G$-set $G/H_j$ is
   the following linear combination of the representations $V_i$ \eqref{transformed}, from Def. \ref{HermiteNormalFormOfBurnsideMultiplication}:
   \begin{equation}
     \label{BetaInBases}
     \beta\big( G/H_j  \big)
     \;=\;
     \underset{i}{\sum} \tfrac{1}{d_i} \tilde H_{i j} \, V_i
     \,.
   \end{equation}
\end{prop}
\begin{proof}
The first statement follows from \cite{PursellTrimble91}: By \eqref{BurnsideMultiplicitiesAsInnerProducts} we have that $M = A^T \cdot A$ is a positive semi-definite matrix of inner products, where $A$ is the
matrix whose columns are the permutation representations $k[G/H]$
expanded in terms of the irreps of $G$.
By \cite[top of p. 5]{PursellTrimble91} this implies that the same row reduction which turns $M$ into upper-triangular form takes $A$ to
a matrix whose non-vanishing columns $V_i$ constitute an orthogonal basis
of the linear span of the $k[G/H]$.

The remaining statement just spells this out by immediate computation:
$$
  \begin{aligned}
    \tilde H_{i j}
    & =
    \underset{\ell}{\sum} \tilde U_i^\ell M_{\ell j}
    \\
    & =
    \underset{\ell}{\sum} \tilde U_i^\ell \left\langle k[G/H_\ell]  , k[G/H_j] \right\rangle
    \\
    & =
     \Big\langle \underset{\ell}{\sum} \tilde U_i^\ell  k[G/H_\ell]  , k[G/H_j] \Big\rangle
    \\
    & =
    \big\langle V_i, k[G/H_j] \big\rangle
    \\
    & =
    \left\langle V_i, \beta(G/H_j)\right\rangle
    \\
    & = d_i \beta_{i j}
    \,.
  \end{aligned}
$$
Here the first line is by Def. \ref{HermiteNormalFormOfBurnsideMultiplication},
and in the second step we used Prop. \ref{InternalHomBetweenPermutationRepresentations}
in the inner product notation from Def. \ref{BurnsideInnerProduct} (as in Remark \ref{MultiplicationTableIsInnerproductTable}).
In the third step
we used the linearity of the inner product from Prop. \ref{Schur},
in the fourth step we inserted the definition of $V_i$ \eqref{transformed}. In the fifth step we just identified
$\beta(H_j) = k[G/H_j]$, for emphasis. Finally we used the assumption \eqref{Orthogonal} to identify the
coefficient $\beta_{ij}$ of $V_i$ in $\beta(H_j)$.
\end{proof}

With a linear basis for the image of $\beta$ thus in hand, it just remains to express it in a form that
may directly be compared to standard classifications available from the linear representation theory of finite groups.
For completeness, recall:
\begin{defn}[Characters]
  \label{Character}
  For $G$ a finite group and $k$ a field, a function from the underlying set of $G$ to $k$
  is called a \emph{class function} if it is constant on conjugacy classes of $G$, hence if it
  factors as
  $$
    \xymatrix@R=1em{
      G \ar@{->>}[dr] \ar[rr] && k\;.
      \\
      & \mathrm{ConjCl}(G) \ar[ur]_-{\phi}
    }
  $$
  Hence class functions form a $k$-vector space of dimension the number of conjugacy classes in  $G$:
  \begin{equation}
    \label{SpaceOfClassFunctions}
    k^{{\vert \mathrm{ConjCl}(G)\vert}}.
  \end{equation}
  For $V \in R_k(G)$ a representation, the map sending any $g \in G$ to its \emph{trace}, when regarded
  as a linear map $V \overset{g}{\to} V$ via this representation, is a class function (by basic properties of the trace),
  called the \emph{character} $\chi_V$ of $V$:
  $$
    \xymatrix@R=-2pt{
      \mathrm{ConjCl}(G)
      \ar[r]
      &
      k
      \\
      [g] \ar@{|->}[r] & \mathrm{tr}_V(g)
    }
  $$
\end{defn}
The following example is immediate but important:
\begin{example}[Character of permutation representation]
  \label{CharacterOfPermutationRepresentation}
  For $S \in A(G)$ a finite $G$-set, the character (Def. \ref{Character}) of its
  permutation representation $k[S]$ \eqref{PermutationRepresentation} is the function that sends $g$ to the
  number of elements in $S$ that are fixed by the given action of $g$:
  $$
    \chi_{ k[S] }
    \;:\;
    [g]
    \longmapsto \vert S^g\vert
    \in \mathbb{N} \subset k
    \,.
  $$
\end{example}

The relevance of characters is that, in characteristic zero, they already completely characterize linear representations,
while being more manifestly tractable:
\begin{prop}[e.g. {\cite[Theorem 2.2.5]{tomDieck09}}]
  \label{CharactersInCharacteristicZeroRecognizeRepresentations}
  If the field $k$ is of characteristic zero,
  then the map that sends a $k$-linear $G$-representation
  to its character (Def. \ref{Character}) is an injection of the $k$-vector space
  of isomorphism classes of finite-dimensional $G$-representations into the vector space \eqref{SpaceOfClassFunctions} of class functions (Def. \ref{Character})
  $$
    \xymatrix@R=-2pt{
      G \mathrm{Rep}_k/_{\sim}
      \; \ar@{^{(}->}[rr]
      &&
      k^{{}^{ {\vert \mathrm{ConjCl}(G)\vert} }}
      \\
      V \ar@{|->}[rr] && \chi_V
    }
  $$
  If $k$ is, in addition, a splitting field for $G$ (notably if $k =\mathbb{C}$ is the complex numbers), then this map is even an isomorphism.
\end{prop}
As usual, it is convenient to organize this data in \emph{character tables}. In order to make our list of examples
in  \cref{ImageOfBetaExamples} be unambiguously intelligible, we briefly dwell on the notation for character
tables.
\begin{defn}[Character table]
  \label{CharacterTable}
  For $(W_i \in R_k(G))_{i \in \{1, \cdots, \}} $ a tuple of (possibly virtual) $k$-linear representations of a finite group $G$, their   \emph{character table} is the $n \times {\vert \mathrm{ConjCl}(G)\vert}$-matrix with values in $k$
   whose $(i,j)$-th entry is the value $\chi_{{}_{W_i}}(g_j)$ of the character $\chi_{{}_{W_i}}$ (Def. \ref{Character})
   on any element $g_j$ of the $j$th conjugacy class $[g_j] \in \mathrm{ConjCl}(G)$.
\end{defn}

\begin{example}[Irreducible character table over $\mathbb{C}$]
  \label{IrreducibleCharacterTableOverC}
By Prop. \ref{CharactersInCharacteristicZeroRecognizeRepresentations} the characters of irreducible representations over $k = \mathbb{C}$ the complex numbers
form a linear basis of the representation ring $R_{\mathbb{C}}(G)$. We will denote these irreducible representations by $(\rho_i \in R_{\mathbb{C}})(G)$
and will display the corresponding character table (Def. \ref{CharacterTable}) as follows (conjucagy classes being labeled by the order of their elements):
\begin{center}
\begin{tabular}{cc|ccccc}
 \multicolumn{2}{c}{   } &  \multicolumn{5}{c}{\bf conjugacy class }
 \\
 \multicolumn{2}{c|}{  } & 1 & 3 & 4A & 4B & $\cdots$
 \\
 \cline{2-7}
 \multirow{5}{*}{ \raisebox{-50pt}{\bf \begin{rotate}{90} irred. repr. \end{rotate}} \hspace{-20pt} }
 &
 $\rho_1$ & $\cdot$ & $\cdot$ & $\cdot$ & $\cdot$ &
 \\
 &
 $\rho_2$ & $\cdot$ & $\chi_{{}_{\rho_2}}(g_2)$ & $\chi_{{}_{\rho_2}}(g_3)$ & $\cdot$ &
 \\
 &
 $\rho_3$ & $\cdot$ & $\chi_{{}_{\rho_3}}(g_2)$ & $\chi_{{}_{\rho_3}}(g_3)$ & $\cdot$ &
 \\
 &
 $\rho_4$ & $\cdot$ & $\cdot$ & $\cdot$ & $\cdot$ &
 \\
 &
 $\vdots$ &  &  &  &  & $\ddots$
\end{tabular}
\end{center}
The character tables of irreducible representations over the complex numbers,
for many finite groups of small order, have been tabulated in the literature,
for instance in \cite{DokGroupNames}.
\end{example}

\begin{example}[Irreducible character table over $\mathbb{R}$ ]
 \label{IrreducibleRealCharacterTable}
For $k \subset \mathbb{C}$ a subfield, under the ring homomorphism of ``extension of scalars''
$$
  \xymatrix{
    R_{k}(G)
    \ar[rr]^-{ (-)_{k} \mathbb{C} }
    &&
    R_{\mathbb{C}}(G)
  },
$$
the values of characters to not change. Hence if $k$ is of characteristic zero, $\mathbb{Q} \subset k \subset \mathbb{C}$,
then, by Prop. \ref{CharactersInCharacteristicZeroRecognizeRepresentations}, we may equivalently express the character
of any linear representation $W \in R_k(G)$ over $k$ after tensoring it with $\mathbb{C}$. This is a linear combination of the complex irreducible characters $\chi_{{}_{\rho_i}}$ from above
$$
  W \otimes_k \mathbb{C} \;=\; \underset{i}{\sum}\, \underset{ \in \mathbb{N} }{\underbrace{w_i}} \cdot \rho_i\;,
\qquad \qquad
  \chi_{{}_{(W \otimes_k \mathbb{C}})}  \;=\; \underset{i}{\sum} w_i \cdot \chi_{{}_{\rho_i}}
  \,.
$$
Therefore, when $\mathbb{Q} \subset k \subset \mathbb{C}$ and for $(\cdots, W \in R_k(G), \cdots)$ a tuple
of $k$-linear representations, we may and will express the corresponding character table as a table of linear
combinations of the irreducible complex characters:
\begin{center}
\begin{tabular}{cc|ccccc}
 \multicolumn{2}{c}{   } &  \multicolumn{5}{c}{ \bf conjugacy class }
 \\
 \multicolumn{2}{c|}{  } & 1 & 3 & 4A & 4B & $\cdots$
 \\
 \cline{2-7}
 \multirow{5}{*}{ \raisebox{-50pt}{\bf \begin{rotate}{90} irred. repr. \end{rotate}} \hspace{-20pt} }
 &
 $\cdot$ & $\cdot$ & $\cdot$ & $\cdot$ & $\cdot$ &
 \\
 &
 $W \otimes_k \mathbb{C} = \underset{i}{\sum} w_{i} \cdot \rho_i $ & $\cdot$ & $\underset{i}{\sum} w_i \cdot \chi_{{}_{\rho_i}}(g_2)$ & $\underset{i}{\sum} w_i\cdot \chi_{{}_{\rho_i}}(g_3)$ & $\cdot$ &
 \\
 $\cdot$ & $\cdot$ & $\cdot$ & $\cdot$ & $\cdot$ & $\cdot$
 \\
 &
 $\vdots$ &  &  &  &  & $\ddots$
\end{tabular}
\end{center}
In this fashion we will in particular state the irreducible character tables over $k =\mathbb{R}$ the real numbers,
which may again be found in the literature for many finite groups of small order.
\end{example}
Specifically, the complex character tables available in the literature (e.g. \cite{DokGroupNames}) list the \emph{type} of the
corresponding complex representation, from which the character table of irreducible representations
over the real numbers may be extracted (or conversely, as in \cite{MonRepresentations}), via the following basic fact:

\begin{prop}[e.g. {\cite[p. 4]{Robbin06}}]
  Let $G$ be a finite group, and consider the complexification map on the representation ring
  $R_{\mathbb{R}}(G) \xrightarrow{ (-) \otimes_{\mathbb{R}} \mathbb{C} } R_{\mathbb{C}}(G)$.
  Then every irreducible complex representation $V \in R_{\mathbb{C}}(G)$
  is of exactly one of the following three \emph{types}, depending on how it arises as a direct summand
  of an irreducible real representation $W \in R_{\mathbb{R}}(G)$:
  $$
    W \otimes_{\mathbb{R}} \mathbb{C}
    \;\simeq\;
    \left\{
    \begin{array}{cccccccc}
      V\phantom{ \oplus V^\ast} &\vert& \mbox{ {\it real type} } &\mbox{/}& \mbox{{\it orthogonal}  }
      \\
      V \oplus V^\ast &\vert& \mbox{{\it complex type}}
      \\
      V \oplus V\phantom{^{\ast}} &\vert& \mbox{{\it quaternionic type}} &\mbox{/}& \mbox{ \it symplectic }
    \end{array}
    \right.
  $$
\end{prop}

In this fashion we may now identify the image of $\beta$ via the character table of its basis elements:
\begin{prop}[Character table of linear basis of image of $\beta$]
  \label{CharacterTableForImageOfBeta}
  The character (Def. \ref{Character}) of a basis element $V_i$ \eqref{transformed} of the image of $\beta$ (Prop. \ref{BetaInBases})
  is the class function given by
  \begin{equation}
    \label{CharacterTableOfLinearBasisOfBetaImage}
    \chi_{ {}_{V_i} }
    \;:\;
    [g]
    \longmapsto
    \underset{\ell}{\sum}
      \tilde U_i^\ell
      \cdot
      \vert (G/H_\ell)^g\vert
    \,,
  \end{equation}

  \vspace{-3mm}
\noindent  where on the right we have the sum over conjugacy classes $H_\ell$
of subgroups of $G$ of the product of the entries of the
  base change matrix from \eqref{ZeroRowsDeleted} with the number of elements
  in $G/H_\ell$ that are fixed by the action of $g$.
\end{prop}
\begin{proof}
  This follows by Example \ref{CharacterOfPermutationRepresentation}.
\end{proof}
Hence, in conclusion, we have the following.
\begin{prop}[Recognizing surjectivity of $\beta$]
  \label{RecognizingSurjectivityOfBeta}
  For $k$ of characteristic zero, let $(\rho_i \in R_k(G))$ be the irreducible $k$-linear representations, spanning
  the representation ring $R_k(G)$. Then the comparison morphism \eqref{TheComparisonMap} is surjective
  precisely if the corresponding tuple of characters $\chi_{{}_{\rho_i}}$ (Def. \ref{Character}) is related to the
  set of characters ${\chi}_{{}_{V_i}}$ in Def. \ref{CharacterTableForImageOfBeta} by an invertible integer matrix
  $$
    \mbox{$\beta$ is surjective over $k$}
    \;\;\;
    \Longleftrightarrow
    \;\;\;
      \chi_{{}_{\rho_i}}
    \;=\;
    \underset{j}{\sum}
    T_i^j
    \cdot
    \chi_{{}_{V_j}}\;,
    \;\;\;
    T \in \mathrm{GL}(N,\mathbb{Z})
    \,.
  $$
\end{prop}
\vspace{-4mm}
\begin{proof}
  Since, by construction, the $V_j$ \eqref{transformed} are linearly independent  and span the image of $\beta$ (Prop. \ref{MatrixBeta})
  and since the $\rho_i$ span $R_k(G)$, the number of the $V_j$ is smaller or equal to the number of $\rho_i$, hence
  the number must be equal if $\beta$ is surjective. This means that for surjectivity there must be an invertible
  integer matrix relating the $(V_j)$ to the $(\rho_i)$. But by Prop. \ref{CharactersInCharacteristicZeroRecognizeRepresentations} this
  is the case precisely if there is such a matrix relating the characters of these representations.
\end{proof}

This concludes our algorithmic description of the image of $\beta$. To summarize, the algorithm proceeds as follows. Here
\begin{itemize}
\vspace{-3mm}
\item $N := {\vert \mathrm{ConjCl}(G)\vert}$ is the number of conjugacy classes of $G$;
\vspace{-3mm}
\item $r := \mathrm{rank}(\mathrm{image}(\beta))$ is the rank of the image of $\beta$ (the number of $V_j$);
\vspace{-3mm}
\item $n$ is the number of isomorphism classes of irreducible representations $\rho_i$.
\end{itemize}

\begin{theorem}[{\bf Algorithm for the cokernel of $\beta$}]
\label{AlgorithmForImageOfBeta}

\noindent Let $G$ be a finite group and $k$ a field of characteristic zero.

\vspace{-.3cm}
\begin{enumerate}[{\bf (1)}]
\item Extract from standard literature:
\vspace{-2mm}
\begin{itemize}
  \item
    the \emph{character table of irreducible $k$-linear representations} (Examples \ref{IrreducibleCharacterTableOverC}, \ref{IrreducibleRealCharacterTable})
    $$
      \big(
        \chi_{{}_{\rho_i}}([g])
      \big)
      \;\;\;
      \in
      \mathrm{Mat}_{n,N}\left( k \right)
      \,.
    $$
\end{itemize}
\vspace{-5mm}
\item Compute:
\vspace{-3mm}
\begin{enumerate}[{\bf (a)}]
  \item the \emph{multiplication table} $\big( n^\ell_{ i j}\big)$ \eqref{BurnsideMultiplicationTable} of the Burnside ring,
      efficiently so via \eqref{BurnsideMultiplicitiesViaTableOfMarks};
  \item the resulting \emph{table of total multiplication multiplicities} $\left(M_{i j}\right) := \
  \big( \underset{\ell}{\sum} n_{i j}^\ell \big)$ \eqref{BurnsideMultiplicationMultiplicities};
  \vspace{-2mm}
  \item its upper triangular form $H := U \cdot M$ \eqref{Hermite};
  \item the result $\tilde H = \tilde U \cdot M$ \eqref{ZeroRowsDeleted} of deleting its zero-rows;
  \item the \emph{character table of the resulting linear basis for the image of $\beta$} \eqref{CharacterTableOfLinearBasisOfBetaImage}
    $$
      \Big(
        \chi_{{}_{V_i}}([g])
        \;=\;
        \underset{\ell}{\sum}
        \tilde U_i^\ell
        \cdot
        \vert (G/H_\ell)^g\vert
      \Big)
      \;\;\in\;
      \mathrm{Mat}_{r, N}(\mathbb{N})
      \subset
      \mathrm{Mat}_{r, N}\left( k \right).
    $$
\end{enumerate}

\vspace{-5mm}
\item Read off the quotient of the lattice spanned by the vectors $\chi_{{}_{\rho_i}}$ by that spanned by the
vectors $\chi_{{}_{V_j}}$:
$$
  \mathrm{coker}(\beta_k)
  \;:=\;
  \frac{ \mathbb{Z}[\chi_{\rho_i} ]_{i = 1}^n }{ \mathbb{Z}[ \chi_{V_j} ]_{j =1}^{N} }
  \,.
$$
\end{enumerate}
\end{theorem}

\begin{remark}[Simplification in examples]
{\bf (i)} In all examples that we compute in \cref{ImageOfBetaExamples},
upper triangular form of the multiplicities matrix in the third step of the algorithm
of Theorem \ref{AlgorithmForImageOfBeta} is achieved by the most straightforward
row reduction, where, incrementally in the row number, a suitable integer multiple
of each row is subtracted from all those beneath it.
\item {\bf (ii)} This is remarkable, since,
in general, row reduction over the integers needs more and more intricate steps
than this; see e.g. \cite[p. 3--4]{GilbertPathria90}. That this happens is due to the
fact that $M$ is a positive semidefinite matrix, as explained in
\cite{PursellTrimble91}.
\item {\bf (iii)} Furthermore, in each case the resulting rows
$V_i$ happen to be actual representations, as opposed to virtual representations.
This makes our algorithm very efficient, and makes it easy to read off the image of $\beta$ in each example.
It seems clear that this particularly nice behavior of row reduction on the Burnside multiplicities matrix is due to
a very special property of the latter. It would be interesting to understand this phenomenon theoretically.
\end{remark}

\section{The image of $\beta$ -- Examples}
 \label{ImageOfBetaExamples}

Given a finite group $G$ and its irreducible characters over a given field $k$ of
characteristic zero, Theorem \ref{AlgorithmForImageOfBeta} provides an effective algorithm for identifying the image of
$A(G) \xrightarrow{\;\beta\;} R_k(G)$ \eqref{TheComparisonMap} and checking whether $\beta$ is surjective.
We have implemented this algorithm in {\tt Python},
available as an ancillary file.
Here we spell out various example computations. In summary, we obtain the result shown in Theorem \ref{CokernelOfBetaInVariousExamples}.
In particular, the \colorbox{lightgray}{shaded entries}
show that over the real numbers $\beta$ has vanishing
cokernel/is surjective onto the ring of integer characters
(i.e.  onto non-irrational characters, by Prop. \ref{RationalCharactersAreInteger}).
%
%

\begin{theorem}[Cokernel of $\beta$ for binary Platonic groups]
  \label{CokernelOfBetaInVariousExamples}
  The following table lists the cokernels
  \vspace{-2mm}
  $$
    \mathrm{coker}(\beta_k)
    \;:=\;
    \frac{R_k(G)}{\mathrm{image}(\beta_k)}
    \;,
    \phantom{AAA}
    \mathrm{coker}(\beta^{\mathrm{int}}_k)
    \;:=\;
    \frac{R^{\mathrm{int}}_k(G)}{\mathrm{image}(\beta_k)}
  $$
  of the permutation-representation morphism
  $A(G) \xrightarrow{\beta_k} R_k(G)$
  \eqref{TheComparisonMap}
  and its corestriction to the integer-valued character ring (Prop. \ref{CorestrictionToIntegerCharacters}, Remark \ref{FactorizationsOfBeta}),
  for finite subgroups of $SU(2)$
  in the D- and E-series (via Prop. \ref{ADEGroups})
  and some relatives, over ground fields $k \in \{\mathbb{Q}, \mathbb{R}, \mathbb{C}\}$:

\medskip

\hspace{-1cm}
{\small
\begin{tabular}{cc|ccc|ccc||lc}
 \multicolumn{2}{r|}{ \bf coker   } &  \multicolumn{3}{c|}{ $A(G)\xrightarrow{\beta_{\mathbb{F}}} R_{\mathbb{F}}(G)$ } &  \multicolumn{3}{c||}{ $A(G) \xrightarrow{\beta^{\mathrm{int}}_{\mathbb{F}}}R^{\mathrm{int}}_{\mathbb{F}}(G)$ }
 \\
 \cline{2-8}
 \raisebox{-12pt}{
   \begin{tabular}{c} \bf  Dynkin \\ \bf label \end{tabular}
 }
 &
 \raisebox{-6pt}{ \bf group } &
 \multicolumn{3}{c|}{ \raisebox{-6pt}{\bf ground field $\mathbb{F}$ } } &  \multicolumn{3}{c||}{ \raisebox{-6pt}{\bf ground field $\mathbb{F}$ } }
 \\
 &
 $G$
 & $\mathbb{Q}$ & $\mathbb{R}$ & $\mathbb{C}$ & $\mathbb{Q}$ & $\mathbb{R}$ & $\mathbb{C}$
  &
   \hspace{-.3cm}
   \begin{tabular}{c}
     {\rm \bf Proof}
     \\
     {\rm via }
     \\
     {\rm  Thm. \ref{AlgorithmForImageOfBeta}:}
   \end{tabular}
 \\
 \cline{2-9}
 &&&&&&&
 \\

 {\rm $\mathbb{A}_1$} & $C_2$ & $0$ & $0$ & $0$
 & $0$ & \cellcolor{lightgray}$0$ & $0$
 &
    ~~~\cref{C2}
 \\
 &&&&&&&
 \\
 {\rm $\mathbb{A}_2$} & $C_3$ & $0$ & $\frac{\mathbb{Z}[\rho_2, \rho_3]}{\mathbb{Z}[\rho_2 + \rho_3]}$ & $0$
 & $0$ & \cellcolor{lightgray}$0$ & $\frac{\mathbb{Z}[\rho_2, \rho_3]}{\mathbb{Z}[\rho_2 + \rho_3]}$
 &
  ~~~\cref{C3}
 \\
 &&&&&&&
 \\
 {\rm $\mathbb{A}_3 = \mathbb{D}_3$} & $C_4$ & $0$ & $\frac{\mathbb{Z}[\rho_2, \rho_4]}{\mathbb{Z}[\rho_2 + \rho_4]}$ & $0$
 & $0$ & \cellcolor{lightgray}$0$ & $\frac{\mathbb{Z}[\rho_2, \rho_4]}{\mathbb{Z}[\rho_2 + \rho_4]}$
 &
 ~~~\cref{C4}
 \\
 &&&&&&&
 \\
 {\rm $\mathbb{D}_4$}
 & $2D_4$ & $0$ & $0$ & $\frac{\mathbb{Z}[\rho_5]}{\mathbb{Z}[2\rho_5]}$
 & $0$ & \cellcolor{lightgray}$0$ & $\frac{\mathbb{Z}[\rho_5]}{\mathbb{Z}[2\rho_5]}$
 &
  ~~~\cref{BinaryDihedral2D4}
 \\
 &&&&&&&
 \\
 {\rm $\mathbb{D}_5$} & $2D_6$ & $0$ & $0$ & $\frac{\mathbb{Z}[ \rho_3, \rho_4, \rho_6 ] }{ \mathbb{Z}[ \rho_3 + \rho_4, 2\rho_6 ] }$
 & $0$ & \cellcolor{lightgray}$0$ & $\frac{\mathbb{Z}[\rho_6 ] }{ \mathbb{Z}[2\rho_6 ] }$
 &
  ~~~\cref{BinaryDihedral2D6}
 \\
 &&&&&&&
 \\
 {\rm $\mathbb{D}_6$} & $2D_8$
 & $0$
 & $ \frac{\mathbb{Z}[2 \rho_6, 2\rho_7]  }{ \mathbb{Z}[ 2\rho_6 + 2\rho_7 ] } $
 & $ \frac{\mathbb{Z}[ \rho_6, \rho_7 ]}{ \mathbb{Z}[ 2\rho_6 + 2\rho_7 ] } $
 &
   $0$ & \cellcolor{lightgray}$0$ & $ \frac{ \mathbb{Z}[ \rho_6 + \rho_7 ] }{ \mathbb{Z}[ 2\rho_6 + 2\rho_7 ] }  $
 &
 ~~~\cref{BinaryDihedraQ16}
 \\
 &&&&&&&
 \\
 {\rm $\mathbb{D}_7$} & $2D_{10}$
 & $0$
 & $ \frac{ \mathbb{Z}[\rho_3, \rho_4, \rho_5, \rho_6, \rho_7, \rho_8] }{ \mathbb{Z}[ \rho_3 + \rho_4, \rho_5 + \rho_6, 2 \rho_7 + 2\rho_8 ] } $
 & $ \frac{ \mathbb{Z}[\rho_3, \rho_4, \rho_5, \rho_6, 2\rho_7, 2\rho_8] }{ \mathbb{Z}[ \rho_3 + \rho_4, \rho_5 + \rho_6, 2 \rho_7 + 2\rho_8 ] } $
 &
   $0$ & \cellcolor{lightgray}$0$ & $ \frac{ \mathbb{Z}[ \rho_7 + \rho_8 ] }{ \mathbb{Z}[ 2\rho_7 + 2\rho_8 ] }  $
 &
 ~~~\cref{BinaryDihedraD10}
 \\
 &&&&&&&
 \\
 {\rm $\mathbb{D}_{8}$} & $2D_{12}$
 & $0$
 & $ \frac{\mathbb{Z}[ \rho_7, \rho_8, \rho_9 ] }{ \mathbb{Z}[ 2\rho_7, 2\rho_8 + 2 \rho_9 ] } $
 & $ \frac{\mathbb{Z}[ 2\rho_8, 2\rho_9 ] }{ \mathbb{Z}[ 2\rho_8 + 2 \rho_9 ] } $
 &
   $0$ & \cellcolor{lightgray}$0$ & $ \frac{ \mathbb{Z}[ \rho_7  ] }{ \mathbb{Z}[ 2\rho_7  ] }  $
 &
 ~~~\cref{BinaryDihedraD12}
 \\
 &&&&&&&
 \\
 {\rm $\mathbb{E}_6$} & $2T$ & $0$ & $0$ & $\frac{\mathbb{Z}[\rho_2, \rho_2^\ast, \rho_4, \rho_4^\ast, \rho_5]}{\mathbb{Z}[\rho_2 + \rho_2^\ast, \rho_4 + \rho_4^\ast, 2 \rho_5]}$
 & $0$ & \cellcolor{lightgray}$0$ & $ \frac{ \mathbb{Z}[\rho_5] }{ \mathbb{Z}[2 \rho_5] } $
 &
 ~~~\cref{BinaryTetrahedral}
 \\
 &&&&&&&
 \\
 {\rm $\mathbb{E}_7$} & $2O$ &
   $0$
   &
   $
     \frac{\mathbb{Z}[2\rho_6, 2\rho_7]}{ \mathbb{Z}[ 2\rho_6 + 2\rho_7 ] }
   $
   & $\frac{ \mathbb{Z}[\rho_6, \rho_7, \rho_8] }{ \mathbb{Z}[2\rho_6 + 2\rho_7, 2\rho_8] }$
 & $0$ & \cellcolor{lightgray}$0$ & $ \frac{ \mathbb{Z}[\rho_8] }{ \mathbb{Z}[ 2\rho_8 ] } $
 &
 ~~~\cref{BinaryOctahedralGroup}
 \\
 &&&&&&&
 \\
 {\rm $\mathbb{E}_8$} & $2I$ &
    $0$
    &
    $\frac{ \mathbb{Z}[ 2\rho_2, 2\rho_3, \rho_4, \rho_5   ] }{ \mathbb{Z}[ 2\rho_2 + 2\rho_3, \rho_4 + \rho_5  ]  }$
    &
    $\frac{ \mathbb{Z}[ \rho_2, \rho_3, \rho_4, \rho_5, \rho_7, \rho_9 ] }{ \mathbb{Z}[ 2\rho_2 + 2\rho_3, \rho_4 + \rho_5, 2\rho_7, 2\rho_9 ]  }$
    &
    $0$
    &
    \cellcolor{lightgray}$0$
    &
    $\frac{ \mathbb{Z}[ \rho_2 + \rho_3, \rho_7, \rho_9 ] }{ \mathbb{Z}[ 2\rho_2 + 2\rho_3, 2\rho_7, 2\rho_9 ]  }$
    &
  ~~~\cref{BinaryIcosahedralGroup}
  \\
  &&&&&&&
  \\
  \hline
  &&&&&&&
  \\
 &
 $\mathrm{GL}(2,\mathbb{F}_3)$ & $0$ & $0$ & $ \frac{ \mathbb{Z}[\rho_6, \rho_7] }{ \mathbb{Z}[ \rho_6 + \rho_7 ] } $
 & $0$ & \cellcolor{lightgray}$0$ & $0$
 &
 ~~~\cref{TheGeneralLinear23}
 \\
 \\
 \\
\end{tabular}
}
\end{theorem}


\vspace{-5mm}
\noindent For emphasis we highlight by example how to read the
table in Theorem \ref{CokernelOfBetaInVariousExamples}:
\begin{itemize}
\vspace{-2mm}
 \item an entry ``$0$'' means that the cokernel vanishes, hence that $\beta$ is surjective;
 \vspace{-2mm}
 \item an entry ``$\frac{\mathbb{Z}[ \rho ]}{ \mathbb{Z}[2 \rho] }$'' means that the image of $\beta$ consists of all those virtual representations whose
 $\rho$-component has even multiplicity;
 \vspace{-2mm}
 \item an entry ``$\frac{\mathbb{Z}[ \rho_1, \rho_2 ]}{ \mathbb{Z}[\rho_1 + \rho_2] }$'' means that the image of $\beta$ consists of all those virtual representations whose
 $\rho_1$-component has the same weight as their $\rho_2$-component.
\end{itemize}
\vspace{-1mm}
Here $\rho_i$ refers to the irreducible representations as tabulated in the respective subsection below.
From the character tables given there one also reads off whether the character of $\rho_i$ is integer-valued
or else (by Prop. \ref{RationalCharactersAreInteger}) irrational.
The cokernel for $\beta^{\mathrm{int}}_k$ is obtained from that of $\beta_k$ by removing
those generators from the numerator that have irrational-valued characters.

\newcommand{\thinplus}{\!+\!}
\newcommand{\Complex}{\mathbb{C}}
\newcommand{\Rational}{\mathbb{Q}}


\subsection{Cyclic groups: $C_n$}

For completeness, we include discussion of the first three cyclic groups,
which may be thought of as completing the D-series 
of finite subgroups of $\mathrm{SU}(2)$ (from Prop. \ref{ADEGroups})\
in the low range.
These simple examples may serve to introduce and illustrate our notation for recording application of the algorithm (Theorem \ref{AlgorithmForImageOfBeta})
in the examples to follow further below.

\subsubsection{The cyclic group $C_2$}
\label{C2}

\noindent Group name: $C_2$ (\cite{DokC2})

\noindent Group order: 2

\medskip

\noindent Subgroups:

$$
\begin{array}{|c|c|c|c|c|}
\hline
\text{subgroup} & \text{order} & \text{cosets} & \text{conjugates} & \text{cyclic} \\
\hline
A & 2 & 1 & 1 & \checkmark \\
\hline
B & 1 & 2 & 1 & \checkmark \\
\hline
\end{array}
$$

\noindent Burnside ring product:

$$
\begin{array}{r|rr}
\times & A & B \\
\hline
A & A & B \\
B & B & 2B\\
\end{array}
$$

\noindent Table of multiplicities:

$$
\begin{array}{c|cc}
 & A & B \\
\hline
A & 1 & 1 \\
B & 1 & 2 \\
\end{array}
$$

\noindent Upper triangular form:

$$
\begin{array}{c|cc}
 & A & B \\
\hline
V_1 & 1 & 1 \\
V_2 & . & 1 \\
\end{array}
$$

\noindent Character table for image of $\beta$:

$$
\begin{array}{r|rr}
\mathrm{class} & 1 & 2 \\
\mathrm{size} & 1 & 1 \\
\hline
V_1 & 1 & 1 \\
V_2 & 1 & -1 \\
\end{array}
$$

\noindent Character table of irreps:

over $\mathbb{C}$:

\begin{center}
\begin{tabular}{cc|cc}
 \multicolumn{2}{c}{  }
 &
 \multicolumn{2}{c}
 {
   $
   \mathclap{
   \begin{tabular}{c}
     conjugacy
     \\
     class
   \end{tabular}
   }
   $
 }
 \\
 \multicolumn{2}{c|}{} & $1$ & $2$
 \\
 \cline{2-4}
 \multirow{2}{*}{
   \raisebox{-25pt}
     {
       \begin{rotate}{90}
         \begin{tabular}{c}
           irred.
           \\
           repr.
           \\
           $\phantom{{A \atop A} \atop {A \atop A}}$
         \end{tabular}
       \end{rotate}
     }
     \hspace{-20pt} }
 &
 $\rho_1$ & $\phantom{+}1$ & $\phantom{+}1$
 \\
 &
 $\rho_2$ & $\phantom{+}1$ & $-1$
\end{tabular}
\end{center}

over $\mathbb{R}$:

\begin{center}
\begin{tabular}{cc|cc}
 \multicolumn{2}{c}{  }
 &
 \multicolumn{2}{c}
 {
   $
   \mathclap{
   \begin{tabular}{c}
     conjugacy
     \\
     class
   \end{tabular}
   }
   $
 }
 \\
 \multicolumn{2}{c|}{} & $1$ & $2$
 \\
 \cline{2-4}
 \multirow{2}{*}{
   \raisebox{-25pt}
     {
       \begin{rotate}{90}
         \begin{tabular}{c}
           irred.
           \\
           repr.
           \\
           $\phantom{{A \atop A} \atop {A \atop A}}$
         \end{tabular}
       \end{rotate}
     }
     \hspace{-20pt} }
 &
 $\rho_1$ & $\phantom{+}1$ & $\phantom{+}1$
 \\
 &
 $\rho_2$ & $\phantom{+}1$ & $-1$
\end{tabular}
\end{center}

\noindent Hence the cokernel of $\beta$ is:

$$
  \begin{array}{c|c}
  \begin{aligned}
    V_1 & = \rho_1
    \\
    V_2 & = \rho_2
  \end{aligned}
  &
  \mathrm{coker}
  \left(
    A(C_2)
    \overset{\beta}{\to}
    R_k(C_2)
  \right)
  \;\simeq\;
  \left\{
  \begin{array}{ccc}
    $0$ &\vert& k = \mathbb{C}
    \\
    $0$ &\vert& k =\mathbb{R}
    \\
    0 &\vert& k = \mathbb{Q}
  \end{array}
  \right.
  \end{array}
$$

\subsubsection{The cyclic group $C_3$}
\label{C3}

\noindent Group name: $C_3$ (\cite{DokC3})

\noindent Group order: 3

\medskip

\noindent Subgroups:

$$
\begin{array}{|c|c|c|c|c|}
\hline
\text{subgroup} & \text{order} & \text{cosets} & \text{conjugates} & \text{cyclic} \\
\hline
A & 3 & 1 & 1 & \checkmark \\
\hline
B & 1 & 3 & 1 & \checkmark \\
\hline
\end{array}
$$

\noindent Burnside ring product:

$$
\begin{array}{r|rr}
\times & A & B \\
\hline
A & A & B \\
B & B & 3B\\
\end{array}
$$

\noindent Table of multiplicities:

$$
\begin{array}{c|cc}
 & A & B \\
\hline
A & 1 & 1 \\
B & 1 & 3 \\
\end{array}
$$

\noindent Upper triangular form:

$$
\begin{array}{c|cc}
 & A & B \\
\hline
V_1 & 1 & 1 \\
V_2 & . & 2 \\
\end{array}
$$

\noindent Character table for image of $\beta$:

$$
\begin{array}{r|rrr}
\mathrm{class} & 1 & 2 & 3 \\
\mathrm{size} & 1 & 1 & 1\\
\hline
V_1 & 1 & 1 & 1 \\
V_2 & 2 & -1 & -1 \\
\end{array}
$$

\noindent Character table of irreps:

over $\mathbb{C}$:

\begin{center}
\begin{tabular}{cc|ccc}
 \multicolumn{2}{c}{  }
 &
 \multicolumn{3}{c}
 {
   $
   \mathclap{
   \begin{tabular}{c}
     conjugacy
     class
   \end{tabular}
   }
   $
 }
 \\
 \multicolumn{2}{c|}{} & $1$ & $2$ & $3$
 \\
 \cline{2-5}
 \multirow{2}{*}{
   \raisebox{-50pt}
     {
       \begin{rotate}{90}
         \begin{tabular}{c}
           irred. repr.
           \\
           $\phantom{ {A \atop A} \atop {A \atop A}}$
         \end{tabular}
       \end{rotate}
     }
     \hspace{-20pt} }
 &
 $\rho_1$ & $\phantom{+}1$ & $\phantom{+}1$ & $\phantom{+}1$
 \\
 &
 $\rho_2$ & $\phantom{+}1$ & $e^{2\pi i \tfrac{1}{3}}$ & $e^{2\pi i \tfrac{2}{3}}$
 \\
 &
 $\rho_3$ & $\phantom{+}1$ & $e^{2\pi i \tfrac{-1}{3}}$ & $e^{-2\pi i \tfrac{2}{3}}$
\end{tabular}
\end{center}

over $\mathbb{R}$:

\begin{center}
\begin{tabular}{cc|ccc}
 \multicolumn{2}{c}{  }
 &
 \multicolumn{3}{c}
 {
   $
   \mathclap{
   \begin{tabular}{c}
     conjugacy
     class
   \end{tabular}
   }
   $
 }
 \\
 \multicolumn{2}{c|}{} & $1$ & $2$ & $3$
 \\
 \cline{2-5}
 \multirow{2}{*}{
   \raisebox{-26pt}
     {
       \begin{rotate}{90}
         \begin{tabular}{c}
           irred.
           \\
           repr.
           \\
           $\phantom{ {A \atop A} \atop {A \atop A}}$
         \end{tabular}
       \end{rotate}
     }
     \hspace{-20pt} }
 &
 $\rho_1$ & $\phantom{+}1$ & $\phantom{+}1$ & $\phantom{+}1$
 \\
 &
 $\rho_2 + \rho_3$ & $\phantom{+}2$ & $-1$ & $-1$
\end{tabular}
\end{center}

\noindent Hence the cokernel of $\beta$ is:

$$
  \begin{array}{c|c}
  \begin{aligned}
    V_1 & = \rho_1
    \\
    V_2 & = \rho_2 + \rho_3
  \end{aligned}
  &
  \mathrm{coker}
  \left(
    A(C_3)
    \overset{\beta}{\to}
    R_k(C_3)
  \right)
  \;\simeq\;
  \left\{
  \begin{array}{ccc}
    \tfrac{\mathbb{Z}[\rho_2, \rho_3]}{\mathbb{Z}[\rho_2 + \rho_3]} &\vert& k = \mathbb{C}
    \\
    $0$ &\vert& k =\mathbb{R}
    \\
    0 &\vert& k = \mathbb{Q}
  \end{array}
  \right.
  \end{array}
$$

\subsubsection{The cyclic group $C_4$}
\label{C4}

\noindent Group name: $C_4$ (\cite{DokC4})

\noindent Group order: 4

\medskip

\noindent Subgroups:

$$
\begin{array}{|c|c|c|c|c|}
\hline
\text{subgroup} & \text{order} & \text{cosets} & \text{conjugates} & \text{cyclic} \\
\hline
A & 4 & 1 & 1 & \checkmark \\
\hline
B & 2 & 2 & 1 & \checkmark \\
\hline
C & 1 & 4 & 1 & \checkmark \\
\hline
\end{array}
$$

\noindent Burnside ring product:

$$
\begin{array}{r|rrr}
\times & A & B & C\\
\hline
A & A & B & C\\
B & B & 2B & 2C \\
C & C & 2C & 4C \\
\end{array}
$$

\noindent Table of multiplicities:

$$
\begin{array}{c|ccc}
 & A & B & C\\
\hline
A & 1 & 1 & 1 \\
B & 1 & 2 & 2 \\
C & 1 & 2 & 4 \\
\end{array}
$$

\noindent Upper triangular form:

$$
\begin{array}{c|cccc}
 & A & B & C\\
\hline
V_1 & 1 & 1 & 1 \\
V_2 & 0 & 1 & 1 \\
V_3 & 0 & 0 & 2 \\
\end{array}
$$

\noindent Character table for image of $\beta$:

$$
\begin{array}{r|rrrr}
\mathrm{class} & 1 & 2 & 3 & 4\\
\mathrm{size} & 1 & 1 & 1 & 1\\
\hline
V_1 & 1 & 1 & 1 & 1\\
V_2 & 1 & -1 & 1 & -1 \\
V_3 & 2 
    & 0 
    & -2 
    & 0 
\end{array}
$$

\noindent Character table of irreps:

over $\mathbb{C}$:

\begin{center}
\begin{tabular}{cc|cccc}
 \multicolumn{2}{c}{  } &  \multicolumn{4}{c}{ conjugacy class }
 \\
 \multicolumn{2}{c|}{} & $1$ & 2 & 3 & 4
 \\
 \cline{2-6}
 \multirow{3}{*}{
   \raisebox{-35pt}{
     \begin{rotate}{90}
       \begin{tabular}{c}
         irred.
         \\
         repr.
         \\
         $\phantom{{A \atop A} \atop {A \atop A}}$
       \end{tabular}
     \end{rotate}
   }
   \hspace{-20pt}
 }
 &
 $\rho_1$ & $\phantom{+}1$ & $\phantom{+}1$ & $\phantom{+}1$ & $\phantom{+}1$
 \\
 &
 $\rho_2$ & $\phantom{+}1$ & $\phantom{+}i$ & $-1$ & $-i$
 \\
 &
 $\rho_3$ & $\phantom{+}1$ & $-1$ & $\phantom{+}1$ & $-1$
 \\
 &
 $\rho_4$ & $\phantom{+}1$ & $-i$ & $-1$ & $\phantom{+}i$
 \\
\end{tabular}
\end{center}

over $\mathbb{R}$:

\begin{center}
\begin{tabular}{cc|cccc}
 \multicolumn{2}{c}{  } &  \multicolumn{4}{c}{ conjugacy class }
 \\
 \multicolumn{2}{c|}{} & $1$ & 2 & 3 & 4
 \\
 \cline{2-6}
 \multirow{3}{*}{
   \raisebox{-30pt}{
     \begin{rotate}{90}
       \begin{tabular}{c}
         irred.
         \\
         repr.
         \\
         $\phantom{{A \atop A} \atop {A \atop A}}$
       \end{tabular}
     \end{rotate}
   }
   \hspace{-20pt}
 }
 &
 $\rho_1$ & $\phantom{+}1$ & $\phantom{+}1$ & $\phantom{+}1$ & $\phantom{+}1$
 \\
 &
 $\rho_3$ & $\phantom{+}1$ & $-1$ & $\phantom{+}1$ & $-1$
 \\
 &
 $\rho_2 + \rho_4$ & $\phantom{+}2$ & $\phantom{+}0$ & $-2$ & $\phantom{+}0$
 \\
\end{tabular}
\end{center}

\noindent Hence the cokernel of $\beta$ is:

$$
  \begin{array}{c|c}
  \begin{aligned}
    V_1 & = \rho_1
    \\
    V_2 & = \rho_3
    \\
    V_3 & = \rho_2 + \rho_4
  \end{aligned}
  &
  \mathrm{coker}
  \left(
    A(C_4)
    \overset{\beta}{\to}
    R_k(C_4)
  \right)
  \;\simeq\;
  \left\{
  \begin{array}{ccc}
    \frac{ \mathbb{Z}[\rho_2, \rho_4] }{ \mathbb{Z}[\rho_2 + \rho_4] } &\vert& k = \mathbb{C}
    \\
    $0$ &\vert& k =\mathbb{R}
    \\
    0 &\vert& k = \mathbb{Q}
  \end{array}
  \right.
  \end{array}
$$

\subsection{Binary dihedral groups: $2D_{2n} \simeq \mathrm{Dic}_n$}

The binary dihedral groups have the following presentation (see e.g. \cite{DokDic} \cite{Lindh2018}):
$$
    2D_{2n} := \langle r, s | r^{2n}=1, s^2=r^n, s^{-1}rs = r^{-1}\rangle.
$$
The order of $2D_{2n}$ is $4n$ and has $n+3$ conjugacy classes:
\begin{align*}
&\{1\}, \{s^2\},\\
&\{r, r^{2n-1}\}, \{r^2,r^{2n-2}\}, ..., \{r^{n-1}, r^{n+1}\}, \\
&\{s, sr^2,...,sr^{2n-2}\}, \{sr, sr^3,...,sr^{2n-1}\}.
\end{align*}
The $n+3$ complex irreducible characters are given by:
$$
\begin{array}{c|cccccccc}
2D_{2n} = \mathrm{Dic}_n  &  1  &  s^2  &  r  &  r^2  &  \cdots  &  r^{n-1}   &  s  &  sr  \\
\hline
\mathrm{Triv.}  &  1   &  1  &  1  &  1  &  \cdots  &  1  &  1  &  1 \\
\mathrm{1A}  &  1   &  1  &  1  &  1  &  \cdots  &  1  &  -1  &  -1 \\
\mathrm{1B}  &  1   &  (-1)^n  &  -1  &  1  &  \cdots  &  (-1)^{n-1}  &  i^n  &  -i^n \\
\mathrm{1C}  &  1   &  (-1)^n  &  -1  &  1  &  \cdots  &  (-1)^{n-1}  &  -i^n  &  i^n \\
\rho_1 & 2 & -2 & \zeta+\zeta^{-1} & \zeta^2+\zeta^{-2} & \cdots & \zeta^{n-1}+\zeta^{1-n} & 0 & 0 \\
\rho_k & 2 & (-1)^k2 & \zeta^k+\zeta^{-k} & \zeta^{2k}+\zeta^{-2k} & \cdots & \zeta^{k(n-1)}+\zeta^{k(1-n)} & 0 & 0 \\
\end{array}
$$
with $k=2,...,n-1$ and $\zeta = e^{2\pi i/2n}.$

The representations Triv. and 1A are always real.
For $n$ even 1B and 1C are real.
For $n$ odd 1A and 1B are complex, and 1A$+$1B is real irreducible.
For $k$ even $\rho_k$ is real. For $k$ odd $\rho_k$ is quaternionic, and so $2\rho_k$ is real.



\subsubsection{Binary dihedral group: $2 D_4 \simeq  \mathrm{Dic}_2 \simeq Q_8$}
 \label{BinaryDihedral2D4}

Group name: $2 D_4 \simeq Q_8$ (\cite{DokQ8})

\noindent Group order: ${\vert 2D_4\vert} = 8$

\medskip

\noindent Subgroups:

$$
\begin{array}{|c|c|c|c|c|}
\hline
\text{subgroup} & \text{order} & \text{cosets} & \text{conjugates} & \text{cyclic} \\
\hline
A & 8 & 1 & 1 &  \\
\hline
B & 4 & 2 & 1 & \checkmark \\
\hline
C & 4 & 2 & 1 & \checkmark \\
\hline
D & 4 & 2 & 1 & \checkmark \\
\hline
E & 2 & 4 & 1 & \checkmark \\
\hline
F & 1 & 8 & 1 & \checkmark \\
\hline
\end{array}
$$

\noindent Burnside ring product:

$$
\begin{array}{r|rrrrrr}
\times & A & B & C & D & E & F \\
\hline
A & A & B & C & D & E & F \\
B & B & 2B & E & E & 2E & 2F \\
C & C & E & 2C & E & 2E & 2F \\
D & D & E & E & 2D & 2E & 2F \\
E & E & 2E & 2E & 2E & 4E & 4F \\
F & F & 2F & 2F & 2F & 4F & 8F \\
\end{array}
$$

\noindent Table of multiplicities:

$$
\begin{array}{c|cccccc}
 & A & B & C & D & E & F \\
\hline
A & 1 & 1 & 1 & 1 & 1 & 1 \\
B & 1 & 2 & 1 & 1 & 2 & 2 \\
C & 1 & 1 & 2 & 1 & 2 & 2 \\
D & 1 & 1 & 1 & 2 & 2 & 2 \\
E & 1 & 2 & 2 & 2 & 4 & 4 \\
F & 1 & 2 & 2 & 2 & 4 & 8 \\
\end{array}
$$

\noindent Upper triangular form:

$$
\begin{array}{c|cccccc}
 & A & B & C & D & E & F \\
\hline
V_1 & 1 & 1 & 1 & 1 & 1 & 1 \\
V_2 & . & 1 & . & . & 1 & 1 \\
V_3 & . & . & 1 & . & 1 & 1 \\
V_4 & . & . & . & 1 & 1 & 1 \\
V_5 & . & . & . & . & . & 4 \\
\end{array}
$$


\noindent Character table for image of $\beta$:

$$
\begin{array}{r|rrrrr}
\mathrm{class} & 1 & 2 & 4A & 4B & 4C \\
\mathrm{size} & 1 & 1 & 2 & 2 & 2 \\
\hline
V_1 & 1 & 1 & 1 & 1 & 1 \\
V_2 & 1 & 1 & -1 & 1 & -1 \\
V_3 & 1 & 1 & -1 & -1 & 1 \\
V_4 & 1 & 1 & 1 & -1 & -1 \\
V_5 & 4 & -4 & 0 & 0 & 0 \\
\end{array}
$$

\medskip

\noindent Character table of irreps \cite{DokQ8, MonQ8}

over $\mathbb{C}$:

\begin{center}
\begin{tabular}{cc|ccccc}
 \multicolumn{2}{c}{  } &  \multicolumn{5}{c}{ conjugacy class }
 \\
 \multicolumn{2}{c|}{} & $1$ & 2 & 4A & 4B & 4C
 \\
 \cline{2-7}
 \multirow{5}{*}{ \raisebox{-50pt}{\begin{rotate}{90} irred. repr. \end{rotate}} \hspace{-20pt} }
 &
 $\rho_1$ & $\phantom{+}1$ & $\phantom{+}1$ & $\phantom{+}1$ & $\phantom{+}1$ & $\phantom{+}1$
 \\
 &
 $\rho_2$ & $\phantom{+}1$ & $\phantom{+}1$ & $-1$ & $\phantom{+}1$ & $-1$
 \\
 &
 $\rho_3$ & $\phantom{+}1$ & $\phantom{+}1$ & $\phantom{+}1$ & $-1$ & $-1$
 \\
 &
 $\rho_4$ & $\phantom{+}1$ & $\phantom{+}1$ & $-1$ & $-1$ & $\phantom{+}1$
 \\
 &
 $\rho_5$ & $\phantom{+}2$ & $-2$ & $\phantom{+}0$ & $\phantom{+}0$ & $\phantom{+}0$
\end{tabular}
\end{center}

\medskip

over $\mathbb{R}$:

\begin{center}
\begin{tabular}{cc|ccccc}
 \multicolumn{2}{c}{  } &  \multicolumn{5}{c}{ conjugacy class }
 \\
 \multicolumn{2}{c|}{} & 1 & 2 & 4A & 4B & 4C
 \\
 \cline{2-7}
 \multirow{5}{*}{ \raisebox{-50pt}{\begin{rotate}{90} irred. repr. \end{rotate}} \hspace{-20pt} }
 &
 $\rho_1$ & $\phantom{+}1$ & $\phantom{+}1$ & $\phantom{+}1$ & $\phantom{+}1$ & $\phantom{+}1$
 \\
 &
 $\rho_2$ & $\phantom{+}1$ & $\phantom{+}1$ & $-1$ & $\phantom{+}1$ & $-1$
 \\
 &
 $\rho_3$ & $\phantom{+}1$ & $\phantom{+}1$ & $\phantom{+}1$ & $-1$ & $-1$
 \\
 &
 $\rho_4$ & $\phantom{+}1$ & $\phantom{+}1$ & $-1$ & $-1$ & $\phantom{+}1$
 \\
 &
 $2 \rho_5$ & $\phantom{+}4$ & $-4$ & $\phantom{+}0$ & $\phantom{+}0$ & $\phantom{+}0$
\end{tabular}
\end{center}

\medskip

\noindent Hence the cokernel of $\beta$ is:

$$
  \begin{array}{c|c}
  \begin{aligned}
    V_1 & = \rho_1
    \\
    V_2 & = \rho_2
    \\
    V_3 & = \rho_4
    \\
    V_4 & = \rho_3
    \\
    V_5 & = 2 \rho_5
  \end{aligned}
  &
  \mathrm{coker}
  \left(
    A(2D_4)
    \overset{\beta}{\to}
    R_k(2D_4)
  \right)
  \;\simeq\;
  \left\{
  \begin{array}{ccc}
    \frac{ \mathbb{Z}[\rho_5] }{ \mathbb{Z}[ 2\rho_5 ]  } &\vert& k = \mathbb{C}
    \\
    $0$ &\vert& k =\mathbb{R}
    \\
    0 &\vert& k = \mathbb{Q}
  \end{array}
  \right.
  \end{array}
$$


\subsubsection{Binary dihedral group: $2 D_6 \simeq \mathrm{Dic}_3$}
 \label{BinaryDihedral2D6}

\noindent Group name: $2 D_6$ (\cite{Dok2D6})

\noindent Group order: $\vert 2 D_6\vert = 12$

\medskip

\noindent Subgroups:

$$
\begin{array}{|c|c|c|c|c|}
\hline
\text{subgroup} & \text{order} & \text{cosets} & \text{conjugates} & \text{cyclic} \\
\hline
A & 12 & 1 & 1 &  \\
\hline
B & 6 & 2 & 1 & \checkmark \\
\hline
C & 4 & 3 & 3 & \checkmark \\
\hline
D & 3 & 4 & 1 & \checkmark \\
\hline
E & 2 & 6 & 1 & \checkmark \\
\hline
F & 1 & 12 & 1 & \checkmark \\
\hline
\end{array}
$$

\noindent Burnside ring product:

$$
\begin{array}{r|rrrrrr}
\times & A & B & C & D & E & F \\
\hline
A & A & B & C & D & E & F \\
B & B & 2B & E & 2D & 2E & 2F \\
C & C & E & C\thinplus E & F & 3E & 3F \\
D & D & 2D & F & 4D & 2F & 4F \\
E & E & 2E & 3E & 2F & 6E & 6F \\
F & F & 2F & 3F & 4F & 6F & 12F \\
\end{array}
$$

\noindent Table of multiplicities:

$$
\begin{array}{c|cccccc}
 & A & B & C & D & E & F \\
\hline
A & 1 & 1 & 1 & 1 & 1 & 1 \\
B & 1 & 2 & 1 & 2 & 2 & 2 \\
C & 1 & 1 & 2 & 1 & 3 & 3 \\
D & 1 & 2 & 1 & 4 & 2 & 4 \\
E & 1 & 2 & 3 & 2 & 6 & 6 \\
F & 1 & 2 & 3 & 4 & 6 & 12 \\
\end{array}
$$

\noindent Upper triangular form:

$$
\begin{array}{c|cccccc}
 & A & B & C & D & E & F \\
\hline
V_1 & 1 & 1 & 1 & 1 & 1 & 1 \\
V_2 & . & 1 & . & 1 & 1 & 1 \\
V_3 & . & . & 1 & . & 2 & 2 \\
V_4 & . & . & . & 2 & . & 2 \\
V_5 & . & . & . & . & . & 4 \\
\end{array}
$$


\noindent Character table for image of $\beta$:

$$
\begin{array}{r|rrrrrr}
\mathrm{class} & 1 & 2 & 3 & 4A & 4B & 6 \\
\mathrm{size} & 1 & 1 & 2 & 3 & 3 & 2 \\
\hline
V_1 & 1 & 1 & 1 & 1 & 1 & 1 \\
V_2 & 1 & 1 & 1 & -1 & -1 & 1 \\
V_3 & 2 & 2 & -1 & 0 & 0 & -1 \\
V_4 & 2 & -2 & 2 & 0 & 0 & -2 \\
V_5 & 4 & -4 & -2 & 0 & 0 & 2 \\
\end{array}
$$

\medskip

\noindent Character table of irreps \cite{Dok2D6}

over $\mathbb{C}$:

\begin{center}
\begin{tabular}{cc|cccccc}
 \multicolumn{2}{c}{  } &  \multicolumn{5}{c}{ conjugacy class }
 \\
 \multicolumn{2}{c|}{} & $1$ & 2 & 3 &  4A & 4B & 6
 \\
 \cline{2-8}
 \multirow{5}{*}{ \raisebox{-50pt}{\begin{rotate}{90} irred. repr. \end{rotate}} \hspace{-20pt} }
 &
 $\rho_1$ & $\phantom{+}1$ & $\phantom{+}1$ & $\phantom{+}1$ & $\phantom{+}1$ & $\phantom{+}1$ & $\phantom{+}1$
 \\
 &
 $\rho_2$ & $\phantom{+}1$ & $\phantom{+}1$ & $\phantom{+}1$ & $-1$ & $-1$ & $\phantom{+}1$
 \\
 &
 $\rho_3$ & $\phantom{+}1$ & $-1$ & $\phantom{+}1$ & $\phantom{+}i$ & $-i$ & $-1$
 \\
 &
 $\rho_4$ & $\phantom{+}1$ & $-1$ & $\phantom{+}1$ & $-i$ & $\phantom{+}i$ & $-1$
 \\
 &
 $\rho_5$ & $\phantom{+}2$ & $\phantom{+}2$ & $-1$ & $\phantom{+}0$ & $\phantom{+}0$ & $-1$
 \\
 &
 $\rho_6$ & $\phantom{+}2$ & $-2$ & $-1$ & $\phantom{+}0$ & $\phantom{+}0$ & $\phantom{+}1$
\end{tabular}
\end{center}

over $\mathbb{R}$:

\begin{center}
\begin{tabular}{cc|cccccc}
 \multicolumn{2}{c}{  } &  \multicolumn{5}{c}{ conjugacy class }
 \\
 \multicolumn{2}{c|}{} & $1$ & 2 & 3 &  4A & 4B & 6
 \\
 \cline{2-8}
 \multirow{5}{*}{ \raisebox{-50pt}{\begin{rotate}{90} irred. repr. \end{rotate}} \hspace{-20pt} }
 &
 $\rho_1$ & $\phantom{+}1$ & $\phantom{+}1$ & $\phantom{+}1$ & $\phantom{+}1$ & $\phantom{+}1$ & $\phantom{+}1$
 \\
 &
 $\rho_2$ & $\phantom{+}1$ & $\phantom{+}1$ & $\phantom{+}1$ & $-1$ & $-1$ & $\phantom{+}1$
 \\
 &
 $\rho_3 + \rho_4$ & $\phantom{+}2$ & $-2$ & $\phantom{+}2$ & $\phantom{+}0$ & $\phantom{+}0$ & $-2$
 \\
 &
 $\rho_5$ & $\phantom{+}2$ & $\phantom{+}2$ & $-1$ & $\phantom{+}0$ & $\phantom{+}0$ & $-1$
 \\
 &
 $2\rho_6$ & $\phantom{+}4$ & $-4$ & $-2$ & $\phantom{+}0$ & $\phantom{+}0$ & $\phantom{+}2$
\end{tabular}
\end{center}

\medskip

Hence the cokernel of $\beta$ is

$$
  \begin{array}{c|c}
  \begin{aligned}
    V_1 & = \rho_1
    \\
    V_2 & = \rho_2
    \\
    V_3 & = \rho_5
    \\
    V_4 & = \rho_3 + \rho_4
    \\
    V_5 & = 2\rho_6
  \end{aligned}
  &
  \mathrm{coker}\left( \beta_k\right)
  \;=\;
  \left\{
  \begin{array}{ccc}
    \frac{\mathbb{Z}[ \rho_3, \rho_4, \rho_6 ] }{ \mathbb{Z}[ \rho_3 + \rho_4, 2\rho_6 ] } &\vert& k = \mathbb{C}
    \\
    0 &\vert& k = \mathbb{R}
    \\
    0 &\vert& k = \mathbb{Q}
  \end{array}
  \right.
  \end{array}
$$


\subsubsection{Binary dihedral group: $2 D_{8} \simeq \mathrm{Dic}_4 \simeq Q_{16}$}
 \label{BinaryDihedraQ16}

\noindent Group name: $2 D_8$ (\cite{DokQ16})

\noindent Group order: $\vert 2 D_8\vert = 16$

\medskip

\noindent Subgroups:

$$
\begin{array}{|c|c|c|c|c|}
\hline
\text{subgroup} & \text{order} & \text{cosets} & \text{conjugates} & \text{cyclic} \\
\hline
A & 16 & 1 & 1 &  \\
\hline
B & 8 & 2 & 1 & \checkmark \\
\hline
C & 8 & 2 & 1 &  \\
\hline
D & 8 & 2 & 1 &  \\
\hline
E & 4 & 4 & 2 & \checkmark \\
\hline
F & 4 & 4 & 1 & \checkmark \\
\hline
G & 4 & 4 & 2 & \checkmark \\
\hline
H & 2 & 8 & 1 & \checkmark \\
\hline
I & 1 & 16 & 1 & \checkmark \\
\hline
\end{array}
$$

\noindent Table of multiplicities:

$$
\begin{array}{c|ccccccccc}
 & A & B & C & D & E & F & G & H & I \\
\hline
A & 1 & 1 & 1 & 1 & 1 & 1 & 1 & 1 & 1 \\
B & 1 & 2 & 1 & 1 & 1 & 2 & 1 & 2 & 2 \\
C & 1 & 1 & 2 & 1 & 2 & 2 & 1 & 2 & 2 \\
D & 1 & 1 & 1 & 2 & 1 & 2 & 2 & 2 & 2 \\
E & 1 & 1 & 2 & 1 & 3 & 2 & 2 & 4 & 4 \\
F & 1 & 2 & 2 & 2 & 2 & 4 & 2 & 4 & 4 \\
G & 1 & 1 & 1 & 2 & 2 & 2 & 3 & 4 & 4 \\
H & 1 & 2 & 2 & 2 & 4 & 4 & 4 & 8 & 8 \\
I & 1 & 2 & 2 & 2 & 4 & 4 & 4 & 8 & 16 \\
\end{array}
$$

\noindent Upper triangular form:

$$
\begin{array}{c|ccccccccc}
 & A & B & C & D & E & F & G & H & I \\
\hline
V_1 & 1 & 1 & 1 & 1 & 1 & 1 & 1 & 1 & 1 \\
V_2 & . & 1 & . & . & . & 1 & . & 1 & 1 \\ 
V_3 & . & . & 1 & . & 1 & 1 & . & 1 & 1 \\ 
V_4 & . & . & . & 1 & . & 1 & 1 & 1 & 1 \\ 
V_5 & . & . & . & . & 1 & . & 1 & 2 & 2 \\
V_6 & . & . & . & . & . & . & . & . & 8 \\
\end{array}
$$


\noindent Character table for image of $\beta$:

$$
\begin{array}{r|rrrrrrr}
\mathrm{class} & 1 & 2 & 4A & 4B & 4C & 8A & 8B \\
\mathrm{size}  & 1 & 1 & 2 & 4 & 4 & 2 & 2 \\
\hline
           V_1 & 1 &  1 & 1 &  1 &  1 &  1 &  1 \\
           V_2 & 1 &  1 & 1 & -1 &  1 &  -1 & -1 \\
           V_3 & 1 &  1 & 1 & -1 & -1 &  1 &  1 \\
           V_4 & 1 &  1 & 1 &  1 & -1 &  -1 & -1 \\
           V_5 & 2 &  2 & -2&  0 &  0 &  0 &  0 \\
           V_6 & 8 & -8 & 0 &  0 &  0 &  0 &  0 \\
\end{array}
$$

\medskip

\noindent Character table of irreps \cite{DokQ16}:

over $\mathbb{C}$

\begin{center}
\begin{tabular}{cc|ccccccc}
 \multicolumn{2}{c}{  } &  \multicolumn{7}{c}{ conjugacy class }
 \\
 \multicolumn{2}{c|}{} & 1 & 2 & 4A & 4B & 4C & 8A & 8B
 \\
 \cline{2-9}
 \multirow{7}{*}{ \raisebox{-50pt}{\begin{rotate}{90} irred. repr. \end{rotate}} \hspace{-20pt} }
 &
 $\rho_1$ & $\phantom{+}1$ & $\phantom{+}1$ & $\phantom{+}1$ & $\phantom{+}1$ & $\phantom{+}1$ & $\phantom{+}1$ & $\phantom{+}1$
 \\
 &
 $\rho_2$ & $\phantom{+}1$ & $\phantom{+}1$ & $\phantom{+}1$ & $\phantom{+}1$ & $-1$ & $-1$ & $-1$
 \\
 &
 $\rho_3$ & $\phantom{+}1$ & $\phantom{+}1$ & $\phantom{+}1$ & $-1$ & $\phantom{+}1$ & $-1$ & $-1$
 \\
 &
 $\rho_4$ & $\phantom{+}1$ & $\phantom{+}1$ & $\phantom{+}1$ & $-1$ & $-1$ & $\phantom{+}1$ & $\phantom{+}1$
 \\
 &
 $\rho_5$ & $\phantom{+}2$ & $\phantom{+}2$ & $-2$ & $\phantom{+}0$ & $\phantom{+}0$ & $\phantom{+}0$ & $\phantom{+}0$
 \\
 &
 $\rho_6$ & $\phantom{+}2$ & $-2$ & $\phantom{+}0$ & $\phantom{+}0$ & $\phantom{+}0$ & $\sqrt{2}$ & $-\sqrt{2}$
 \\
 &
 $\rho_7$ & $\phantom{+}2$ & $-2$ & $\phantom{+}0$ & $\phantom{+}0$ & $\phantom{+}0$ & $-\sqrt{2}$ & $\sqrt{2}$
\end{tabular}
\end{center}

over $\mathbb{R}$

\begin{center}
\begin{tabular}{cc|ccccccc}
 \multicolumn{2}{c}{  } &  \multicolumn{7}{c}{ conjugacy class }
 \\
 \multicolumn{2}{c|}{} & 1 & 2 & 4A & 4B & 4C & 8A & 8B
 \\
 \cline{2-9}
 \multirow{7}{*}{ \raisebox{-50pt}{\begin{rotate}{90} irred. repr. \end{rotate}} \hspace{-20pt} }
 &
 $\rho_1$ & $\phantom{+}1$ & $\phantom{+}1$ & $\phantom{+}1$ & $\phantom{+}1$ & $\phantom{+}1$ & $\phantom{+}1$ & $\phantom{+}1$
 \\
 &
 $\rho_2$ & $\phantom{+}1$ & $\phantom{+}1$ & $\phantom{+}1$ & $\phantom{+}1$ & $-1$ & $-1$ & $-1$
 \\
 &
 $\rho_3$ & $\phantom{+}1$ & $\phantom{+}1$ & $\phantom{+}1$ & $-1$ & $\phantom{+}1$ & $-1$ & $-1$
 \\
 &
 $\rho_4$ & $\phantom{+}1$ & $\phantom{+}1$ & $\phantom{+}1$ & $-1$ & $-1$ & $\phantom{+}1$ & $\phantom{+}1$
 \\
 &
 $\rho_5$ & $\phantom{+}2$ & $\phantom{+}2$ & $-2$ & $\phantom{+}0$ & $\phantom{+}0$ & $\phantom{+}0$ & $\phantom{+}0$
 \\
 &
 $2\rho_6$ & $\phantom{+}4$ & $-4$ & $\phantom{+}0$ & $\phantom{+}0$ & $\phantom{+}0$ & $2\sqrt{2}$ & $-2\sqrt{2}$
 \\
 &
 $2\rho_7$ & $\phantom{+}4$ & $-4$ & $\phantom{+}0$ & $\phantom{+}0$ & $\phantom{+}0$ & $-2\sqrt{2}$ & $2\sqrt{2}$
\end{tabular}
\end{center}

\medskip

\noindent Hence the cokernel of $\beta$ is:

$$
  \begin{array}{c|c}
  \begin{aligned}
    V_1 & = \rho_1
    \\
    V_2 & = \rho_3
    \\
    V_3 & = \rho_4
    \\
    V_4 & = \rho_2
    \\
    V_5 & = \rho_5
    \\
    V_6 & = 2 \rho_6 + 2 \rho_7
  \end{aligned}
  &
  \mathrm{coker}
  \left(
    A(2D_8)
    \overset{\beta}{\to}
    R_k(2D_8)
  \right)
  \;\simeq\;
  \left\{
  \begin{array}{cccc}
    \frac{\mathbb{Z}[ \rho_6, \rho_7 ]}{ \mathbb{Z}[ 2\rho_6 + 2\rho_7 ] } &\vert& k = \mathbb{C}
    \\
    \frac{\mathbb{Z}[2 \rho_6, 2\rho_7]  }{ \mathbb{Z}[ 2\rho_6 + 2\rho_7 ] } &\vert& k =\mathbb{R}
    \\
    0 &\vert& k = \mathbb{Q}
  \end{array}
  \right.
  \end{array}
$$



\subsubsection{Binary dihedral group: $2 D_{10} \simeq \mathrm{Dic}_5$ }
 \label{BinaryDihedraD10}

\noindent Group name: $2 D_{10}$ (\cite{Dok2D10})

\noindent Group order: $\vert 2 D_{10}\vert = 20$

\medskip

\noindent Subgroups:

$$
\begin{array}{|c|c|c|c|c|}
\hline
\text{subgroup} & \text{order} & \text{cosets} & \text{conjugates} & \text{cyclic} \\
\hline
A & 20 & 1 & 1 &  \\
\hline
B & 10 & 2 & 1 & \checkmark \\
\hline
C & 5 & 4 & 1 & \checkmark \\
\hline
D & 4 & 5 & 5 & \checkmark \\
\hline
E & 2 & 10 & 1 & \checkmark \\
\hline
F & 1 & 20 & 1 & \checkmark \\
\hline
\end{array}
$$

\noindent Table of multiplicities:

$$
\begin{array}{c|cccccc}
 & A & B & C & D & E & F \\
\hline
A & 1 & 1 & 1 & 1 & 1 & 1 \\
B & 1 & 2 & 2 & 1 & 2 & 2 \\
C & 1 & 2 & 4 & 1 & 2 & 4 \\
D & 1 & 1 & 1 & 3 & 5 & 5 \\
E & 1 & 2 & 2 & 5 & 10 & 10 \\
F & 1 & 2 & 4 & 5 & 10 & 20 \\
\end{array}
$$

\noindent Upper triangular form:

$$
\begin{array}{c|cccccc}
 & A & B & C & D & E & F \\
\hline
V_1 & 1 & 1 & 1 & 1 & 1 & 1 \\
V_2 & . & 1 & 1 & . & 1 & 1 \\
V_3 & . & . & 2 & . & . & 2 \\
V_4 & . & . & . & 2 & 4 & 4 \\
V_5 & . & . & . & . & . & 8 \\
\end{array}
$$


\noindent Character table for image of $\beta$:

$$
\begin{array}{r|rrrrrrrr}
\mathrm{class} & 1 & 2 & 4A & 4B & 5 & 5 & 10 & 10 \\
\mathrm{size} & 1 & 1 & 5 & 5 & 2 & 2 & 2 & 2 \\
\hline
V_{1} & 1 & 1 & 1 & 1 & 1 & 1 & 1 & 1 \\
V_{2} & 1 & 1 & -1 & -1 & 1 & 1 & 1 & 1 \\
V_{3} & 2 & -2 & 0 & 0 & 2 & 2 & -2 & -2 \\
V_{4} & 4 & 4 & 0 & 0 & -1 & -1 & -1 & -1 \\
V_{5} & 8 & -8 & 0 & 0 & -2 & -2 & 2 & 2 \\
\end{array}
$$

\medskip

\noindent Character table of irreps \cite{Dok2D10}

over $\mathbb{C}$


$$
\begin{array}{c|cccccccc}
\mathrm{class} & 1 & 2 & 4A & 4B & 5A & 5B & 10A & 10B \\
\hline
\rho_{1} & 1 & 1 & 1 & 1 & 1 & 1 & 1 & 1 \\
\rho_{2} & 1 & 1 & -1 & -1 & 1 & 1 & 1 & 1 \\
\rho_{3} & 1 & -1 & -i & i & 1 & 1 & -1 & -1 \\
\rho_{4} & 1 & -1 & i & -i & 1 & 1 & -1 & -1 \\
\rho_{5} & 2 & 2 & 0 & 0 & \zeta_{5}^2+\zeta_{5}^3 & \zeta_{5}+\zeta_{5}^4 & \zeta_{5}^2+\zeta_{5}^3 & \zeta_{5}+\zeta_{5}^4 \\
\rho_{6} & 2 & 2 & 0 & 0 & \zeta_{5}+\zeta_{5}^4 & \zeta_{5}^2+\zeta_{5}^3 & \zeta_{5}+\zeta_{5}^4 & \zeta_{5}^2+\zeta_{5}^3 \\
\rho_{7} & 2 & -2 & 0 & 0 & \zeta_{5}^2+\zeta_{5}^3 & \zeta_{5}+\zeta_{5}^4 & -\zeta_{5}^2-\zeta_{5}^3 & -\zeta_{5}-\zeta_{5}^4 \\
\rho_{8} & 2 & -2 & 0 & 0 & \zeta_{5}+\zeta_{5}^4 & \zeta_{5}^2+\zeta_{5}^3 & -\zeta_{5}-\zeta_{5}^4 & -\zeta_{5}^2-\zeta_{5}^3 \\
\end{array}
$$

over $\mathbb{R}$

$$
\begin{array}{c|cccccccc}
\mathrm{class} & 1 & 2 & 4A & 4B & 5A & 5B & 10A & 10B \\
\hline
\rho_{1} & 1 & 1 & 1 & 1 & 1 & 1 & 1 & 1 \\
\rho_{2} & 1 & 1 & -1 & -1 & 1 & 1 & 1 & 1 \\
\rho_{3}+\rho_4 & 2 & -2 & 0 & 0 & 2 & 2 & -2 & -2 \\
\rho_{5} & 2 & 2 & 0 & 0 & \zeta_{5}^2+\zeta_{5}^3 & \zeta_{5}+\zeta_{5}^4 & \zeta_{5}^2+\zeta_{5}^3 & \zeta_{5}+\zeta_{5}^4 \\
\rho_{6} & 2 & 2 & 0 & 0 & \zeta_{5}+\zeta_{5}^4 & \zeta_{5}^2+\zeta_{5}^3 & \zeta_{5}+\zeta_{5}^4 & \zeta_{5}^2+\zeta_{5}^3 \\
2 \rho_{7} & 4 & -4 & 0 & 0 & 2\left(\zeta_{5}^2+\zeta_{5}^3\right) & 2\left(\zeta_{5}+\zeta_{5}^4\right) & -2\left(\zeta_{5}^2-\zeta_{5}^3\right) & -2\left(\zeta_{5}-\zeta_{5}^4\right)
 \\
2 \rho_{8} & 4 & -4 & 0 & 0 & 2\left(\zeta_{5}+\zeta_{5}^4\right) & 2\left(\zeta_{5}^2+\zeta_{5}^3\right) & -2\left(\zeta_{5}-\zeta_{5}^4\right) & -2\left(\zeta_{5}^2-\zeta_{5}^3\right) \\
\end{array}
$$

\medskip

Hence the cokernel of $\beta$ is:

$$
  \begin{array}{c|c}
  \begin{aligned}
    V_1 & = \rho_1
    \\
    V_2 & = \rho_2
    \\
    V_3 & = \rho_3 + \rho_4
    \\
    V_4 & = \rho_5 + \rho_6
    \\
    V_5 & = 2 \rho_7 + 2 \rho_8
  \end{aligned}
  &
  \mathrm{coker}\left(
    A(2 D_{10})
      \overset{\beta}{\to}
    R_k(2 D_{10})
  \right)
  =
  \left\{
  \begin{array}{cccc}
    \frac{ \mathbb{Z}[\rho_3, \rho_4, \rho_5, \rho_6, \rho_7, \rho_8] }{ \mathbb{Z}[ \rho_3 + \rho_4, \rho_5 + \rho_6, 2 \rho_7 + 2\rho_8 ] }
    &\vert& k = \mathbb{C}
    \\
    \frac{ \mathbb{Z}[\rho_3, \rho_4, \rho_5, \rho_6, 2\rho_7, 2\rho_8] }{ \mathbb{Z}[ \rho_3 + \rho_4, \rho_5 + \rho_6, 2 \rho_7 + 2\rho_8 ] }
    &\vert& k = \mathbb{R}
    \\
    0 &\vert& k = \mathbb{Q}
  \end{array}
  \right.
  \end{array}
$$


\subsubsection{Binary dihedral group: $2 D_{12} \simeq \mathrm{Dic}_6$ }
 \label{BinaryDihedraD12}

\noindent Group name: $2 D_{12}$ (\cite{Dok2D12})

\noindent Group order: $\vert 2 D_{12}\vert = 24$

\medskip

\noindent Subgroups:

$$
\begin{array}{|c|c|c|c|c|}
\hline
\text{subgroup} & \text{order} & \text{cosets} & \text{conjugates} & \text{cyclic} \\
\hline
A & 24 & 1 & 1 &  \\
\hline
B & 12 & 2 & 1 & \checkmark \\
\hline
C & 12 & 2 & 1 &  \\
\hline
D & 12 & 2 & 1 &  \\
\hline
E & 8 & 3 & 3 &  \\
\hline
F & 6 & 4 & 1 & \checkmark \\
\hline
G & 4 & 6 & 3 & \checkmark \\
\hline
H & 4 & 6 & 1 & \checkmark \\
\hline
I & 4 & 6 & 3 & \checkmark \\
\hline
J & 3 & 8 & 1 & \checkmark \\
\hline
K & 2 & 12 & 1 & \checkmark \\
\hline
L & 1 & 24 & 1 & \checkmark \\
\hline
\end{array}
$$

\noindent Table of multiplicities:

$$
\begin{array}{c|cccccccccccc}
 & A & B & C & D & E & F & G & H & I & J & K & L \\
\hline
A & 1 & 1 & 1 & 1 & 1 & 1 & 1 & 1 & 1 & 1 & 1 & 1 \\
B & 1 & 2 & 1 & 1 & 1 & 2 & 1 & 2 & 1 & 2 & 2 & 2 \\
C & 1 & 1 & 2 & 1 & 1 & 2 & 2 & 1 & 1 & 2 & 2 & 2 \\
D & 1 & 1 & 1 & 2 & 1 & 2 & 1 & 1 & 2 & 2 & 2 & 2 \\
E & 1 & 1 & 1 & 1 & 2 & 1 & 2 & 3 & 2 & 1 & 3 & 3 \\
F & 1 & 2 & 2 & 2 & 1 & 4 & 2 & 2 & 2 & 4 & 4 & 4 \\
G & 1 & 1 & 2 & 1 & 2 & 2 & 4 & 3 & 3 & 2 & 6 & 6 \\
H & 1 & 2 & 1 & 1 & 3 & 2 & 3 & 6 & 3 & 2 & 6 & 6 \\
I & 1 & 1 & 1 & 2 & 2 & 2 & 3 & 3 & 4 & 2 & 6 & 6 \\
J & 1 & 2 & 2 & 2 & 1 & 4 & 2 & 2 & 2 & 8 & 4 & 8 \\
K & 1 & 2 & 2 & 2 & 3 & 4 & 6 & 6 & 6 & 4 & 12 & 12 \\
L & 1 & 2 & 2 & 2 & 3 & 4 & 6 & 6 & 6 & 8 & 12 & 24 \\
\end{array}
$$

\noindent Upper triangular form:

$$
\begin{array}{c|cccccccccccc}
 & A & B & C & D & E & F & G & H & I & J & K & L \\
\hline
V_1 & 1 & 1 & 1 & 1 & 1 & 1 & 1 & 1 & 1 & 1 & 1 & 1 \\
V_2 & . & 1 & . & . & . & 1 & . & 1 & . & 1 & 1 & 1 \\ 
V_3 & . & . & 1 & . & . & 1 & 1 & . & . & 1 & 1 & 1 \\ 
V_4 & . & . & . & 1 & . & 1 & . & . & 1 & 1 & 1 & 1 \\ 
V_5 & . & . & . & . & 1 & . & 1 & 2 & 1 & . & 2 & 2 \\ 
V_6 & . & . & . & . & . & . & 1 & . & 1 & . & 2 & 2 \\ 
V_7 & . & . & . & . & . & . & . & . & . & 4 & . & 4 \\
V_8 & . & . & . & . & . & . & . & . & . & . & . & 8 \\
\end{array}
$$


\noindent Character table for image of $\beta$:

$$
\begin{array}{r|rrrrrrrrr}
\mathrm{class} & 1 & 2 & 3 & 4 & 4 & 4 & 6 & 12 & 12 \\
\mathrm{size} & 1 & 1 & 2 & 6 & 6 & 2 & 2 & 2 & 2 \\
\hline
V_{1} & 1 & 1 & 1 & 1 & 1 & 1 & 1 & 1 & 1 \\
V_{2} & 1 & 1 & 1 & -1 & -1 & 1 & 1 & 1 & 1 \\
V_{3} & 1 & 1 & 1 & -1 & 1 & -1 & 1 & -1 & -1 \\
V_{4} & 1 & 1 & 1 & 1 & -1 & -1 & 1 & -1 & -1 \\
V_{5} & 2 & 2 & -1 & 0 & 0 & 2 & -1 & -1 & -1 \\
V_{6} & 2 & 2 & -1 & 0 & 0 & -2 & -1 & 1 & 1 \\
V_{7} & 4 & -4 & 4 & 0 & 0 & 0 & -4 & 0 & 0 \\
V_{8} & 8 & -8 & -4 & 0 & 0 & 0 & 4 & 0 & 0 \\
\end{array}
$$

\medskip

\noindent Character table of irreps

over $\mathbb{C}$


$$
\begin{array}{c|ccccccccc}
\mathrm{class} & 1 & 2 & 3 & 4A & 4B & 4C & 6 & 12A & 12B \\
\hline
\rho_{1} & 1 & 1 & 1 & 1 & 1 & 1 & 1 & 1 & 1 \\
\rho_{2} & 1 & 1 & 1 & -1 & -1 & 1 & 1 & -1 & -1 \\
\rho_{3} & 1 & 1 & 1 & 1 & -1 & -1 & 1 & 1 & 1 \\
\rho_{4} & 1 & 1 & 1 & -1 & 1 & -1 & 1 & -1 & -1 \\
\rho_{5} & 2 & 2 & -1 & 2 & 0 & 0 & -1 & -1 & -1 \\
\rho_{6} & 2 & 2 & -1 & -2 & 0 & 0 & -1 & 1 & 1 \\
\rho_{7} & 2 & -2 & 2 & 0 & 0 & 0 & -2 & 0 & 0 \\
\rho_{8} & 2 & -2 & -1 & 0 & 0 & 0 & 1 & \sqrt{3} & -\sqrt{3} \\
\rho_{9} & 2 & -2 & -1 & 0 & 0 & 0 & 1 & -\sqrt{3} & \sqrt{3} \\
\end{array}
$$

over $\mathbb{R}$

$$
\begin{array}{c|ccccccccc}
\mathrm{class} & 1 & 2 & 3 & 4A & 4B & 4C & 6 & 12A & 12B \\
\hline
\rho_{1} & 1 & 1 & 1 & 1 & 1 & 1 & 1 & 1 & 1 \\
\rho_{2} & 1 & 1 & 1 & -1 & -1 & 1 & 1 & -1 & -1 \\
\rho_{3} & 1 & 1 & 1 & 1 & -1 & -1 & 1 & 1 & 1 \\
\rho_{4} & 1 & 1 & 1 & -1 & 1 & -1 & 1 & -1 & -1 \\
\rho_{5} & 2 & 2 & -1 & 2 & 0 & 0 & -1 & -1 & -1 \\
\rho_{6} & 2 & 2 & -1 & -2 & 0 & 0 & -1 & 1 & 1 \\
2\rho_{7} & 4 & -4 & 4 & 0 & 0 & 0 & -4 & 0 & 0 \\
2 \rho_{8} & 4 & -4 & -2 & 0 & 0 & 0 & 2 & 2\sqrt{3} & -2\sqrt{3} \\
2\rho_{9} & 4 & -4 & -2 & 0 & 0 & 0 & 2 & -2\sqrt{3} & 2\sqrt{3} \\
\end{array}
$$

\medskip

Hence the cokernel of $\beta$ is

$$
  \begin{array}{c|c}
  \begin{aligned}
    V_1 & = \rho_1
    \\
    V_2 & = \rho_3
    \\
    V_3 & = \rho_2
    \\
    V_4 & = \rho_4
    \\
    V_5 & = \rho_5
    \\
    V_6 & = \rho_6
    \\
    V_7 & = 2 \rho_7
    \\
    V_8 & = 2 \rho_8 + 2 \rho_9
  \end{aligned}
  &
  \mathrm{coker}\left(
    A(2 D_{12})
      \overset{\beta}{\to}
    R_k(2 D_{12})
  \right)
  \;=\;
  \left\{
  \begin{array}{cccc}
    \frac{\mathbb{Z}[ \rho_7, \rho_8, \rho_9 ] }{ \mathbb{Z}[ 2\rho_7, 2\rho_8 + 2 \rho_9 ] }
    &\vert& k = \mathbb{C}
    \\
    \frac{\mathbb{Z}[ 2\rho_8, 2\rho_9 ] }{ \mathbb{Z}[ 2\rho_8 + 2 \rho_9 ] }
    &\vert& k = \mathbb{R}
    \\
    0 &\vert& k = \mathbb{Q}
  \end{array}
  \right.
  \end{array}
$$


\subsubsection{Binary dihedral group: $ 2 D_{14} \simeq \mathrm{Dic}_7$ }
 \label{BinaryDihedraD14}

\noindent Group name: $2 D_{14}$

\noindent Group order: $\vert 2 D_{14}\vert = 28$

\medskip

\noindent Subgroups:

$$
\begin{array}{|c|c|c|c|c|}
\hline
\text{subgroup} & \text{order} & \text{cosets} & \text{conjugates} & \text{cyclic} \\
\hline
A & 28 & 1 & 1 &  \\
\hline
B & 14 & 2 & 1 & \checkmark \\
\hline
C & 7 & 4 & 1 & \checkmark \\
\hline
D & 4 & 7 & 7 & \checkmark \\
\hline
E & 2 & 14 & 1 & \checkmark \\
\hline
F & 1 & 28 & 1 & \checkmark \\
\hline
\end{array}
$$

\noindent Table of multiplicities:

$$
\begin{array}{c|cccccc}
 & A & B & C & D & E & F \\
\hline
A & 1 & 1 & 1 & 1 & 1 & 1 \\
B & 1 & 2 & 2 & 1 & 2 & 2 \\
C & 1 & 2 & 4 & 1 & 2 & 4 \\
D & 1 & 1 & 1 & 4 & 7 & 7 \\
E & 1 & 2 & 2 & 7 & 14 & 14 \\
F & 1 & 2 & 4 & 7 & 14 & 28 \\
\end{array}
$$

\noindent Upper triangular form:

$$
\begin{array}{c|cccccc}
 & A & B & C & D & E & F \\
\hline
V_1 & 1 & 1 & 1 & 1 & 1 & 1 \\
V_2 & . & 1 & 1 & . & 1 & 1 \\
V_3 & . & . & 2 & . & . & 2 \\
V_4 & . & . & . & 3 & 6 & 6 \\
V_5 & . & . & . & . & . & 12 \\
\end{array}
$$

\noindent Character table for image of $\beta$:

$$
\begin{array}{r|rrrrrrrrrr}
\mathrm{class} & 1 & 2 & 4 & 4 & 7 & 7 & 7 & 14 & 14 & 14 \\
\mathrm{size} & 1 & 1 & 7 & 7 & 2 & 2 & 2 & 2 & 2 & 2 \\
\hline
V_{1} & 1 & 1 & 1 & 1 & 1 & 1 & 1 & 1 & 1 & 1 \\
V_{2} & 1 & 1 & -1 & -1 & 1 & 1 & 1 & 1 & 1 & 1 \\
V_{3} & 2 & -2 & 0 & 0 & 2 & 2 & 2 & -2 & -2 & -2 \\
V_{4} & 6 & 6 & 0 & 0 & -1 & -1 & -1 & -1 & -1 & -1 \\
V_{5} & 12 & -12 & 0 & 0 & -2 & -2 & -2 & 2 & 2 & 2
\end{array}
$$


$$
\begin{array}{c|cccccccccc}
\mathrm{class} & 1 & 2 & 4A & 4B & 7A & 7B & 7C & 14A & 14B & 14C \\
\hline
\rho_{1} & 1 & 1 & 1 & 1 & 1 & 1 & 1 & 1 & 1 & 1 \\
\rho_{2} & 1 & 1 & -1 & -1 & 1 & 1 & 1 & 1 & 1 & 1 \\
\rho_{3} & 1 & -1 & -i & i & 1 & 1 & 1 & -1 & -1 & -1 \\
\rho_{4} & 1 & -1 & i & -i & 1 & 1 & 1 & -1 & -1 & -1 \\
\rho_{5} & 2 & 2 & 0 & 0 & \zeta_{7}^3+\zeta_{7}^4 & \zeta_{7}+\zeta_{7}^6 & \zeta_{7}^2+\zeta_{7}^5 & \zeta_{7}^3+\zeta_{7}^4 & \zeta_{7}+\zeta_{7}^6 & \zeta_{7}^2+\zeta_{7}^5 \\
\rho_{6} & 2 & 2 & 0 & 0 & \zeta_{7}^2+\zeta_{7}^5 & \zeta_{7}^3+\zeta_{7}^4 & \zeta_{7}+\zeta_{7}^6 & \zeta_{7}^2+\zeta_{7}^5 & \zeta_{7}^3+\zeta_{7}^4 & \zeta_{7}+\zeta_{7}^6 \\
\rho_{7} & 2 & 2 & 0 & 0 & \zeta_{7}+\zeta_{7}^6 & \zeta_{7}^2+\zeta_{7}^5 & \zeta_{7}^3+\zeta_{7}^4 & \zeta_{7}+\zeta_{7}^6 & \zeta_{7}^2+\zeta_{7}^5 & \zeta_{7}^3+\zeta_{7}^4 \\
\rho_{8} & 2 & -2 & 0 & 0 & \zeta_{7}^3+\zeta_{7}^4 & \zeta_{7}+\zeta_{7}^6 & \zeta_{7}^2+\zeta_{7}^5 & -\zeta_{7}^3-\zeta_{7}^4 & -\zeta_{7}-\zeta_{7}^6 & -\zeta_{7}^2-\zeta_{7}^5 \\
\rho_{9} & 2 & -2 & 0 & 0 & \zeta_{7}^2+\zeta_{7}^5 & \zeta_{7}^3+\zeta_{7}^4 & \zeta_{7}+\zeta_{7}^6 & -\zeta_{7}^2-\zeta_{7}^5 & -\zeta_{7}^3-\zeta_{7}^4 & -\zeta_{7}-\zeta_{7}^6 \\
\rho_{10} & 2 & -2 & 0 & 0 & \zeta_{7}+\zeta_{7}^6 & \zeta_{7}^2+\zeta_{7}^5 & \zeta_{7}^3+\zeta_{7}^4 & -\zeta_{7}-\zeta_{7}^6 & -\zeta_{7}^2-\zeta_{7}^5 & -\zeta_{7}^3-\zeta_{7}^4 \\
\end{array}
$$



\subsubsection{Binary dihedral group: $ 2 D_{16} \simeq \mathrm{Dic}_8 \simeq Q_{32} $}
 \label{BinaryDihedraQ32}

\noindent Group name: $2 D_{16}$

\noindent Group order: $\vert 2 D_{16}\vert = 32 $

\medskip

\noindent Subgroups:

$$
\begin{array}{|c|c|c|c|c|}
\hline
\text{subgroup} & \text{order} & \text{cosets} & \text{conjugates} & \text{cyclic} \\
\hline
A & 32 & 1 & 1 &  \\
\hline
B & 16 & 2 & 1 &  \\
\hline
C & 16 & 2 & 1 & \checkmark \\
\hline
D & 16 & 2 & 1 &  \\
\hline
E & 8 & 4 & 1 & \checkmark \\
\hline
F & 8 & 4 & 2 &  \\
\hline
G & 8 & 4 & 2 &  \\
\hline
H & 4 & 8 & 4 & \checkmark \\
\hline
I & 4 & 8 & 4 & \checkmark \\
\hline
J & 4 & 8 & 1 & \checkmark \\
\hline
K & 2 & 16 & 1 & \checkmark \\
\hline
L & 1 & 32 & 1 & \checkmark \\
\hline
\end{array}
$$

\noindent Table of multiplicities:

$$
\begin{array}{c|cccccccccccc}
 & A & B & C & D & E & F & G & H & I & J & K & L \\
\hline
A & 1 & 1 & 1 & 1 & 1 & 1 & 1 & 1 & 1 & 1 & 1 & 1 \\
B & 1 & 2 & 1 & 1 & 2 & 1 & 2 & 2 & 1 & 2 & 2 & 2 \\
C & 1 & 1 & 2 & 1 & 2 & 1 & 1 & 1 & 1 & 2 & 2 & 2 \\
D & 1 & 1 & 1 & 2 & 2 & 2 & 1 & 1 & 2 & 2 & 2 & 2 \\
E & 1 & 2 & 2 & 2 & 4 & 2 & 2 & 2 & 2 & 4 & 4 & 4 \\
F & 1 & 1 & 1 & 2 & 2 & 3 & 2 & 2 & 3 & 4 & 4 & 4 \\
G & 1 & 2 & 1 & 1 & 2 & 2 & 3 & 3 & 2 & 4 & 4 & 4 \\
H & 1 & 2 & 1 & 1 & 2 & 2 & 3 & 5 & 4 & 4 & 8 & 8 \\
I & 1 & 1 & 1 & 2 & 2 & 3 & 2 & 4 & 5 & 4 & 8 & 8 \\
J & 1 & 2 & 2 & 2 & 4 & 4 & 4 & 4 & 4 & 8 & 8 & 8 \\
K & 1 & 2 & 2 & 2 & 4 & 4 & 4 & 8 & 8 & 8 & 16 & 16 \\
L & 1 & 2 & 2 & 2 & 4 & 4 & 4 & 8 & 8 & 8 & 16 & 32 \\
\end{array}
$$

\noindent Upper triangular form:

$$
\begin{array}{c|cccccccccccc}
 & A & B & C & D & E & F & G & H & I & J & K & L \\
\hline
V_1 & 1 & 1 & 1 & 1 & 1 & 1 & 1 & 1 & 1 & 1 & 1 & 1 \\
V_2  & . & 1 & . & . & 1 & . & 1 & 1 & . & 1 & 1 & 1 \\ 
V_3  & . & . & 1 & . & 1 & . & . & . & . & 1 & 1 & 1 \\ 
V_4  & . & . & . & 1 & 1 & 1 & . & . & 1 & 1 & 1 & 1 \\ 
V_5 & . & . & . & . & . & 1 & 1 & 1 & 1 & 2 & 2 & 2 \\
V_6 & . & . & . & . & . & . & . & 2 & 2 & . & 4 & 4 \\
V_7 & . & . & . & . & . & . & . & . & . & . & . & 16 \\
\end{array}
$$


\noindent Character table for image of $\beta$:

$$
\begin{array}{r|rrrrrrrrrrr}
\mathrm{class} & 1 & 2 & 4 & 4 & 4 & 8 & 8 & 16 & 16 & 16 & 16 \\
\mathrm{size} & 1 & 1 & 8 & 8 & 2 & 2 & 2 & 2 & 2 & 2 & 2 \\
\hline
V_{1} & 1 & 1 & 1 & 1 & 1 & 1 & 1 & 1 & 1 & 1 & 1 \\
V_{2} & 1 & 1 & 1 & -1 & 1 & 1 & 1 & -1 & -1 & -1 & -1 \\
V_{3} & 1 & 1 & -1 & -1 & 1 & 1 & 1 & 1 & 1 & 1 & 1 \\
V_{4} & 1 & 1 & -1 & 1 & 1 & 1 & 1 & -1 & -1 & -1 & -1 \\
V_{5} & 2 & 2 & 0 & 0 & 2 & -2 & -2 & 0 & 0 & 0 & 0 \\
V_{6} & 4 & 4 & 0 & 0 & -4 & 0 & 0 & 0 & 0 & 0 & 0 \\
V_{7} & 16 & -16 & 0 & 0 & 0 & 0 & 0 & 0 & 0 & 0 & 0 \\
\end{array}
$$


$$
\begin{array}{c|ccccccccccc}
\mathrm{class} & 1 & 2 & 4A & 4B & 4C & 8A & 8B & 16A & 16B & 16C & 16D \\
\hline
\rho_{1} & 1 & 1 & 1 & 1 & 1 & 1 & 1 & 1 & 1 & 1 & 1 \\
\rho_{2} & 1 & 1 & 1 & 1 & -1 & 1 & 1 & -1 & -1 & -1 & -1 \\
\rho_{3} & 1 & 1 & 1 & -1 & 1 & 1 & 1 & -1 & -1 & -1 & -1 \\
\rho_{4} & 1 & 1 & 1 & -1 & -1 & 1 & 1 & 1 & 1 & 1 & 1 \\
\rho_{5} & 2 & 2 & 2 & 0 & 0 & -2 & -2 & 0 & 0 & 0 & 0 \\
\rho_{6} & 2 & 2 & -2 & 0 & 0 & 0 & 0 & -\sqrt{2} & -\sqrt{2} & \sqrt{2} & \sqrt{2} \\
\rho_{7} & 2 & 2 & -2 & 0 & 0 & 0 & 0 & \sqrt{2} & \sqrt{2} & -\sqrt{2} & -\sqrt{2} \\
\rho_{8} & 2 & -2 & 0 & 0 & 0 & -\sqrt{2} & \sqrt{2} & \zeta_{16}-\zeta_{16}^7 & -\zeta_{16}+\zeta_{16}^7 & -\zeta_{16}^3+\zeta_{16}^5 & \zeta_{16}^3-\zeta_{16}^5 \\
\rho_{9} & 2 & -2 & 0 & 0 & 0 & -\sqrt{2} & \sqrt{2} & -\zeta_{16}+\zeta_{16}^7 & \zeta_{16}-\zeta_{16}^7 & \zeta_{16}^3-\zeta_{16}^5 & -\zeta_{16}^3+\zeta_{16}^5 \\
\rho_{10} & 2 & -2 & 0 & 0 & 0 & \sqrt{2} & -\sqrt{2} & -\zeta_{16}^3+\zeta_{16}^5 & \zeta_{16}^3-\zeta_{16}^5 & -\zeta_{16}+\zeta_{16}^7 & \zeta_{16}-\zeta_{16}^7 \\
\rho_{11} & 2 & -2 & 0 & 0 & 0 & \sqrt{2} & -\sqrt{2} & \zeta_{16}^3-\zeta_{16}^5 & -\zeta_{16}^3+\zeta_{16}^5 & \zeta_{16}-\zeta_{16}^7 & -\zeta_{16}+\zeta_{16}^7 \\
\end{array}
$$


\subsection{Binary exceptional groups: $2T$, $2I$, $2O$}

We discuss the three exceptional cases in the E-series
of the finite subgroups of $\mathrm{SU}(2)$ (from Prop. \ref{ADEGroups}).

\subsubsection{Binary tetrahedral group: $2T = \mathrm{SL}(2, 3).$}
  \label{BinaryTetrahedral}

\noindent Group name: $2T$ (\cite{Dok2T})

\noindent Group order: $\vert 2T\vert = 24$

\medskip

\noindent Subgroups:

$$
\begin{array}{|c|c|c|c|c|}
\hline
\text{subgroup} & \text{order} & \text{cosets} & \text{conjugates} & \text{cyclic} \\
\hline
A & 24 & 1 & 1 &  \\
\hline
B & 8 & 3 & 1 &  \\
\hline
C & 6 & 4 & 4 & \checkmark \\
\hline
D & 4 & 6 & 3 & \checkmark \\
\hline
E & 3 & 8 & 4 & \checkmark \\
\hline
F & 2 & 12 & 1 & \checkmark \\
\hline
G & 1 & 24 & 1 & \checkmark \\
\hline
\end{array}
$$


\noindent Table of multiplicities:

$$
\begin{array}{c|ccccccc}
 & A & B & C & D & E & F & G \\
\hline
A & 1 & 1 & 1 & 1 & 1 & 1 & 1 \\
B & 1 & 3 & 1 & 3 & 1 & 3 & 3 \\
C & 1 & 1 & 2 & 2 & 2 & 4 & 4 \\
D & 1 & 3 & 2 & 4 & 2 & 6 & 6 \\
E & 1 & 1 & 2 & 2 & 4 & 4 & 8 \\
F & 1 & 3 & 4 & 6 & 4 & 12 & 12 \\
G & 1 & 3 & 4 & 6 & 8 & 12 & 24 \\
\end{array}
$$

\noindent Upper triangular form:

$$
\begin{array}{c|ccccccc}
 & A & B & C & D & E & F & G \\
\hline
V_1 & 1 & 1 & 1 & 1 & 1 & 1 & 1 \\
V_2 & . & 2 & . & 2 & . & 2 & 2 \\
V_3 & . & . & 1 & 1 & 1 & 3 & 3 \\
V_4 & . & . & . & . & 2 & . & 4 \\
V_5 & . & . & . & . & . & . & 4 \\
\end{array}
$$


\noindent Character table for image of $\beta$:

$$
\begin{array}{r|rrrrrrr}
\mathrm{class} & 1 & 2 & 3A & 3B & 4 & 6A & 6B \\
\mathrm{size} & 1 & 1 & 4 & 4 & 6 & 4 & 4 \\
\hline
          V_1 & 1 & 1 & 1 & 1 & 1 & 1 & 1 \\
          V_2 & 2 & 2 & -1 & -1 & 2 & -1 & -1 \\
          V_3 & 3 & 3 & 0 & 0 & -1 & 0 & 0 \\
          V_4 & 4 & -4 & 1 & 1 & 0 & -1 & -1 \\
          V_5 & 4 & -4 & -2 & -2 & 0 & 2 & 2 \\
\end{array}
$$

\medskip

\noindent Character table of irreps \cite{Dok2T, Mon2T}:

over $\mathbb{C}$:

\begin{center}
\begin{tabular}{cc|ccccccc}
 \multicolumn{2}{c}{  } &  \multicolumn{7}{c}{ conjugacy class }
 \\
 \multicolumn{2}{c|}{} & $1$ & 2 & 3A & 3B & 4 & 6A & 6B
 \\
 \cline{2-9}
 \multirow{7}{*}{ \raisebox{-50pt}{\begin{rotate}{90} irred. repr. \end{rotate}} \hspace{-20pt} }
 & $\rho_1$ & $\phantom{+}1$ & $\phantom{+}1$ & $\phantom{+}1$ & $\phantom{+}1$ & $\phantom{+}1$ & $\phantom{+}1$ & $\phantom{+}1$ \\
 & $\rho_2$ & $\phantom{+}1$ & $\phantom{+}1$ & $\phantom{+}\zeta_3^2$ & $\phantom{+}\zeta_3$ & $\phantom{+}1$ & $\phantom{+}\zeta_3^2$ &  $\phantom{+}\zeta_3$ \\
 & $\rho_2^\ast$ & $\phantom{+}1$ & $\phantom{+}1$ & $\phantom{+}\zeta_3$ & $\phantom{+}\zeta_3^2$ & $\phantom{+}1$ & $\phantom{+}\zeta_3$ &  $\phantom{+}\zeta_3^2$ \\
 & $\rho_3$ & $\phantom{+}3$ & $\phantom{+}3$ & $\phantom{+}0$ & $\phantom{+}0$ & $-1$ & $\phantom{+}0$ & $\phantom{+}0$ \\
 & $\rho_4$ & $\phantom{+}2$ & $-2$ & $-\zeta_3^2$ & $-\zeta_3$ & $\phantom{+}0$ & $\phantom{+}\zeta_3^2$ & $\phantom{+}\zeta_3$ \\
 & $\rho_4^\ast$ & $\phantom{+}2$ & $-2$ & $-\zeta_3$ & $-\zeta_3^2$ & $\phantom{+}0$ & $\phantom{+}\zeta_3$ & $\phantom{+}\zeta_3^2$ \\
 & $\rho_5$ & $\phantom{+}2$ & $-2$ & $-1$ & $-1$ & $\phantom{+}0$ & $\phantom{+}1$ & $\phantom{+}1$
\end{tabular}
\end{center}

over $\mathbb{R}$

\begin{center}
\begin{tabular}{cc|ccccccc}
 \multicolumn{2}{c}{  } &  \multicolumn{7}{c}{ conjugacy class }
 \\
 \multicolumn{2}{c|}{} & $1$ & 2 & 3A & 3B & 4 & 6A & 6B
 \\
 \cline{2-9}
 \multirow{5}{*}{ \raisebox{-50pt}{\begin{rotate}{90} irred. repr. \end{rotate}} \hspace{-20pt} }
 & $\rho_1$ & $\phantom{+}1$ & $\phantom{+}1$ & $\phantom{+}1$ & $\phantom{+}1$ & $\phantom{+}1$ & $\phantom{+}1$ & $\phantom{+}1$ \\
 & $\rho_2 + \rho_2^\ast$ & $\phantom{+}2$ & $\phantom{+}2$ & $-1$ & $-1$ & $\phantom{+}2$ & $-1$ &  $-1$ \\
 & $\rho_3$ & $\phantom{+}3$ & $\phantom{+}3$ & $\phantom{+}0$ & $\phantom{+}0$ & $-1$ & $\phantom{+}0$ & $\phantom{+}0$ \\
 & $\rho_4 + \rho_4^\ast$ & $\phantom{+}4$ & $-4$ & $\phantom{+}1$ & $\phantom{+}1$ & $\phantom{+}0$ & $-1$ & $-1$ \\
 & $2 \rho_5$ & $\phantom{+}4$ & $-4$ & $-2$ & $-2$ & $\phantom{+}0$ & $\phantom{+}2$ & $\phantom{+}2$
\end{tabular}
\end{center}

\noindent where $\zeta_3 := \tfrac{1}{2}\left(-1 + \sqrt{3} i \right)$

\medskip

\noindent Hence the cokernel of $\beta$ is:
$$
  \mathrm{coker}
  \left(
    A(2T) \overset{\beta}{\to} R_k(2T)
  \right)
  \;\simeq\;
  \left\{
  \begin{array}{ccc}
    \frac{ \mathbb{Z}[ \rho_2, \rho_2^\ast, \rho_4, \rho_4^\ast, \rho_8 ] }{ \mathbb{Z}[ \rho_2 + \rho_2^\ast, \rho_4 + \rho_4^\ast, 2 \rho_8 ] }
    &\vert& k = \mathbb{C}
    \\
    0 &\vert& k = \mathbb{R}
    \\
    0 &\vert& k = \mathbb{Q}
  \end{array}
  \right.
$$



\subsubsection{Binary octahedral group: $2O \simeq \mathrm{CSU}(2, \mathbb{F}_3).$}
 \label{BinaryOctahedralGroup}

\noindent Group name: $2O$ (\cite{Dok2O})

\noindent Group order: ${\vert 2 O\vert} = 48$

\medskip

\noindent Subgroups:

$$
\begin{array}{|c|c|c|c|c|}
\hline
\text{subgroup} & \text{order} & \text{cosets} & \text{conjugates} & \text{cyclic} \\
\hline
A & 48 & 1 & 1 &  \\
\hline
B & 24 & 2 & 1 &  \\
\hline
C & 16 & 3 & 3 &  \\
\hline
D & 12 & 4 & 4 &  \\
\hline
E & 8 & 6 & 3 &  \\
\hline
F & 8 & 6 & 1 &  \\
\hline
G & 8 & 6 & 3 & \checkmark \\
\hline
H & 6 & 8 & 4 & \checkmark \\
\hline
I & 4 & 12 & 6 & \checkmark \\
\hline
J & 4 & 12 & 3 & \checkmark \\
\hline
K & 3 & 16 & 4 & \checkmark \\
\hline
L & 2 & 24 & 1 & \checkmark \\
\hline
M & 1 & 48 & 1 & \checkmark \\
\hline
\end{array}
$$

\noindent Table of multiplicities:

$$
\begin{array}{c|ccccccccccccc}
 & A & B & C & D & E & F & G & H & I & J & K & L & M \\
\hline
A & 1 & 1 & 1 & 1 & 1 & 1 & 1 & 1 & 1 & 1 & 1 & 1 & 1 \\
B & 1 & 2 & 1 & 1 & 1 & 2 & 1 & 2 & 1 & 2 & 2 & 2 & 2 \\
C & 1 & 1 & 2 & 1 & 2 & 3 & 2 & 1 & 2 & 3 & 1 & 3 & 3 \\
D & 1 & 1 & 1 & 2 & 2 & 1 & 1 & 2 & 3 & 2 & 2 & 4 & 4 \\
E & 1 & 1 & 2 & 2 & 3 & 3 & 2 & 2 & 4 & 4 & 2 & 6 & 6 \\
F & 1 & 2 & 3 & 1 & 3 & 6 & 3 & 2 & 3 & 6 & 2 & 6 & 6 \\
G & 1 & 1 & 2 & 1 & 2 & 3 & 3 & 2 & 3 & 4 & 2 & 6 & 6 \\
H & 1 & 2 & 1 & 2 & 2 & 2 & 2 & 4 & 4 & 4 & 4 & 8 & 8 \\
I & 1 & 1 & 2 & 3 & 4 & 3 & 3 & 4 & 7 & 6 & 4 & 12 & 12 \\
J & 1 & 2 & 3 & 2 & 4 & 6 & 4 & 4 & 6 & 8 & 4 & 12 & 12 \\
K & 1 & 2 & 1 & 2 & 2 & 2 & 2 & 4 & 4 & 4 & 8 & 8 & 16 \\
L & 1 & 2 & 3 & 4 & 6 & 6 & 6 & 8 & 12 & 12 & 8 & 24 & 24 \\
M & 1 & 2 & 3 & 4 & 6 & 6 & 6 & 8 & 12 & 12 & 16 & 24 & 48 \\
\end{array}
$$

\noindent Upper triangular form:

$$
\begin{array}{c|ccccccccccccc}
 & A & B & C & D & E & F & G & H & I & J & K & L & M \\
\hline
V_1 & 1 & 1 & 1 & 1 & 1 & 1 & 1 & 1 & 1 & 1 & 1 & 1 & 1 \\
V_2 & . & 1 & . & . & . & 1 & . & 1 & . & 1 & 1 & 1 & 1 \\
V_3 & . & . & 1 & . & 1 & 2 & 1 & . & 1 & 2 & . & 2 & 2 \\
V_4 & . & . & . & 1 & 1 & . & . & 1 & 2 & 1 & 1 & 3 & 3 \\ 
V_5 & . & . & . & . & . & . & 1 & 1 & 1 & 1 & 1 & 3 & 3 \\ 
V_6 & . & . & . & . & . & . & . & . & . & . & 4 & . & 8 \\ 
V_7 & . & . & . & . & . & . & . & . & . & . & . & . & 8 \\ 
\end{array}
$$


\noindent Character table for image of $\beta$:

$$
\begin{array}{r|rrrrrrrr}
\mathrm{class} & 1 & 2 & 3 & 4A & 4B & 6 & 8A & 8B \\
\mathrm{size} & 1 & 1 & 8 & 6 & 12 & 8 & 6 & 6 \\
\hline
V_{1} & 1 & 1 & 1 & 1 & 1 & 1 & 1 & 1 \\
V_{2} & 1 & 1 & 1 & 1 & -1 & 1 & -1 & -1 \\
V_{3} & 2 & 2 & -1 & 2 & 0 & -1 & 0 & 0 \\
V_{4} & 3 & 3 & 0 & -1 & 1 & 0 & -1 & -1 \\
V_{5} & 3 & 3 & 0 & -1 & -1 & 0 & 1 & 1 \\
V_{6} & 8 & -8 & 2 & 0 & 0 & -2 & 0 & 0 \\
V_{7} & 8 & -8 & -4 & 0 & 0 & 4 & 0 & 0 \\
\end{array}
$$

\medskip

\noindent Character table of irreps \cite{Dok2O, Mon2O}

over $\mathbb{C}$

\begin{center}
\begin{tabular}{cc|cccccccc}
 \multicolumn{2}{c}{  } &  \multicolumn{8}{c}{ conjugacy class }
 \\
 \multicolumn{2}{c|}{} & 1 & 2 & 3 & 4A & 4B & 6 & 8A & 8B
 \\
 \cline{2-10}
 \multirow{8}{*}{ \raisebox{-50pt}{\begin{rotate}{90} irred. repr. \end{rotate}} \hspace{-20pt} }
 & $\rho_1$ & $\phantom{+}1$ & $\phantom{+}1$ & $\phantom{+}1$ & $\phantom{+}1$ & $\phantom{+}1$ & $\phantom{+}1$ & $\phantom{+}1$ & $\phantom{+}1$ \\
 & $\rho_2$ & $\phantom{+}1$ & $\phantom{+}1$ & $\phantom{+}1$& $\phantom{+}1$ & $-1$ & $\phantom{+}1$ & $-1$ & $-1$ \\
 & $\rho_3$ & $\phantom{+}2$ & $\phantom{+}2$ & $-1$ & $\phantom{+}2$ & $\phantom{+}0$ & $-1$ & $\phantom{+}0$ & $\phantom{+}0$ \\
 & $\rho_4$ & $\phantom{+}3$ & $\phantom{+}3$ & $\phantom{+}0$ & $-1$ & $-1$ & $\phantom{+}0$ & $\phantom{+}1$ & $\phantom{+}1$ \\
 & $\rho_5$ & $\phantom{+}3$ & $\phantom{+}3$ & $\phantom{+}0$ & $-1$ & $\phantom{+}1$ & $\phantom{+}0$ & $-1$ & $-1$ \\
 & $\rho_6$ & $\phantom{+}2$ & $-2$ & $-1$ & $\phantom{+}0$ & $\phantom{+}0$ & $\phantom{+}1$ & $\sqrt{2}$ & $-\sqrt{2}$ \\
 & $\rho_7$ & $\phantom{+}2$ & $-2$ & $-1$ & $\phantom{+}0$ & $\phantom{+}0$ & $\phantom{+}1$ & $-\sqrt{2}$ & $\sqrt{2}$ \\
 & $\rho_8$ & $\phantom{+}4$ & $-4$ & $\phantom{+}1$ & $\phantom{+}0$ & $\phantom{+}0$ & $-1$ & $\phantom{+}0$ & $\phantom{+}0$
\end{tabular}
\end{center}

over $\mathbb{R}$

\begin{center}
\begin{tabular}{cc|cccccccc}
 \multicolumn{2}{c}{  } &  \multicolumn{8}{c}{ conjugacy class }
 \\
 \multicolumn{2}{c|}{} & 1 & 2 & 3 & 4A & 4B & 6 & 8A & 8B
 \\
 \cline{2-10}
 \multirow{8}{*}{ \raisebox{-50pt}{\begin{rotate}{90} irred. repr. \end{rotate}} \hspace{-20pt} }
 & $\rho_1$ & $\phantom{+}1$ & $\phantom{+}1$ & $\phantom{+}1$ & $\phantom{+}1$ & $\phantom{+}1$ & $\phantom{+}1$ & $\phantom{+}1$ & $\phantom{+}1$ \\
 & $\rho_2$ & $\phantom{+}1$ & $\phantom{+}1$ & $\phantom{+}1$& $\phantom{+}1$ & $-1$ & $\phantom{+}1$ & $-1$ & $-1$ \\
 & $\rho_3$ & $\phantom{+}2$ & $\phantom{+}2$ & $-1$ & $\phantom{+}2$ & $\phantom{+}0$ & $-1$ & $\phantom{+}0$ & $\phantom{+}0$ \\
 & $\rho_4$ & $\phantom{+}3$ & $\phantom{+}3$ & $\phantom{+}0$ & $-1$ & $-1$ & $\phantom{+}0$ & $\phantom{+}1$ & $\phantom{+}1$ \\
 & $\rho_5$ & $\phantom{+}3$ & $\phantom{+}3$ & $\phantom{+}0$ & $-1$ & $\phantom{+}1$ & $\phantom{+}0$ & $-1$ & $-1$ \\
 & $2 \rho_6$ & $\phantom{+}4$ & $-4$ & $-2$ & $\phantom{+}0$ & $\phantom{+}0$ & $\phantom{+}2$ & $2\sqrt{2}$ & $-2\sqrt{2}$ \\
 & $2 \rho_7$ & $\phantom{+}4$ & $-4$ & $-2$ & $\phantom{+}0$ & $\phantom{+}0$ & $\phantom{+}2$ & $-2 \sqrt{2}$ & $2\sqrt{2}$ \\
 & $2 \rho_8$ & $\phantom{+}8$ & $-8$ & $\phantom{+}2$ & $\phantom{+}0$ & $\phantom{+}0$ & $-2$ & $\phantom{+}0$ & $\phantom{+}0$
\end{tabular}
\end{center}

\medskip

\noindent Hence the cokernel of $\beta$ is:

$$
  \mathrm{coker}
  \left(
    A(2O)
      \overset{\beta}{\to}
    R_k(2O)
  \right)
  \;\simeq\;
  \left\{
  \begin{array}{cccc}
    \frac{ \mathbb{Z}[\rho_6, \rho_7, \rho_8] }{ \mathbb{Z}[2\rho_6 + 2\rho_7, 2\rho_8] } &\vert& k= \mathbb{C};
    \\
    \frac{\mathbb{Z}[2\rho_6, 2\rho_7]}{ \mathbb{Z}[ 2\rho_6 + 2\rho_7 ] } &\vert& k = \mathbb{R};
    \\
    0 &\vert& k = \mathbb{Q}
  \end{array}
  \right.
$$


\subsubsection{Binary icosahedral group: $2I \simeq \mathrm{SL}(2, 5).$}
 \label{BinaryIcosahedralGroup}

\noindent Group name: $2I$ (\cite{Dok2I})

\noindent Group order: $\vert 2 I\vert = 120$

\medskip

\noindent Subgroups:

$$
\begin{array}{|c|c|c|c|c|}
\hline
\text{subgroup} & \text{order} & \text{cosets} & \text{conjugates} & \text{cyclic} \\
\hline
A & 120 & 1 & 1 &  \\
\hline
B & 24 & 5 & 5 &  \\
\hline
C & 20 & 6 & 6 &  \\
\hline
D & 12 & 10 & 10 &  \\
\hline
E & 10 & 12 & 6 & \checkmark \\
\hline
F & 8 & 15 & 5 &  \\
\hline
G & 6 & 20 & 10 & \checkmark \\
\hline
H & 5 & 24 & 6 & \checkmark \\
\hline
I & 4 & 30 & 15 & \checkmark \\
\hline
J & 3 & 40 & 10 & \checkmark \\
\hline
K & 2 & 60 & 1 & \checkmark \\
\hline
L & 1 & 120 & 1 & \checkmark \\
\hline
\end{array}
$$

\noindent Table of multiplicities:

$$
\begin{array}{c|cccccccccccc}
 & A & B & C & D & E & F & G & H & I & J & K & L \\
\hline
A & 1 & 1 & 1 & 1 & 1 & 1 & 1 & 1 & 1 & 1 & 1 & 1 \\
B & 1 & 2 & 1 & 2 & 1 & 2 & 3 & 1 & 3 & 3 & 5 & 5 \\
C & 1 & 1 & 2 & 2 & 2 & 3 & 2 & 2 & 4 & 2 & 6 & 6 \\
D & 1 & 2 & 2 & 3 & 2 & 4 & 4 & 2 & 6 & 4 & 10 & 10 \\
E & 1 & 1 & 2 & 2 & 4 & 3 & 4 & 4 & 6 & 4 & 12 & 12 \\
F & 1 & 2 & 3 & 4 & 3 & 6 & 5 & 3 & 9 & 5 & 15 & 15 \\
G & 1 & 3 & 2 & 4 & 4 & 5 & 8 & 4 & 10 & 8 & 20 & 20 \\
H & 1 & 1 & 2 & 2 & 4 & 3 & 4 & 8 & 6 & 8 & 12 & 24 \\
I & 1 & 3 & 4 & 6 & 6 & 9 & 10 & 6 & 16 & 10 & 30 & 30 \\
J & 1 & 3 & 2 & 4 & 4 & 5 & 8 & 8 & 10 & 16 & 20 & 40 \\
K & 1 & 5 & 6 & 10 & 12 & 15 & 20 & 12 & 30 & 20 & 60 & 60 \\
L & 1 & 5 & 6 & 10 & 12 & 15 & 20 & 24 & 30 & 40 & 60 & 120 \\
\end{array}
$$

\noindent Upper triangular form:

$$
\begin{array}{c|cccccccccccc}
 & A & B & C & D & E & F & G & H & I & J & K & L \\
\hline
V_1 & 1 & 1 & 1 & 1 & 1 & 1 & 1 & 1 & 1 & 1 & 1 & 1 \\
V_2 & . & 1 & . & 1 & . & 1 & 2 & . & 2 & 2 & 4 & 4 \\
V_3 & . & . & 1 & 1 & 1 & 2 & 1 & 1 & 3 & 1 & 5 & 5 \\
V_4 & . & . & . & . & 2 & . & 2 & 2 & 2 & 2 & 6 & 6 \\
V_5 & . & . & . & . & . & . & . & 4 & . & 4 & . & 12 \\
V_6 & . & . & . & . & . & . & . & . & . & 4 & . & 8 \\
V_7 & . & . & . & . & . & . & . & . & . & . & . & 8 \\
\end{array}
$$


\noindent Character table for image of $\beta$:

$$
\begin{array}{r|rrrrrrrrr}
\mathrm{class} & 1 & 2 & 3 & 4 & 5A & 5B & 6 & 10A & 10B \\
\mathrm{size} & 1 & 1 & 20 & 30 & 12 & 12 & 20 & 12 & 12 \\
\hline
V_1 & 1 & 1 & 1 & 1 & 1 & 1 & 1 & 1 & 1 \\
V_2 & 4 & 4 & 1 & 0 & -1 & -1 & 1 & -1 & -1 \\
V_3 & 5 & 5 & -1 & 1 & 0 & 0 & -1 & 0 & 0 \\
V_4 & 6 & 6 & 0 & -2 & 1 & 1 & 0 & 1 & 1 \\
V_5 & 12 & -12 & 0 & 0 & 2 & 2 & 0 & -2 & -2 \\
V_6 & 8 & -8 & 2 & 0 & -2 & -2 & -2 & 2 & 2 \\
V_7 & 8 & -8 & -4 & 0 & -2 & -2 & 4 & 2 & 2 \\
\end{array}
$$

\medskip

\noindent Character table of irreps \cite{Dok2I}

over $\mathbb{C}$

\begin{center}
\begin{tabular}{cc|ccccccccc}
 \multicolumn{2}{c}{  } &  \multicolumn{8}{c}{ conjugacy class }
 \\
 \multicolumn{2}{c|}{} & 1 & 2 & 3 & 4 & 5A & 5B & 6 & 10A & 10B
 \\
 \cline{2-11}
 \multirow{8}{*}{ \raisebox{-50pt}{\begin{rotate}{90} irred. repr. \end{rotate}} \hspace{-20pt} }
 & $\rho_1$ & $\phantom{+}1$ & $\phantom{+}1$ & $\phantom{+}1$ & $\phantom{+}1$ & $\phantom{+}1$ & $\phantom{+}1$ & $\phantom{+}1$ & $\phantom{+}1$ & $\phantom{+}1$
 \\
 & $\rho_2$ &  $\phantom{+}2$ &  $-2$ & $-1$ & $\phantom{+}0$ & $\phi-1$ & $-\phi$ & $\phantom{+}1$ & $\phantom{+}\phi$ & $1-\phi$
 \\
 & $\rho_3$ & $\phantom{+}2$ & $-2$ & $-1$ & $\phantom{+}0$ & $-\phi$ & $\phi-1$ & $\phantom{+}1$ & $1-\phi$ & $\phantom{+}\phi$
 \\
 & $\rho_4$ & $\phantom{+}3$ & $\phantom{+}3$ & $\phantom{+}0$ & $-1$ & $1-\phi$ & $\phantom{+}\phi$ & $\phantom{+}0$ & $\phantom{+}\phi$ & $1 - \phi$
 \\
 & $\rho_5$ & $\phantom{+}3$ & $\phantom{+}3$ & $\phantom{+}0$ & $-1$ & $\phantom{+}\phi$ & $1-\phi$ & $\phantom{+}0$ & $1-\phi$ & $\phantom{+}\phi$
 \\
 & $\rho_6$ & $\phantom{+}4$ & $\phantom{+}4$ & $\phantom{+}1$ & $\phantom{+}0$ & $-1$ & $-1$ & $1$ & $-1$ & $-1$
 \\
 & $\rho_7$ & $\phantom{+}4$ & $-4$ & $\phantom{+}1$ & $\phantom{+}0$ & $-1$ & $-1$ & $-1$ & $\phantom{+}1$& $\phantom{+}1$
 \\
 & $\rho_8$ & $\phantom{+}5$ & $-5$ & $-1$ & $\phantom{+}1$ & $\phantom{+}0$ & $\phantom{+}0$ & $-1$ & $\phantom{+}0$ & $\phantom{+}0$
 \\
 & $\rho_9$ & $\phantom{+}6$ & $-6$ & $\phantom{+}0$ & $\phantom{+}0$ & $\phantom{+}1$ & $\phantom{+}1$ & $\phantom{+}0$ & $-1$ & $-1$
\end{tabular}
\end{center}

over $\mathbb{R}$

\begin{center}
\begin{tabular}{cc|ccccccccc}
 \multicolumn{2}{c}{  } &  \multicolumn{8}{c}{ conjugacy class }
 \\
 \multicolumn{2}{c|}{} & 1 & 2 & 3 & 4 & 5A & 5B & 6 & 10A & 10B
 \\
 \cline{2-11}
 \multirow{8}{*}{ \raisebox{-50pt}{\begin{rotate}{90} irred. repr. \end{rotate}} \hspace{-20pt} }
 & $\rho_1$ & $\phantom{+}1$ & $\phantom{+}1$ & $\phantom{+}1$ & $\phantom{+}1$ & $\phantom{+}1$ & $\phantom{+}1$ & $\phantom{+}1$ & $\phantom{+}1$ & $\phantom{+}1$
 \\
 & $2\rho_2$ &  $\phantom{+}4$ &  $-4$ & $-2$ & $\phantom{+}0$ & $2(\phi-1)$ & $-2\phi$ & $\phantom{+}2$ & $2\phi$ & $2(1-\phi)$
 \\
 & $2\rho_3$ & $\phantom{+}4$ & $-4$ & $-2$ & $\phantom{+}0$ & $-2\phi$ & $2(\phi-1)$ & $\phantom{+}2$ & $2(1-\phi)$ & $2\phi$
 \\
 & $\rho_4$ & $\phantom{+}3$ & $\phantom{+}3$ & $\phantom{+}0$ & $-1$ & $1-\phi$ & $\phantom{+}\phi$ & $\phantom{+}0$ & $\phantom{+}\phi$ & $1 - \phi$
 \\
 & $\rho_5$ & $\phantom{+}3$ & $\phantom{+}3$ & $\phantom{+}0$ & $-1$ & $\phantom{+}\phi$ & $1-\phi$ & $\phantom{+}0$ & $1-\phi$ & $\phantom{+}\phi$
 \\
 & $\rho_6$ & $\phantom{+}4$ & $\phantom{+}4$ & $\phantom{+}1$ & $\phantom{+}0$ & $-1$ & $-1$ & $1$ & $-1$ & $-1$
 \\
 & $2\rho_7$ & $\phantom{+}8$ & $-8$ & $\phantom{+}2$ & $\phantom{+}0$ & $-2$ & $-2$ & $-2$ & $\phantom{+}2$& $\phantom{+}2$
 \\
 & $\rho_8$ & $\phantom{+}5$ & $5$ & $-1$ & $\phantom{+}1$ & $\phantom{+}0$ & $\phantom{+}0$ & $-1$ & $\phantom{+}0$ & $\phantom{+}0$
 \\
 & $2\rho_9$ & $\phantom{+}12$ & $-12$ & $\phantom{+}0$ & $\phantom{+}0$ & $\phantom{+}2$ & $\phantom{+}2$ & $\phantom{+}0$ & $-2$ & $-2$
\end{tabular}
\end{center}

\medskip where $\phi := \tfrac{1}{2}\left( 1 + \sqrt{5}\right)$ is the golden ratio.

\noindent Hence the cokernel of $\beta$ is

$$
  \begin{array}{c|c}
  \begin{aligned}
    V_1 & = \rho_1
    \\
    V_2 & = \rho_6
    \\
    V_3 & =\rho_8
    \\
    V_4 & = \rho_4 + \rho_5
    \\
    V_5 & = 2 \rho_9
    \\
    V_6 & = 2\rho_7
    \\
    V_7 & = 2 \rho_2 + 2 \rho_3
  \end{aligned}
  &
  \mathrm{coker}\left( \beta_k\right)
  \;=\;
  \left\{
  \begin{array}{ccc}
    \frac{ \mathbb{Z}[ \rho_2, \rho_3, \rho_4, \rho_5, \rho_7, \rho_9 ] }{ \mathbb{Z}[ 2\rho_2 + 2\rho_3, \rho_4 + \rho_5, 2\rho_7, 2\rho_9 ]  }
    &\vert& k =\mathbb{C}
    \\
    \\
    \frac{ \mathbb{Z}[ 2\rho_2, 2\rho_3, \rho_4, \rho_5   ] }{ \mathbb{Z}[ 2\rho_2 + 2\rho_3, \rho_4 + \rho_5  ]  }
    &\vert& k =\mathbb{R}
    \\
    \\
    0
    &\vert& k =\mathbb{Q}
  \end{array}
  \right.
  \end{array}
$$



\subsubsection{The general linear group: $\mathrm{GL}(2, \mathbb{F}_3)$}
  \label{TheGeneralLinear23}

\medskip

Group name: $\mathrm{GL}(2,\mathbb{F}_3)$ (\cite{DokGL23}\footnote{
The representation theory of this group $\mathrm{GL}(2,\mathbb{F}_3)$ is deceptively similar to that of
the binary octahedral group $2O \simeq \mathrm{CSU}(2,\mathbb{F}_3)$
(discussed Sect. \ref{BinaryOctahedralGroup}):
Both have the same character table over $\mathbb{C}$, the only difference being in the Schur indices, hence in the real
character table. In fact, several online databases of character tables had misidentified the two groups, which became apparent when our
computation of the image of $\beta$ revealed real representations of $2O$ that contradicted available character tables. We are indebted
to James Montaldi for patiently double-checking computations with us and to Tim Dokchitser for swiftly recognizing and fixing the issue with the databases.
})

\noindent Group order: $\vert \mathrm{GL}(2,\mathbb{F}_3)\vert = 48$

\medskip

\noindent Subgroups:

$$
\begin{array}{|c|c|c|c|c|}
\hline
\text{subgroup} & \text{order} & \text{cosets} & \text{conjugates} & \text{cyclic} \\
\hline
A & 48 & 1 & 1 &  \\
\hline
B & 24 & 2 & 1 &  \\
\hline
C & 16 & 3 & 3 &  \\
\hline
D & 12 & 4 & 4 &  \\
\hline
E & 8 & 6 & 3 & \checkmark \\
\hline
F & 8 & 6 & 1 &  \\
\hline
G & 8 & 6 & 3 &  \\
\hline
H & 6 & 8 & 4 &  \\
\hline
I & 6 & 8 & 4 &  \\
\hline
J & 6 & 8 & 4 & \checkmark \\
\hline
K & 4 & 12 & 6 &  \\
\hline
L & 4 & 12 & 3 & \checkmark \\
\hline
M & 3 & 16 & 4 & \checkmark \\
\hline
N & 2 & 24 & 1 & \checkmark \\
\hline
P & 2 & 24 & 12 & \checkmark \\
\hline
Q & 1 & 48 & 1 & \checkmark \\
\hline
\end{array}
$$


\noindent Table of multiplicities:

$$
\begin{array}{c|cccccccccccccccc}
 & A & B & C & D & E & F & G & H & I & J & K & L & M & N & P & Q \\
\hline
A & 1 & 1 & 1 & 1 & 1 & 1 & 1 & 1 & 1 & 1 & 1 & 1 & 1 & 1 & 1 & 1 \\
B & 1 & 2 & 1 & 1 & 1 & 2 & 1 & 2 & 1 & 1 & 1 & 2 & 2 & 1 & 2 & 2 \\
C & 1 & 1 & 2 & 1 & 2 & 3 & 2 & 1 & 1 & 1 & 2 & 3 & 1 & 2 & 3 & 3 \\
D & 1 & 1 & 1 & 2 & 1 & 1 & 2 & 2 & 2 & 2 & 3 & 2 & 2 & 3 & 4 & 4 \\
E & 1 & 1 & 2 & 1 & 3 & 3 & 2 & 2 & 1 & 1 & 3 & 4 & 2 & 3 & 6 & 6 \\
F & 1 & 2 & 3 & 1 & 3 & 6 & 3 & 2 & 1 & 1 & 3 & 6 & 2 & 3 & 6 & 6 \\
G & 1 & 1 & 2 & 2 & 2 & 3 & 3 & 2 & 2 & 2 & 4 & 4 & 2 & 4 & 6 & 6 \\
H & 1 & 2 & 1 & 2 & 2 & 2 & 2 & 4 & 2 & 2 & 4 & 4 & 4 & 4 & 8 & 8 \\
I & 1 & 1 & 1 & 2 & 1 & 1 & 2 & 2 & 3 & 3 & 3 & 2 & 4 & 5 & 4 & 8 \\
J & 1 & 1 & 1 & 2 & 1 & 1 & 2 & 2 & 3 & 3 & 3 & 2 & 4 & 5 & 4 & 8 \\
K & 1 & 1 & 2 & 3 & 3 & 3 & 4 & 4 & 3 & 3 & 7 & 6 & 4 & 7 & 12 & 12 \\
L & 1 & 2 & 3 & 2 & 4 & 6 & 4 & 4 & 2 & 2 & 6 & 8 & 4 & 6 & 12 & 12 \\
M & 1 & 2 & 1 & 2 & 2 & 2 & 2 & 4 & 4 & 4 & 4 & 4 & 8 & 8 & 8 & 16 \\
N & 1 & 1 & 2 & 3 & 3 & 3 & 4 & 4 & 5 & 5 & 7 & 6 & 8 & 13 & 12 & 24 \\
P & 1 & 2 & 3 & 4 & 6 & 6 & 6 & 8 & 4 & 4 & 12 & 12 & 8 & 12 & 24 & 24 \\
Q & 1 & 2 & 3 & 4 & 6 & 6 & 6 & 8 & 8 & 8 & 12 & 12 & 16 & 24 & 24 & 48 \\
\end{array}
$$

\noindent Upper triangular form:

$$
\begin{array}{c|cccccccccccccccc}
 & A & B & C & D & E & F & G & H & I & J & K & L & M & N & P & Q \\
\hline
V_1 & 1 & 1 & 1 & 1 & 1 & 1 & 1 & 1 & 1 & 1 & 1 & 1 & 1 & 1 & 1 & 1 \\
V_2 & . & 1 & . & . & . & 1 & . & 1 & . & . & . & 1 & 1 & . & 1 & 1 \\
V_3 & . & . & 1 & . & 1 & 2 & 1 & . & . & . & 1 & 2 & . & 1 & 2 & 2 \\
V_4  & . & . & . & 1 & . & . & 1 & 1 & 1 & 1 & 2 & 1 & 1 & 2 & 3 & 3 \\ 
V_5  & . & . & . & . & 1 & . & . & 1 & . & . & 1 & 1 & 1 & 1 & 3 & 3 \\ 
V_6  & . & . & . & . & . & . & . & . & 1 & 1 & . & . & 2 & 2 & . & 4 \\ 
V_7  & . & . & . & . & . & . & . & . & . & . & . & . & . & 2 & . & 4 \\ 
\end{array}
$$


\noindent Character table for image of $\beta$:

$$
\begin{array}{r|rrrrrrrr}
\mathrm{class} & 1 & 2A & 2B & 3 & 4 & 6 & 8A & 8B \\
\mathrm{size}  & 1 & 1  & 12 & 8 & 6 & 8 & 6 & 6 \\
\hline
           V_1 & 1 &  1 &  1 & 1 & 1 & 1 & 1 & 1 \\
           V_2 & 1 &  1 & -1 & 1 & 1 & 1 & -1 & -1 \\
           V_3 & 2 &  2 &  0 & -1 & 2 & -1 & 0 & 0 \\
           V_4 & 3 &  3 &  1 & 0 & -1 & 0 & -1 & -1 \\
           V_5 & 3 &  3 & -1 & 0 & -1 & 0 & 1 & 1 \\
           V_6 & 4 & -4 &  0 & 1 & 0 & -1 & 0 & 0 \\
           V_7 & 4 & -4 &  0 & -2 & 0 & 2 & 0 & 0 \\
\end{array}
$$

\medskip

\noindent Character table of irreps \cite{DokGL23}:

over $\mathbb{C}$

\begin{center}
\begin{tabular}{cc|cccccccc}
 \multicolumn{2}{c}{  } &  \multicolumn{8}{c}{ conjugacy class }
 \\
 \multicolumn{2}{c|}{} & 1 & 2A & 2B & 3 & 4 & 6 & 8A & 8B
 \\
 \cline{2-10}
 \multirow{8}{*}{ \raisebox{-50pt}{\begin{rotate}{90} irred. repr. \end{rotate}} \hspace{-20pt} }
 & $\rho_1$ & $\phantom{+}1$ & $\phantom{+}1$ & $\phantom{+}1$ & $\phantom{+}1$ & $\phantom{+}1$ & $\phantom{+}1$ & $\phantom{+}1$ & $\phantom{+}1$ \\
 & $\rho_2$ & $\phantom{+}1$ & $\phantom{+}1$ & $-1$ & $\phantom{+}1$& $\phantom{+}1$ & $\phantom{+}1$ & $-1$ & $-1$ \\
 & $\rho_3$ & $\phantom{+}2$ & $\phantom{+}2$ & $\phantom{+}0$ & $-1$ & $\phantom{+}2$ & $-1$ & $\phantom{+}0$ & $\phantom{+}0$ \\
 & $\rho_4$ & $\phantom{+}3$ & $\phantom{+}3$ & $-1$ & $\phantom{+}0$ & $-1$ & $\phantom{+}0$ & $\phantom{+}1$ & $\phantom{+}1$ \\
 & $\rho_5$ & $\phantom{+}3$ & $\phantom{+}3$ & $\phantom{+}1$ & $\phantom{+}0$ & $-1$ & $\phantom{+}0$ & $-1$ & $-1$ \\
 & $\rho_6$ & $\phantom{+}2$ & $-2$ & $\phantom{+}0$ & $-1$ & $\phantom{+}0$ & $\phantom{+}1$ & $-\sqrt{2}$ & $\sqrt{2}$ \\
 & $\rho_7$ & $\phantom{+}2$ & $-2$ & $\phantom{+}0$ & $-1$ & $\phantom{+}0$ & $\phantom{+}1$ & $\sqrt{2}$ & $-\sqrt{2}$ \\
 & $\rho_8$ & $\phantom{+}4$ & $-4$ & $\phantom{+}0$ & $\phantom{+}1$ & $\phantom{+}0$ & $-1$ & $\phantom{+}0$ & $\phantom{+}0$
\end{tabular}
\end{center}

over $\mathbb{R}$

\begin{center}
\begin{tabular}{cc|cccccccc}
 \multicolumn{2}{c}{  } &  \multicolumn{8}{c}{ conjugacy class }
 \\
 \multicolumn{2}{c|}{} & 1 & 2A & 2B & 3 & 4 & 6 & 8A & 8B
 \\
 \cline{2-10}
 \multirow{7}{*}{ \raisebox{-50pt}{\begin{rotate}{90} irred. repr. \end{rotate}} \hspace{-20pt} }
 & $\rho_1$ & $\phantom{+}1$ & $\phantom{+}1$ & $\phantom{+}1$ & $\phantom{+}1$ & $\phantom{+}1$ & $\phantom{+}1$ & $\phantom{+}1$ & $\phantom{+}1$ \\
 & $\rho_2$ & $\phantom{+}1$ & $\phantom{+}1$ & $-1$ & $\phantom{+}1$& $\phantom{+}1$ & $\phantom{+}1$ & $-1$ & $-1$ \\
 & $\rho_3$ & $\phantom{+}2$ & $\phantom{+}2$ & $\phantom{+}0$ & $-1$ & $\phantom{+}2$ & $-1$ & $\phantom{+}0$ & $\phantom{+}0$ \\
 & $\rho_4$ & $\phantom{+}3$ & $\phantom{+}3$ & $-1$ & $\phantom{+}0$ & $-1$ & $\phantom{+}0$ & $\phantom{+}1$ & $\phantom{+}1$ \\
 & $\rho_5$ & $\phantom{+}3$ & $\phantom{+}3$ & $\phantom{+}1$ & $\phantom{+}0$ & $-1$ & $\phantom{+}0$ & $-1$ & $-1$ \\
 & $\rho_6 + \rho_7$ & $\phantom{+}4$ & $-4$ & $\phantom{+}0$ & $-2$ & $\phantom{+}0$ & $\phantom{+}2$ & $\phantom{+}0$ & $\phantom{+}0$ \\
 & $\rho_8$ & $\phantom{+}4$ & $-4$ & $\phantom{+}0$ & $\phantom{+}1$ & $\phantom{+}0$ & $-1$ & $\phantom{+}0$ & $\phantom{+}0$
\end{tabular}
\end{center}

\medskip

\noindent Hence the cokernel of $\beta$ is
$$
  \mathrm{coker}
  \left(
    A\big(\mathrm{GL}(2,\mathbb{F}_3)\big)
    \overset{\beta_k}{\to}
    R_k\big(\mathrm{GL}(2,\mathbb{F}_3)\big)
  \right)
  \;\simeq\;
  \left\{
  \begin{array}{ccc}
     \frac{ \mathbb{Z}[ \rho_6, \rho_7 ] }{ \mathbb{Z}[\rho_6 + \rho_7] } &\vert& k = \mathbb{C}
     \\
     $0$ &\vert& k = \mathbb{R}
     \\
     $0$ &\vert& k =\mathbb{Q}
  \end{array}
  \right.
$$



\newpage

\appendix

\section{Background on platonic groups and categorical algebra}
\label{Background}

\vspace{-4pt} 
To be reasonably self-contained, we briefly collect some background material
which will be directly useful for our discussion and for connecting the physical
and the mathematical sides of the arguments.

\vspace{-8pt} 
\subsection{The Platonic groups}
 \label{ThePlatonicGroups}

We are interested in the ADE groups as the main groups appearing in orbifolds
in string theory. There are two variants, ones that are subgroups of the orthogonal
group $\mathrm{SO}(3)$ and ones that are subgroups of the unitary group
$\mathrm{SU}(2)$. Our focus will be on the latter.

\begin{prop}[ADE-classification of the finite rotation groups \cite{Klein1884}]
  \label{ADEGroups}
  The finite subgroups of $\mathrm{SU}(2)$
  are given, up to conjugacy, by the following classification (where $n \in \mathbb{N}$):
 {\small
  \begin{center}
  \begin{tabular}{|c||c|c||c|c||}
    \hline
    \begin{tabular}{c}
      \bf Dynkin
      \\
      \bf label
    \end{tabular}
    &
    \begin{tabular}{c}
      \bf Finite
      \\
      \bf subgroup
      \\
      \bf of $\mathrm{SO}(3)$
    \end{tabular}
    &
     \begin{tabular}{c}
       \bf Name of
      \\
      \bf group
    \end{tabular}
    &
    \begin{tabular}{c}
      \bf Finite
      \\
      \bf subgroup
      \\
      \bf of $\mathrm{SU}(2)$
    \end{tabular}
    &
     \begin{tabular}{c}
       \bf Name of
      \\
      \bf group
    \end{tabular}
    \\
    \hline
    \hline
    $\mathbb{A}_{\mathrlap{n \geq 1}}$
    &
    $C_{\mathrlap{n+1}}$ &   Cyclic
    &
    $\phantom{2}C_{\mathrlap{n+1}}$ &   Cyclic
    \\
    \hline
    $\mathbb{D}_{\mathrlap{n \geq 4}}$
    &
    $\mathrm{D}_{\mathrlap{2(n-2)}}$ & Dihedral
    &
    $2\mathrm{D}_{\mathrlap{2(n-2)}}$ & Binary dihedral
    \\
    \hline
    $\mathbb{E}_{\mathrlap{6}}$
    & $\mathrm{T}$ & Tetrahedral
    & $2\mathrm{T}$ & Binary tetrahedral
    \\
    \hline
    $\mathbb{E}_{\mathrlap{7}}$
    & $\mathrm{O}$ & Octahedral
    & $2\mathrm{O}$ & Binary octahedral
    \\
    \hline
    $\mathbb{E}_{\mathrlap{8}}$
    &
    $\mathrm{I}$ & Icosahedral
    &
    $2\mathrm{I}$ & Binary icosahedral
    \\
    \hline
  \end{tabular}
  \end{center}
  }
\end{prop}
Full proof for finite subgroups of $\mathrm{SL}(2,\mathbb{C})$ is in \cite{MBD1916}, recalled in \cite[Sec. 2]{Serrano14}.
Full proof for $\mathrm{SO}(3)$ is spelled out in \cite[Theorem 11]{Rees05}; from this the case of $\mathrm{SU}(2)$
is given in \cite[Theorem 4]{Keenan03}.
See also \cite[Chapter 1,2]{Lindh2018} for an elementary treatment.

\medskip
Our discussion will require details on the structure of these finite
groups and how they relate to each other. Hence we highlight the following
pattern. The full subgroup lattice of ${\rm SU}(2)$ under the three exceptional finite subgroups from Prop. \ref{ADEGroups}
(using the subgroup lattices from \cite{Dok2T, Dok2O, Dok2I}).

\vspace{-.1cm}

\begin{center}
\scalebox{.55}{
$
  \xymatrix@C=4.5pt@R=9pt{
    &&&&
    &&&&&&&&
  \mathrm{SU}(2)
    \\
    \\
    \\
    \\
    &&&&
    2O
    \ar@{^{(}->}[uuuurrrrrrrr]
    &&&&&&&& &&&&&&&&
    2I
    \ar@{^{(}->}[uuuullllllll]
    \\
    &&&&
    &&&&&&&& &&&&&&&&
    \\
    \fbox{\large \bf E}
    &&&&
    &&&&&&&& &&&&&&&&
    \\
    &&&&
    &&&&&&&& &&&&&&&&
    \\
    &&&&
    &&&&
    &&&&
    2T
    \ar@{^{(}->}[uuuullllllll]
    \ar@{^{(}->}[uuuurrrrrrrr]
    \\
    &&&&
    &&&&&&&& &&&&&&&&
    \\
    &&&&
    &&&&&&&& &&&&&&&&
    \\
    &&&&
    &&&&
    2 D_8
    \ar@{^{(}->}[uuuuuuullll]
    &&&&
    &&&&&&&&
    \\
    &&&&
    2 D_6
    \ar@{^{(}->}[uuuuuuuu]
    &&&&
    &&&&
    &
    2 D_{10}
    \ar@{^{(}->}[uuuuuuuurrrrrrr]|<<<<{\phantom{AA \atop AA}}
    & &&
    {\phantom{2D_8}}
    &&&&
    2 D_6
    \ar@{^{(}->}[uuuuuuuu]
    \\
    \fbox{\large \bf D}
    &&&&
    \\
    &&&&
    &&&&&&
    2 D_4
    \ar@{^{(}->}[uuull]|>>>>>>>>>>>>>>>>{\phantom{AA \atop AA}}
    &&
    2D_4
    \ar@{^{(}->}[uuuuuu]
    \ar@{^{(}->}[uuullll]|>>>>>>>>>>>>>>{\phantom{AA}}
    \\
    &&&&
    \\
    &&&&
    &&&& &&&&
    &&
    C_{10}
    \ar@{^{(}->}[uuuul]
    &&&&
    \\
    &&&&
    &&&&
    C_8
    \ar@{^{(}->}[uuuuuu]|>>>>>>>>>>>>>>>>>>>>>>>>>>>>>{\phantom{{AA \atop AA} \atop A}}
    \\
    &&&&
    &&&&
    \\
    &&&&
    C_6
    \ar@{^{(}->}[uuuuuuu]
    \ar@{^{(}->}[uuuuuuuuuuurrrrrrrr]
    &&&&
    &&&&&&&&&&&&
    C_6
    \ar@{^{(}->}[uuuuuuu]
    \ar@{^{(}->}[uuuuuuuuuuullllllll]
    \\
    \fbox{\large \bf A}
    &&&&
    &&&& &&&& &&&& C_5
    \ar@{^{(}->}[uuuull]
    \\
    &&&&
    &&&&&
    C_4
    \ar@{^{(}->}[uuuuuuur]
    \ar@{^{(}->}[uuuuuuuuulllll]|>>>>>>>>>>>>>>>>>>>>>>>>>>>>>>>>>{ \phantom{AA \atop AA} }
    &&&
    C_4
    \ar@{^{(}->}[uuuuuuu]
    \ar@{^{(}->}[uuuuuuull]
    \ar@{^{(}->}[uuuuuuuuurrrrrrrr]|>>>>>>>>>>>>>>>>>>>>>>>>>>>>>>>>>>>>>>>>>>>>>>>>>>>>>>>>>>>>>>{\phantom{{AAA \atop AAA} \atop AA} }|>>>>>>>>>>>>>>>>>>>>>>>>>>>>>>>>>>>>>>>>>>>>>>>{\phantom{AA \atop AA}}|>>>>>>>>>>>>>>>>>>>>>>>>>>>>>{ \phantom{AA \atop AA } }
    \ar@{^{(}->}[uuuullll]|>>>>>>>>>>>>>>{\phantom{AA \atop AA}}
    \ar@{^{(}->}[uuuuuuuuur]
    \\
    &&&&
    \\
    &&&&
    C_3
    \ar@{^{(}->}[uuuu]
    &&&&&&&&&&&&&&&&
    C_3
    \ar@{^{(}->}[uuuu]
    \\
    &&&&
    \\
    &&&&
    &&&&&&&&
    C_2
    \ar@{^{(}->}[uuuu]
    \ar@{^{(}->}[uuuuuullllllll]
    \ar@{^{(}->}[uuuuuurrrrrrrr]
    \ar@{^{(}->}[uuuulll]
    \ar@{^{(}->}[uuuuuuuuurr]
    \\
    &&&&
    \\
    &&&&
    \\
    &&&&
    \\
    &&&&
    &&&&&&&&
    \ast
    \ar@{^{(}->}[uuuu]
    \ar@{^{(}->}[uuuuuullllllll]
    \ar@{^{(}->}[uuuuuurrrrrrrr]
    \ar@{^{(}->}[uuuuuuuuurrrr]|>>>>>>>>>>>>>>{\phantom{{AAAA \atop AAAA} \atop AAA}}
  }
$}
\end{center}

\subsection{Categorical algebra}
\label{CategoricalAlgebra}
For ease of reference, we briefly recall the concept
of internal homs in compact closed categories \cite[vol 2, 6.1]{Bor},
also \cite{Mc}.

\medskip

The categories $G \mathrm{Set}^{\mathrm{fin}}$ (Def. \ref{GSets}) and  $G \mathrm{Rep}^{\mathrm{fin}}_k$ (Def. \ref{GReps})
enjoy various properties that are directly analogous to
familiar properties of the category $\mathrm{Vect}^{\mathrm{fin}}_k$ of finite-dimensional $k$-vector spaces. The language of categorical algebra allows to
make these analogies explicit, and such that one may reason uniformly in all three cases.

\medskip
For $V,W \in \mathrm{Vect}^{\mathrm{fin}}_k$ two vector spaces, the set $\mathrm{Hom}(V,W)$ of linear maps (``homomorphisms'') $V \to W$ between them
becomes itself canonically a vector space, by pointwise multiplication with $k$ and pointwise addition of values of functions.
When we want to emphasize that we regard the set $\mathrm{Hom}(V,W)$ as equipped with this vector space structure, we write $[V,W]$ for it.

\medskip
One way to make this vector space of linear functions $[V,W]$ more explicit is to consider the \emph{dual} vector space $V^\ast$.
With that in hand we have a canonical linear isomorphism
$$
  \xymatrix@R=-10pt{
    W \otimes V^\ast  \ar[rr]^{\simeq} && [V,W]
    \\
    ( \vert w\rangle \otimes \langle v \vert  )
    \ar[rr] &&
    \big( \vert q \rangle \mapsto \vert w\rangle \cdot \underset{\in k}{\underbrace{\langle v , q \rangle}}  \,\big)
  }
$$

\vspace{-4mm}
\noindent which identifies the tensor product space of $V^\ast$ with $W$ as the vector space of linear maps
from $V$ to $W$. Here
$$
  \langle -,- \rangle
  \;\colon\;
  V^\ast \otimes V
  \longrightarrow
  k
  = \mathbf{1}
$$
denotes the pairing map that defines the dual vector space, and we denote elements of $V$ by $\vert q \rangle \in V$
and those of the dual vector space by $\langle v \vert \in V^\ast $, just so as to bring out the pattern better.
Note that this pairing map is itself $k$-linear.
Hence if we regard the ground field $k$ as the canonical 1-dimensional $k$-vector $\mathbf{1}$, as indicated, then this is
actually a morphism in $\mathrm{Vect}_k^{\mathrm{fin}}$.
There is also a closely related linear going the other way around:
$$
  \xymatrix@R=-2pt{
    \mathbf{1}
    \ar[rr]^-{ \eta }
    &&
    [V,V]
    \ar@{}[r]|\simeq
    &
    V \otimes V^\ast
    \\
    1 \ar@{|->}[rr] && \mathrm{id}_V
  }
$$
which, under the above identification, sends any element $c \in k$ to the linear map from $V$ to itself that is given by multiplication with $c$.
One readily checks that these two functions make the following triangles commute
$$
  \xymatrix{
    & V \otimes V^\ast \otimes V
    \ar[dr]^-{ \mathrm{id} \otimes \langle -,-\rangle  }
    \\
    V
    \ar[ur]^-{ \eta \otimes \mathrm{id} }
    \ar[rr]^{\mathrm{id}}
    &&
    V
  }
  \phantom{AAAAA}
  \xymatrix{
    & V^\ast \otimes V \otimes V^\ast
    \ar[dr]^-{ \langle -,-\rangle \otimes \mathrm{id} }
    \\
    V^\ast
    \ar[ur]^-{ \mathrm{id} \otimes \eta }
    \ar[rr]^{ \mathrm{id} }
    &&
    V^\ast
  }
$$
whence called the \emph{triangle identities}.

\medskip
The quickest way to convince oneself that this indeed holds is to choose linear identifications
$V \simeq \mathbb{R}^n$ and $W \simeq \mathbb{R}^m$, which means to choose \emph{linear bases}. This in turn induces a canonical identification
$V^\ast \simeq (\mathbb{R}^n)^\ast \simeq \mathbb{R}^n$ (the \emph{dual linear basis}), hence a linear identification
$$
  [V,W]
  \simeq
  V^\ast \otimes W \;\simeq\; \mathbb{R}^n \otimes \mathbb{R}^m \simeq \mathbb{R}^{n \times m}
  \simeq
  \mathrm{Mat}_{n \times m}(k)
$$
of the vector space of linear maps $V \to W$ with the vector space of $n \times m$ matrices.

\medskip
The description of dual vector spaces in terms of pairing and co-pairing maps satisfying triangle identities, as above,
turns out to be \emph{equivalent} to the traditional definition. It may seem more involved than the direct definition,
but it has the great advantage that it makes sense without any actual reference to the nature of vector spaces:
all that is needed to speak of \emph{dual objects} is the analogue of the tensor product $\otimes$.

\medskip
Categorical algebra shows that the \emph{triangle identities} guarantee
that $V^\ast \otimes W$ behaves like an ``internalized'' version of the Hom-set. The same applies to the tensor product of representations
used in  \cref{TheAlgorithm}.

\medskip
\medskip
\noindent {\bf Acknowledgements.}
We thank Tim Dokchitser, James Dolan, James Montaldi and Todd Trimble for discussion.
Our algorithm is inspired by the note \cite{Trimble09}, which in turn
goes back to private communication with James Dolan.


\begin{thebibliography}{10}

\bibitem[Ad74]{Adams74}
J. F. Adams,
{\it Stable homotopy and generalized homology},
University of Chicago Press, 1974, \newline
[\href{https://www.press.uchicago.edu/ucp/books/book/chicago/S/bo21302708.html}{\tt ucp:bo21302708}].

\vspace{-3mm}
\bibitem[BDS00]{BachasDouglasSchweigert00}
C. Bachas, M. Douglas, C. Schweigert,
{\it Flux Stabilization of D-branes}, JHEP {\bf 0005} (2000), 048, \newline
[\href{https://arxiv.org/abs/hep-th/0003037}{\tt arXiv:hep-th/0003037}].


\vspace{-3mm}
\bibitem[BD16]{BartelDokchitser16}
A. Bartel, T. Dokchitser, {\it Rational representations and permutation representations of finite groups},
Math. Ann. {\bf 364} (2016), 539-558,
[\href{https://arxiv.org/abs/1405.6616}{\tt arXiv:1405.6616}].

\vspace{-3mm}
\bibitem[Be91]{Be}
D. J. Benson,
{\it Representations and Cohomology Volume 1: Basic Representation Theory of Finite Groups
and Associative Algebras},
 Cambridge University Press, 1991.

\vspace{-3mm}
\bibitem[BCR00]{BilloCrapsRoose00}
M. Bill{\'o}, B. Craps, F. Roose,
{\it Orbifold boundary states from Cardy's condition},
JHEP {\bf 0101} (2001), 038,
[\href{https://arxiv.org/abs/hep-th/0011060}{\tt arXiv:hep-th/0011060}].


\vspace{-3mm}
\bibitem[BDH${}^+$02]{BDHKMMS02}
J. de Boer, R. Dijkgraaf, K. Hori, A. Keurentjes, J. Morgan, D. Morrison and S. Sethi,
{\it Triples, Fluxes, and Strings},
Adv. Theor. Math. Phys. {\bf 4} (2002) 995-1186,
[\href{https://arxiv.org/abs/hep-th/0103170}{\tt arXiv:hep-th/0103170}].

\vspace{-3mm}
\bibitem[Bl17]{Blu17}
A. Blumberg,
{\it Equivariant homotopy theory},
lecture notes 2017,
\newline
[\href{https://github.com/adebray/equivariant\_homotopy\_theory}{\tt github.com/adebray/equivariant\_homotopy\_theory}].


\vspace{-3mm}
\bibitem[Bo94]{Bor}
F. Borceux, {\it Handbook of Categorical Algebra},
Cambridge University Press, Cambridge, 1994.

\vspace{-3mm}
\bibitem[Bo36]{Borsuk36}
K. Borsuk,
{\it Sur les groupes des classes de transformations continues},
CR Acad. Sci. Paris {\bf 202} (1936), 1400-1403,
[\href{https://doi.org/10.24033/asens.603}{\tt doi:10.24033/asens.603}].


\vspace{-3mm}
\bibitem[Bo10]{Bo}
S. Bouc,
{\it Biset Functors for Finite Groups},
Springer-Verlag, Berlin, 2010.

\vspace{-3mm}
\bibitem[BSS18]{BSS18}
V. Braunack-Mayer, H. Sati, U. Schreiber,
{\it Gauge enhancement for Super M-branes via Parameterized stable homotopy theory},
Comm. Math. Phys. (2019)
[\href{https://arxiv.org/abs/1805.05987}{\tt arXiv:1805.05987}][hep-th].


\vspace{-3mm}
\bibitem[Ca84]{Carlsson84}
G. Carlsson,
{\it Equivariant Stable Homotopy and Segal's Burnside Ring Conjecture},
Ann. Math. {\bf  120} (1984), 189-224,
[\href{https://www.jstor.org/stable/2006940}{\tt jstor:2006940}].


\vspace{-3mm}
\bibitem[CP18]{ChesterPerlmutter18}
S. M. Chester and E. Perlmutter,
{\it M-Theory Reconstruction from $(2,0)$ CFT and the Chiral Algebra Conjecture},
J. High Energy Phys. 2018 (2018) 116,
[\href{https://arxiv.org/abs/1805.00892}{\tt arXiv:1805.00892}].

\vspace{-3mm}
\bibitem[CCM09]{CCM}
 A. Connes, C. Consani, M. Marcolli, {\it  Fun with $\mathbb{F}_1$},
 J. Number Th. {\bf 129} (2009),  1532-1561.


\vspace{-3mm}
\bibitem[DGM97]{DouglasGreeneMorrison97}
M. R. Douglas, B. R. Greene, D. R. Morrison,
{\it Orbifold Resolution by D-Branes},
Nucl. Phys. {\bf B506} (1997), 84-106,
[\href{https://arxiv.org/abs/hep-th/9704151}{\tt hep-th/9704151}].


\vspace{-3mm}
\bibitem[Dr71]{Dr71}
A. Dress, {\it Notes on the theory of representations of finite groups},
	Bielefeld, 1971.
	
\vspace{-3mm}
\bibitem[Dr86]{Dress86}
A. Dress,
{\it Congruence relations characterizing the representation ring of the symmetric group}, J. of Algebra {\bf 101} (1986),  350-364,
[\href{https://ncatlab.org/nlab/files/Dress86.pdf}{\tt ncatlab.org/nlab/files/Dress86.pdf}]

\vspace{-3mm}
\bibitem[Du99]{Duff99B}
M. Duff (ed.)
{\it The World in Eleven Dimensions: Supergravity, Supermembranes and M-theory},
Institute of Physics Publishing, Bristol, 1999.

\vspace{-3mm}
\bibitem[FSS15]{FSS15}
D. Fiorenza, H. Sati and U. Schreiber,
{\it The WZW term of the M5-brane and differential cohomotopy},
J. Math. Phys. {\bf 56} (2015), 102301,
[\href{https://arxiv.org/abs/1506.07557}{\tt arXiv:1506.07557}].


\vspace{-3mm}
\bibitem[FSS16a]{FSS16a}
D. Fiorenza, H. Sati and U. Schreiber,
{\it Rational sphere valued supercocycles in M-theory and type IIA string theory},
J. Geom. Phys. {\bf 114} (2017) 91-108,
[\href{https://arxiv.org/abs/1606.03206}{\tt arXiv:1606.03206}].


\vspace{-3mm}
\bibitem[FSS19a]{FSS19a}
D. Fiorenza, H. Sati and U. Schreiber,
\href{https://ncatlab.org/schreiber/show/The+rational+higher+structure+of+M-theory}
{\it The rational higher structure of M-theory},
Fortsch. Phys. 2019, DOI:	10.1002/prop.201910017
[\href{https://arxiv.org/abs/1903.02834}{\tt	arXiv:1903.02834}] [hep-th].



\vspace{-3mm}
\bibitem[FSS19b]{FSS19b}
D. Fiorenza, H. Sati and U. Schreiber,
{\it Twisted Cohomotopy implies M-theory anomaly cancellation},
[\href{https://arxiv.org/abs/1904.10207}{\tt arXiv:1904.10207}].

\vspace{-3mm}
\bibitem[FSS19c]{FSS19c}
D. Fiorenza, H. Sati and U. Schreiber,
{\it Twisted Cohomotopy implies M5 WZ term level quantization},
[\href{https://arxiv.org/abs/1906.07417}{\tt arXiv:1906.07417}].


\vspace{-3mm}
\bibitem[GC99]{GarciaCompean99}
H. Garcia-Compean,
{\it D-branes in Orbifold Singularities and Equivariant K-Theory},
Nucl. Phys. {\bf B557} (1999), 480-504,
[\href{https://arxiv.org/abs/hep-th/9812226}{\tt arXiv:hep-th/9812226}].

\vspace{-3mm}
\bibitem[GP90]{GilbertPathria90}
W. Gilbert and A. Pathria,
{\it Linear Diophantine Equations},
preprint 1990,  \newline
[\href{https://ncatlab.org/nlab/files/GilbertPathria90.pdf}{\tt ncatlab.org/nlab/files/GilbertPathria90.pdf}]

\vspace{-3mm}
\bibitem[GS19]{GS19}
D. Grady and H. Sati,
{\it Ramond-Ramond fields and twisted differential K-theory},
[\href{https://arxiv.org/abs/1903.08843}{\tt arXiv:1903.08843}] [hep-th].


\vspace{-3mm}
\bibitem[Gr05]{Greenlees05}
J. Greenlees,
{\it Equivariant version of real and complex connective K-theory},
Homology Homotopy Appl. {\bf 7} (2005), 63-82,
[\href{http://projecteuclid.org/euclid.hha/1139839291}{\tt euclid.hha/1139839291}].

\vspace{-3mm}
\bibitem[HSS18]{HSS18}
J. Huerta, H. Sati, U. Schreiber,
{\it Real ADE-equivariant (co)homotpy of super M-branes},
Commun. Math. Phys. (2019),
[\href{https://arxiv.org/abs/1805.05987}{\tt arXiv:1805.05987}].


\vspace{-3mm}
\bibitem[KS]{KS}
M. Kapranov and  A. Smirnov,
{\it  Cohomology determinants and reciprocity laws: number field case},
 unpublished preprint, \href{http://cage.ugent.be/\~kthas/Fun/library/KapranovSmirnov.pdf}
{\tt cage.ugent.be/\~kthas/Fun/library/KapranovSmirnov.pdf}


\vspace{-3mm}
\bibitem[Kee03]{Keenan03}
A. Keenan,
{\it Which finite groups act freely on spheres?}, 2003, \newline
\href{http://www.math.utah.edu/\~keenan/actions.pdf}{\tt www.math.utah.edu/$\sim$keenan/actions.pdf}

\vspace{-3mm}
\bibitem[Ke99]{Ker}
A. Kerber, {\it Applied Finite Group Actions}, Springer-Verlag, Berlin, 1999.

\vspace{-3mm}
\bibitem[Kl84]{Klein1884}
F. Klein, {\it Vorlesungen uber das Ikosaeder und die Aufl{\"o}sung der Gleichungen vom funften Grade}, 1884,
translated as {\it Lectures on the Icosahedron and the Resolution of Equations of Degree Five} by George Morrice 1888,
[\href{https://archive.org/details/cu31924059413439}{\tt archive.org/details/cu31924059413439}].


\vspace{-3mm}
\bibitem[Li18]{Lindh2018}
M. Lindh,
{\it An Introduction to the McKay Correspondence: Master Thesis in Physics}, 2018, \newline
[\href{http://www.diva-portal.org/smash/get/diva2:1184051/FULLTEXT01.pdf}{\tt www.diva-portal.org/smash/get/diva2:1184051/FULLTEXT01.pdf}]

\vspace{-3mm}
\bibitem[Lo18]{Lo}
O. Lorscheid, {\it $\mathbb{F}_1$ for everyone},
[\href{https://arxiv.org/abs/1801.05337}{\tt arXiv:1801.05337}] [math.AG].

\vspace{-3mm}
\bibitem[Lu05]{Lueck}
W. L\"uck,
{\it The Burnside ring and equivariant stable cohomotopy for infinite groups},
Pure Appl. Math. Q.  {\bf 1} (2005),
479--541,
[\href{https://arxiv.org/abs/math/0504051}{\tt	arXiv:math/0504051}] [math.AT].


\vspace{-3mm}
\bibitem[LP12]{LP}
K. Lux and H. Pahlings,
{\it Representations of Groups: A Computational Approach},
Cambridge University Press, 2012.


\vspace{-3mm}
\bibitem[Mc65]{Mc}
S. MacLane,
{\it Categorical algebra},
Bull. Amer. Math. Soc. {\bf 71} (1965), 40-106.


\vspace{-3mm}
\bibitem[Mad00]{Mader00}
A. Mader,
{\it Almost completely decomposable groups},
CRC Press, 2000.

\vspace{-3mm}
\bibitem[Man08]{Ma}
Y. Manin, {\it Cyclotomy and analytic geometry over $\mathbb{F}_1$},
Quanta of Maths, Conference in honour of Alain Connes, Clay Math.
Proceedings {\bf 11} (2008), 385-408,
[\href{https://arxiv.org/abs/0809.1564}{\tt arXiv:0809.1564}] [math.AG].


\vspace{-3mm}
\bibitem[Mar30]{Marczewski30}
E. Marczewski,
{\it Sur l'extension de l'ordre partiel},
Fund. Math. {\bf 16} (1930), 386-389,
\newline
[\href{http://matwbn.icm.edu.pl/ksiazki/fm/fm16/fm16125.pdf}{\tt matwbn.icm.edu.pl/ksiazki/fm/fm16/fm16125.pdf}]

\vspace{-3mm}
\bibitem[MBD1916]{MBD1916}
G. A. Miller, H. F. Blichfeldt, L. E. Dickson,
{\it Theory and applications of finite groups},
Dover, New York, 1916.

\vspace{-3mm}
\bibitem[Mo14]{Moore14}
G. Moore,
{\it Physical Mathematics and the Future},
talk at \href{http://physics.princeton.edu/strings2014/}{Strings 2014}
\newline
\href{http://www.physics.rutgers.edu/~gmoore/PhysicalMathematicsAndFuture.pdf}{[\tt www.physics.rutgers.edu/$\sim$gmoore/PhysicalMathematicsAndFuture.pdf}].


\vspace{-3mm}
\bibitem[Na-cycl]{Naik}
V. Naik,
{\it Characters are cyclotomic integers}, \newline
[\href{https://groupprops.subwiki.org/wiki/Characters_are_cyclotomic_integers}{\tt groupprops.subwiki.org/wiki/Characters\_are\_cyclotomic\_integers}]


\vspace{-3mm}
\bibitem[NH98]{NicolaiHelling98}
H. Nicolai and R. Helling,
{\it Supermembranes and M(atrix) Theory},
In:
M. Duff, E. Sezgin, B. Greene et. al. (eds.)
{\it Nonperturbative aspects of strings, branes and supersymmetry},
World Scientific (1999),
[\href{https://arxiv.org/abs/hep-th/9809103}{\tt arXiv:hep-th/9809103}].


\vspace{-3mm}
\bibitem[Ra02]{Rajan02}
P. Rajan,
{\it D2-brane RR-charge on $\mathrm{SU}(2)$},
Phys. Lett. {\bf B533} (2002), 307-312, \newline
[\href{https://arxiv.org/abs/hep-th/0111245}{\tt arXiv:hep-th/0111245}].

\vspace{-3mm}
\bibitem[Pe56]{Peterseon56}
F. P. Peterson,
{\it Some Results on Cohomotopy Groups},
Amer.  J. Math. {\bf 78} (1956), 243-258, \newline
[\href{https://www.jstor.org/stable/2372514}{\tt jstor:2372514}].

\vspace{-3mm}
\bibitem[Pf97]{Pfeiffer97}
G. Pfeiffer,
{\it The Subgroups of $M_{24}$, or How to Compute the Table of Marks of a Finite Group}, Experiment. Math. {\bf 6} (1997), 247-270,
[\href{https://doi.org/10.1080/10586458.1997.10504613}{doi:10.1080/10586458.1997.10504613}]
\newline [\href{http://schmidt.ucg.ie/\~goetz/pub/marks/marks.html}{\tt schmidt.ucg.ie/$\sim$goetz/pub/marks/marks.html}]


\vspace{-3mm}
\bibitem[PT91]{PursellTrimble91}
L. Pursell, S. Y. Trimble,
{\it Gram-Schmidt orthogonalization by Gauss Elimination},
Amer. Math. Month. {\bf 98}(1991), 544-549,
[\href{https://www.jstor.org/stable/2324877}{\tt jstor:2324877}].


\vspace{-3mm}
\bibitem[RS13]{RecknagelSchomerus13}
A. Recknagel, V. Schomerus,
{\it Boundary Conformal Field Theory and the Worldsheet Approach to D-Branes},
Cambridge University Press, 2013.


\vspace{-3mm}
\bibitem[Re05]{Rees05}
E. Rees,
{\it Notes on Geometry},
Springer, Berlin, 2005.

\vspace{-3mm}
\bibitem[Rob06]{Robbin06}
J. Robbin,
{\it Real, Complex and Quaternionic representations}, 2006,
\newline
[\href{http://www.math.wisc.edu/\~robbin/angelic/RCH-G.pdf}{\tt www.math.wisc.edu/$\sim$robbin/angelic/RCH-G.pdf}]

\vspace{-3mm}
\bibitem[Sa18]{Sati13}
H. Sati,
{\it Framed M-branes, corners, and topological invariants},
J. Math. Phys. {\bf 59} (2018), 062304, \newline
[\href{https://arxiv.org/abs/ arXiv:1310.1060}{\tt arXiv:1310.1060}].
	

\vspace{-3mm}
\bibitem[SS19]{SS19}
H. Sati and U. Schreiber,
{\it Equivariant Cohomotopy implies orientifold tadpole cancellation},
\newline
[\href{https://arxiv.org/abs/1909.12277}{\tt arXiv:1909.12277}].



\vspace{-3mm}
\bibitem[Seg71]{Segal70}
G. Segal,
{\it Equivariant stable homotopy theory},
In {\it Actes du Congr{\`e}s International des Math\'ematiciens} (Nice, 1970), Tome 2 , pages 59-63, Gauthier-Villars, Paris, 1971,
\newline
[\href{https://ncatlab.org/nlab/files/SegalEquivariantStableHomotopyTheory.pdf}{\tt ncatlab.org/nlab/files/SegalEquivariantStableHomotopyTheory.pdf}].


\vspace{-3mm}
\bibitem[Seg72]{Segal72}
G. Segal, {\it Permutation representations of finite $p$-groups},
Quart. J. Math. Oxford (2) {\bf 23} (1972), 375-381,
[\href{https://doi.org/10.1093/qmath/23.4.375}{\tt 10.1093/qmath/23.4.375}].

\vspace{-3mm}
\bibitem[Sen97]{Sen97}
A. Sen,
{\it A Note on Enhanced Gauge Symmetries in M- and String Theory},
JHEP {\bf 9709} (1997), 001,
[\href{https://arxiv.org/abs/hep-th/9707123}{\tt arXiv:hep-th/9707123}].


\vspace{-3mm}
\bibitem[Ser14]{Serrano14}
J. Serrano,
{\it Finite subgroups of $SL(2,\mathbb{C})$ and $SL(3,\mathbb{C})$},
Warwick 2014,
\newline
[\href{https://homepages.warwick.ac.uk/\~masda/McKay/Carrasco\_Project.pdf}{\tt homepages.warwick.ac.uk/$\sim$masda/McKay/Carrasco\_Project.pdf}]


\vspace{-3mm}
\bibitem[Ser77]{Serre77}
J.-P. Serre,
{\it Linear Representations of Finite Groups},
Graduate Texts in Math., vol. 42, Springer-Verlag, New York, 1977.

\vspace{-3mm}
\bibitem[So67]{Sol}
L. Solomon, {\it The Burnside algebra of a finite group}, J. Comb. Theory {\bf 1} (1967), 603-615.


\vspace{-3mm}
\bibitem[So04]{So}
 C. Soul\'e, {\it  Les vari\'et\'es sur le corps \'a un \'el\'ement},
 Mosc. Math. J. {\bf 4} (2004), 217-244.

\vspace{-3mm}
\bibitem[Sp49]{Spanier49}
E. Spanier,
{\it Borsuk's Cohomotopy Groups},
Ann. Math. {\bf 50} (1949), 203-245,
[\href{http://www.jstor.org/stable/1969362}{\tt jstor:1969362}].


\vspace{-3mm}
\bibitem[Ta00]{Taylor00}
W. Taylor, {\it D2-branes in B fields},
JHEP {\bf 0007} (2000) 039,
[\href{https://arxiv.org/abs/hep-th/0004141}{\tt arXiv:hep-th/0004141}].


\vspace{-3mm}
\bibitem[Th16]{Th}
K. Thas, {\it Absolute Arithmetic and $\mathbb{F}_1$-geometry},
European Mathematical Society, 2016.


\vspace{-3mm}
\bibitem[Ti56]{Ti}
J. Tits, {\it  Sur les analogues alg\'ebriques des groupes semi-simples complexes},
 Colloque d'alg\'ebre sup\'erieure,
tenu \'a Bruxelles du 19 au 22 d\'ecembre 1956 (1957), 261-289.

\vspace{-3mm}
\bibitem[tD79]{tomDieck79}
T. tom Dieck,
{\it Transformation Groups and Representation Theory}
Lecture Notes in Mathematics vol. {\bf 766}, Springer, 1979,
[\href{https://link.springer.com/book/10.1007/BFb0085965}{\tt doi:10.1007/BFb0085965}].


\vspace{-3mm}
\bibitem[tDi09]{tomDieck09}
T. tom Dieck, {\it Representation theory}, lecture notes 2009,
[\href{http://www.uni-math.gwdg.de/tammo/rep.pdf}{\tt www.uni-math.gwdg.de/tammo/rep.pdf}]

\vspace{-3mm}
\bibitem[Tr09]{Trimble09}
T. Trimble,
{\it Categorified Gram-Schmidt process},
$n$Lab note, November 2009
\newline
\href{https://ncatlab.org/nlab/revision/Gram-Schmidt+process/7\#categorified\_gramschmidt\_process}{\tt ncatlab.org/nlab/revision/Gram-Schmidt+process/7\#categorified\_gramschmidt\_process}
\newline
latest revision April 2018
\newline
\href{https://ncatlab.org/nlab/revision/Gram-Schmidt+process/14\#categorified\_gramschmidt\_process}{\tt ncatlab.org/nlab/revision/Gram-Schmidt+process/14\#categorified\_gramschmidt\_process}

\vspace{-3mm}
\bibitem[Va05]{Vafa05}
C. Vafa,
{\it The String Landscape and the Swampland},
[\href{http://arxiv.org/abs/hepth/0509212}{\tt arXiv:hepth/0509212}]

\vspace{-3mm}
\bibitem[Wi98]{Witten98}
E. Witten,
{\it D-Branes And K-Theory},
JHEP {\bf 9812} (1998), 019,
[\href{https://arxiv.org/abs/hep-th/9810188}{\tt arXiv:hep-th/9810188}].

\vspace{-3mm}
\bibitem[Zh01]{Zhou01}
J.-G. Zhou,
{\it D-branes in B Fields},
Nucl. Phys. {\bf B607} (2001) 237-246, [\href{https://arxiv.org/abs/hep-th/0102178}{\tt hep-th/0102178}].



\noindent \hspace{-1cm}
{\footnotesize \bf We list the following items for ease of reference:}



\vspace{-3mm}
\bibitem[Dok-GroupNames]{DokGroupNames}
T. Dokchitser,
{\it GroupNames},
\newline
[\href{https://people.maths.bris.ac.uk/\~matyd/GroupNames/1/Q8.html}{\tt people.maths.bris.ac.uk/$\sim$matyd/GroupNames/}


\vspace{-3mm}
\bibitem[Dok-$C_2$]{DokC2}
T. Dokchitser,
{\it GroupNames: $C_2$},
\newline
[\href{people.maths.bris.ac.uk/\~matyd/GroupNames/1/C2.html}{\tt people.maths.bris.ac.uk/~matyd/GroupNames/1/C2.html}]


\vspace{-3mm}
\bibitem[Dok-$C_3$]{DokC3}
T. Dokchitser,
{\it GroupNames: $C_3$},
\newline
[\href{people.maths.bris.ac.uk/\~matyd/GroupNames/1/C3.html}{\tt people.maths.bris.ac.uk/~matyd/GroupNames/1/C3.html}]


\vspace{-3mm}
\bibitem[Dok-$C_4$]{DokC4}
T. Dokchitser,
{\it GroupNames: $C_4$},
\newline
[\href{people.maths.bris.ac.uk/\~matyd/GroupNames/1/C4.html}{\tt people.maths.bris.ac.uk/~matyd/GroupNames/1/C4.html}]


\vspace{-3mm}
\bibitem[Dok-$2 D_{2n}$]{DokDic}
T. Dokchitser,
{\it GroupNames: Dicyclic groups $\mathrm{Dic}_n$},
\newline
\href{https://people.maths.bris.ac.uk/\~matyd/GroupNames/dicyclic.html}{\tt people.maths.bris.ac.uk/~matyd/GroupNames/dicyclic.html}


\vspace{-3mm}
\bibitem[Dok-$2 D_4$]{DokQ8}
T. Dokchitser,
{\it GroupNames: $Q_8$},
\newline
[\href{https://people.maths.bris.ac.uk/\~matyd/GroupNames/1/Q8.html}{\tt people.maths.bris.ac.uk/$\sim$matyd/GroupNames/1/Q8.html}


\vspace{-3mm}
\bibitem[Dok-$2D_6$]{Dok2D6}
T. Dokchitser,
{\it GroupNames: $\mathrm{Dic}_3$},
\newline
[\href{https://people.maths.bris.ac.uk/\~matyd/GroupNames/1/Dic3.html}{\tt people.maths.bris.ac.uk/$\sim$matyd/GroupNames/1/Dic3.html}


\vspace{-3mm}
\bibitem[Dok-$2 D_8$]{DokQ16}
T. Dokchitser,
{\it GroupNames: $Q_{16}$},
\newline
[\href{https://people.maths.bris.ac.uk/\~matyd/GroupNames/1/Q16.html}{\tt people.maths.bris.ac.uk/$\sim$matyd/GroupNames/1/Q16.html}

\vspace{-3mm}
\bibitem[Dok-$2 D_{10}$]{Dok2D10}
T. Dokchitser,
{\it GroupNames: $\mathrm{Dic}_5$},
\newline
[\href{https://people.maths.bris.ac.uk/~matyd/GroupNames/1/Dic5.html}{\tt people.maths.bris.ac.uk/$\sim$matyd/GroupNames/1/Dic5.html}

\vspace{-3mm}
\bibitem[Dok-$2 D_{12}$]{Dok2D12}
T. Dokchitser,
{\it GroupNames: $\mathrm{Dic}_6$},
\newline
[\href{https://people.maths.bris.ac.uk/~matyd/GroupNames/1/Dic6.html}{\tt people.maths.bris.ac.uk/$\sim$matyd/GroupNames/1/Dic6.html}


\vspace{-3mm}
\bibitem[Dok-$2T$]{Dok2T}
T. Dokchitser,
{\it GroupNames: $\mathrm{SL}(2,\mathbb{F}_3)$},
\newline
[\href{https://people.maths.bris.ac.uk/\~matyd/GroupNames/1/SL(2,3).html}{\tt people.maths.bris.ac.uk/$\sim$matyd/GroupNames/1/SL(2,3).html}

\vspace{-3mm}
\bibitem[Dok-$2O$]{Dok2O}
T. Dokchitser,
{\it GroupNames: $\mathrm{CSU}(2,\mathbb{F}_3)$},
\newline
[\href{https://people.maths.bris.ac.uk/\~matyd/GroupNames/1/SL(2,3).html}{\tt people.maths.bris.ac.uk/$\sim$matyd/GroupNames/1/CSU(2,3).html}

\vspace{-3mm}
\bibitem[Dok-$2I$]{Dok2I}
T. Dokchitser,
{\it GroupNames: $\mathrm{SL}(2, \mathbb{F}_5)$},
\newline
[\href{https://people.maths.bris.ac.uk/\~matyd/GroupNames/97/SL(2,5).html}{\tt people.maths.bris.ac.uk/$\sim$matyd/GroupNames/97/SL(2,5).html}

\vspace{-3mm}
\bibitem[Dok-$\mathrm{GL}(2,3)$]{DokGL23}
T. Dokchitser,
{\it GroupNames: $\mathrm{GL}(2,3)$},
\newline
[\href{https://people.maths.bris.ac.uk/\~matyd/GroupNames/1/GL(2,3).html}{\tt people.maths.bris.ac.uk/$\sim$matyd/GroupNames/1/GL(2,3).html}

\vspace{-3mm}
\bibitem[Dok-$C_3 \times Q_8$]{DokC3TimesQ8}
T. Dokchitser,
{\it GroupNames: $C_3 \times Q_8$},
\newline
[\href{https://people.maths.bris.ac.uk/\~matyd/GroupNames/1/C3xQ8.html}{\tt people.maths.bris.ac.uk/$\sim$matyd/GroupNames/1/C3xQ8.html}

\medskip

\vspace{-3mm}
\bibitem[Mon-Reps]{MonRepresentations}
J. Montaldi,
{\it Representations}
\newline[\href{http://www.maths.manchester.ac.uk/\~jm/wiki/Representations/BinaryCubic}{\tt www.maths.manchester.ac.uk/$\sim$jm/wiki/Representations/BinaryCubic}]


\vspace{-3mm}
\bibitem[Mon-$Q_8$]{MonQ8}
J. Montaldi,
{\it Representations: $Q_8$}
\newline[\href{http://www.maths.manchester.ac.uk/\~jm/wiki/Representations/BinaryCubic\#Q8}{\tt www.maths.manchester.ac.uk/$\sim$jm/wiki/Representations/BinaryCubic\#Q8}]

\vspace{-3mm}
\bibitem[Mon-$2T$]{Mon2T}
J. Montaldi,
{\it Representations: $2T$}
\newline[\href{http://www.maths.manchester.ac.uk/\~jm/wiki/Representations/BinaryCubic\#BinTet}{\tt www.maths.manchester.ac.uk/$\sim$jm/wiki/Representations/BinaryCubic\#BinTet}]


\vspace{-3mm}
\bibitem[Mon-$2O$]{Mon2O}
J. Montaldi,
{\it Representations: $2O$}
\newline[\href{http://www.maths.manchester.ac.uk/\~jm/wiki/Representations/BinaryCubic\#BinOct}{\tt www.maths.manchester.ac.uk/$\sim$jm/wiki/Representations/BinaryCubic\#BinOct}]



\end{thebibliography}
\end{document}